\documentclass[a4paper, 12pt]{amsart}

\usepackage{amsmath}
\usepackage{amssymb}
\usepackage{ascmac}
\usepackage{amsthm}
\usepackage{color}

\makeatletter

\@addtoreset{equation}{section}
\makeatother 

\theoremstyle{plain}
\newtheorem{thm}{Theorem}[section]
\newtheorem*{thm*}{Theorem}
\newtheorem{prop}[thm]{Proposition}
\newtheorem{lem}[thm]{Lemma}

\theoremstyle{definition}

\theoremstyle{remark}
\newtheorem{rem}[thm]{Remark}

\newcommand{\Var}{\operatorname{Var}}
\newcommand{\Ric}{\operatorname{Ric}}

\newcommand{\Div}{\operatorname{div}}

\newcommand{\tr}{\operatorname{tr}}

\newcommand{\Cn}{H_n}

\newcommand{\IR}{\mathbb{R}}

\newcommand{\NN}{\mathcal{N}}

\renewcommand{\SS}{\mathcal{S}}

\newcommand{\WW}{\mathcal{W}}
\newcommand{\HH}{\mathcal{H}}
\newcommand{\eps}{\varepsilon}
\newcommand{\la}{\lambda}

\newcommand{\ov}[1]{\overline{#1}}

\usepackage{latexsym}

\makeatletter
\@namedef{subjclassname@2020}{%
  \textup{2020} Mathematics Subject Classification}
\makeatother

\title[Almost splitting and quantitative stratification]{Almost splitting and quantitative stratification\\ for super Ricci flow}

\author{Keita Kunikawa}
\address{Faculty of Science and Technology, Tokushima University, 2-1 Minamijosanjima-cho, Tokushima, 770-8506, Japan}
\email{kunikawa@tokushima-u.ac.jp}

\author{Yohei Sakurai}
\address{Department of Mathematics, Saitama University, 255 Shimo-Okubo, Sakura-ku, Saitama-City, Saitama, 338-8570, Japan}
\email{ysakurai@rimath.saitama-u.ac.jp}

\subjclass[2020]{Primary 53E20; Secondary 58J35}
\keywords{Super Ricci flow; M\"uller quantity; Almost splitting; Quantitative stratification}
\date{September 4, 2025}

\begin{document}
\maketitle

\begin{abstract}
The aim of this paper is to study almost rigidity properties of super Ricci flow whose M\"uller quantity is non-negative.
We conclude almost splitting and quantitative stratification theorems that have been established by Bamler \cite{B3} for Ricci flow.
As a byproduct,
we obtain an almost constancy for a certain integral quantity concerning scalar curvature at an almost selfsimilar point,
which is new even for Ricci flow.
\end{abstract}

\section{Introduction}

\subsection{Background}
Let $(M,(g_t)_{t\in I})$ be a compact manifold equipped with a time-dependent Riemannian metric.
Such a time-dependent manifold is called \textit{Ricci flow} when
\begin{equation*}
\partial_{t}g= -2 \Ric,
\end{equation*}
which has been introduced by Hamilton \cite{H}.
It is well-known that
Perelman \cite{P} has vastly developed the Ricci flow theory in three dimensional case,
and utilized it for the resolution of the Poincar\'e and geometrization conjectures.
A supersolution to the Ricci flow is called \textit{super Ricci flow};
namely,
$(M,(g_t)_{t\in I})$ is called super Ricci flow when
\begin{equation*}
\partial_{t}g\geq -2 \Ric,
\end{equation*}
which has been introduced by McCann-Topping \cite{MT} in view of the relation between Ricci flow and mass transport.
Super Ricci flow can be viewed as an interpolation between Ricci flow and (static) manifolds of non-negative Ricci curvature.
Actually,
it has been examined in the literature of the geometry of metric measure spaces with a lower Ricci curvature bound developed by Sturm \cite{S1}, \cite{S2}, Lott-Villani \cite{LV}, Ambrosio-Gigli-Savar\'{e} \cite{AGS}, Gigli \cite{G} and so on (see the work of Sturm \cite{S}, and also subsequent works \cite{K}, \cite{KS1}, \cite{KS2}).  

One of the central problems in Ricci flow theory is to classify singularity models at the first singular time,
and it has been already resolved in low dimensional cases.
In two dimensional case,
the classification can be derived from classical works by Chow \cite{C} and Hamilton \cite{H2}.
In three dimensional case,
it has been completed by Perelman \cite{P} and Brendle \cite{B}.
To deal with high dimensional singularity models,
which may be non-smooth (see \cite{A}),
Bamler \cite{B1}, \cite{B2}, \cite{B3} has established a convergence theory for super Ricci flow within the framework of \textit{metric flows} and their \textit{$\mathbb{F}$-convergence}.
Based on fundamental results in \cite{B1},
he has proved a precompactness theorem for super Ricci flow in \cite{B2}.
Furthermore,
he has concluded regularity properties of non-collapsed $\mathbb{F}$-limits of Ricci flows in \cite{B3},
where the non-collapsed condition has been described in terms of the Nash entropy.
The key ingredient was to study almost rigidity phenomena concerning the monotonicity of the Nash entropy,
which is similar to the structure theory of Gromov-Hausdorff limits of manifolds with a lower Ricci curvature bound due to Cheeger-Colding \cite{CC}, \cite{CC1}, \cite{CC2}, \cite{CC3}, Colding-Naber \cite{CN} and so on.

\subsection{Aim and setting}

We are now concerned with regularity properties of $\mathbb{F}$-limits of super Ricci flows (not necessarily Ricci flow).
In this article,
we aim to generalize almost splitting and quantitative stratification theorems obtained in \cite[Part 2]{B3} for super Ricci flow whose M\"uller quantity is non-negative.
We recall that
for a vector field $V$,
the \textit{M\"uller quantity} is defined as follows (see \cite[Definition 1.3]{M1}):
\begin{equation*}
\mathcal{D}(V):=\partial_{t} \HH-\Delta \HH-2| h |^2+4\Div h(V)-2 \langle \nabla \HH,V \rangle+2\Ric(V,V)-2h(V,V),
\end{equation*}
where
\begin{equation*}
h:=-\frac{1}{2} \partial_t g,\quad \HH:=\tr h.
\end{equation*}
Notice that
this quantity vanishes for Ricci flow by the evolution formula for scalar curvature and the contracted second Bianchi identity.
The non-negativity of the M\"uller quantity brings various benefits such as a lower bound of the (generalized) scalar curvature, Harnack estimate and the monotonicity of entropies (see Lemmas \ref{lem:scal}, \ref{lem:Nash} and Theorem \ref{thm:Harnack} below).
There are several examples of super Ricci flow whose M\"uller quantity is non-negative (see \cite[Section 2]{M1}, \cite[Section 7]{FZ}):
\begin{enumerate}\setlength{\itemsep}{+1.0mm}
\item Static manifolds of non-negative Ricci curvature;
\item Ricci flow;
\item Ricci flow coupled with heat equation, called \textit{List flow} (\cite{L});
\item Ricci flow coupled with harmonic map heat flow, called \textit{M\"uller flow} (\cite{M2});
\item mean curvature flow for spacelike hypersurfaces in Lorentzian manifold of non-negative sectional curvature;
\item (scaled) twisted K\"{a}hler-Ricci flow.
\end{enumerate}

In \cite{KuS},
the authors have extended some results by Bamler-Zhang \cite{BZ} concerning geometric analysis for Ricci flow under a scalar curvature bound to this object.

\subsection{Results and organization}

In Sections \ref{sec:Preliminaries}, \ref{sec:SRF}, \ref{sec:Nash entropy}, \ref{sec:vol_heat}, \ref{sec:noncollapse},
we produce several estimates for the proof of our main results.
In Section \ref{sec:Preliminaries},
we recall basic facts on the heat kernel.
In Section \ref{sec:SRF},
we collect known results for super Ricci flow.
In Section \ref{sec:Nash entropy},
we examine monotonicity properties for the Nash entropy.
In Section \ref{sec:vol_heat},
we deduce volume and heat kernel estimates.
In Section \ref{sec:noncollapse},
we show some estimates under a non-collapsed condition for the Nash entropy. 

In Section \ref{sec:almost_ss},
we will study almost selfsimilar points.
Recall that
the essence of the regularity theory in \cite{B3} was to analyze the structure of \textit{tangent flows} at each point in an $\mathbb{F}$-limit,
which are blowup limits for parabolic rescaling with respect to the $\mathbb{F}$-convergence.
It has been proved that
every tangent flow has constant Nash entropy,
and it must be a \textit{metric soliton},
which is a synthetic notion of shrinking Ricci soliton.
From this point of view,
on smooth Ricci flow,
the notion of almost selfsimilar point has been introduced such that it characterizes the almost constancy of Nash entropy.
The investigation of such points has led to various consequences for tangent flows.

We formulate almost selfsimilar points in our general setting by observing examples and adding a term for the M\"uller quantity.
As one of our results,
we obtain an almost constancy for a certain integral quantity concerning scalar curvature (Proposition \ref{prop:almost soliton scal}).
We emphasize that this result is new even for Ricci flow.
In \cite{B3},
an almost monotonicity for that quantity has been derived along Ricci flow (see \cite[Proposition 7.9]{B3}).
The refinement of calculations in \cite{B3} enables us to point out that
it may be extended to our general setting,
and the reverse direction also holds.

Sections \ref{sec:almost_st}, \ref{sec:almost_split}, \ref{sec:stratification} are devoted to the proof of our main theorems.
In Section \ref{sec:almost_st},
we show an almost static cone splitting theorem (Theorem \ref{thm:almost static}). 
In Section \ref{sec:almost_split},
we prove an almost splitting theorem (Theorem \ref{thm:main almost splitting}). 
In Section \ref{sec:stratification},
we conclude a quantitative stratification theorem (Theorem \ref{thm:stratification}).

\section{Preliminaries}\label{sec:Preliminaries}
In this section,
we recall some basic notions.
In what follows,
let $(M,(g_t)_{t\in I})$ be an $n$-dimensional compact manifold with a time-dependent metric.

\subsection{Convention and notation}\label{sec:Basics}
For positive constants $a,b,c,d,\dots$,
we denote by $C_{a,b,c,d,\dots}$ positive constants depending only on $a,b,c,d,\dots$ and the dimension $n$ without distinction,
where we stress that
the dependence on the dimension $n$ are omitted.
When we specify constants,
we put numbers such as $C_{1,a,b,c,d,\dots},C_{2,a,b,c,d,\dots}$. 
When we assert that ``there exists $\ov C_{a,b,c,d,\dots}$ such that if $C \leq \ov C_{a,b,c,d,\dots}$",
we abbreviate ``if $C \leq \ov C_{a,b,c,d,\dots}$" for short.
In the same manner,
for a lower bound,
we say ``if $C \geq \underline{C}_{a,b,c,d,\dots}$" instead of ``if $C \leq \ov C_{a,b,c,d,\dots}$".

Let $d_t$ and $m_t$ be the Riemannian distance and Riemannian volume measure with respect to $g_t$,
respectively.
For $x\in M, t\in I$ and $r>0$,
let $B(x,t,r)$ denote the open ball of radius $r$ centered at $x$ with respect to $g_t$.
Note that $\partial_t(d m_t)=-\HH\,dm_t$.
The non-negativity of the M\"uller quantity (i.e., $\mathcal{D}(V)\geq 0$ for all vector fields $V$) is expressed by $\mathcal{D}\geq 0$.

At the beginning of the proof of each statement,
we carry out the parallel translation for time and the parabolic rescaling.
Recall that
for $r>0$,
the parabolic rescaling (at time $0$) is given by $\bar{g}_s:=r^{-2}\,g_{r^2 s}=r^{-2}g_t$.
Note that
the super Ricci flow and the non-negativity of the M\"uller quantity are preserved under this rescaling (cf. \cite[Remark 2.2]{KuS}).

We notice the following lower bound for $\HH$,
which is a direct consequence of the maximum principle (see \cite[Lemma 3.2]{FZ}):
\begin{lem}[\cite{FZ}]\label{lem:scal}
Assume $\mathcal{D}\geq 0$.
If $\HH (\cdot, t_0) \geq -A$ for $A>0$, then for all $t\in [t_0,\infty) \cap I$,
\begin{equation*}
\HH (\cdot, t) \geq -\frac{n}2 \frac{A}{(n/2) + A (t - t_0)}.
\end{equation*}
\end{lem}

\begin{rem}\label{rem:scal2}
The following also holds (see \cite[Lemma 3.2]{FZ}):
Assume $\mathcal{D}\geq 0$.
If $t_0\in I$,
then for all $t\in [t_0,\infty) \cap I$, we have
\begin{equation*}
\HH (\cdot, t) \geq -\frac{n}{2(t-t_0)},\quad \min_M \HH(\cdot, t)\geq \min_M \HH(\cdot,t_0).
\end{equation*}
\end{rem}

\subsection{Heat kernels}\label{sec:heat kernel}
The \textit{heat operator} and the \textit{conjugate heat operator} are defined by
\begin{equation*}
\square:=\partial_t -\Delta,\quad \square^*:=-\partial_t-\Delta + \HH.
\end{equation*}
We recall the following Duhamel principle (see e.g., \cite{CCGG}):
For any $u,v \in C^{\infty}(M\times I)$,
\begin{equation}\label{eq:conjugate integration}
\frac{d}{dt}\int_M \,uv\,dm_t=\int_M\,(\square u) v\, dm_t-\int_M\, u (\square^* v)\,dm_t.
\end{equation}
For $x,y \in M$ and $s,t\in I$ with $s< t$,
we denote by $G(x,t;y,s)$ the \textit{heat kernel} (see e.g., \cite{CCGG}, \cite{G});
namely,
for a fixed $(y,s) \in M \times I$,
it solves
\begin{equation*}
 (\partial_t - \Delta_x ) G(\cdot,\cdot;y,s) = 0,\quad \lim_{t \searrow s} G(\cdot,t;y,s)=\delta_y.
\end{equation*}
Notice that
$G(x,t;\cdot,\cdot)$ is the fundamental solution to the conjugate heat equation;
namely,
for any $(x,t) \in M \times I$,
\begin{equation*}
(-\partial_s -\Delta_y +\HH)G(x,t;\cdot,\cdot) = 0,\quad \lim_{s \nearrow t} G(x,t;\cdot, s) = \delta_x.
\end{equation*}
The reproduction formula (or semigroup property)
\begin{equation}\label{eq:semigroup}
G(x,t;y,s)=\int_{M}\,G(x,t; \cdot, \sigma)G(\cdot,\sigma; y,s)\,dm_{\sigma}
\end{equation}
holds for any $\sigma \in (s,t)$.

For $(x_0,t_0)\in M\times I$,
the \textit{conjugate heat kernel measure} $\nu_{(x_0,t_0)}=(\nu_{(x_0,t_0);t})_{t \in (-\infty,t_0] \cap I}$ is defined by
\begin{equation*}
\nu_{(x_0,t_0);t}:=G(x_0,t_0;\cdot,t)\,m_{t},\quad \nu_{(x_0,t_0);t_0} := \delta_{x_0}
\end{equation*}
for $t<t_0$,
which are probability measures.

\begin{rem}
We frequently set $\nu:=\nu_{(x_0,t_0)}$ or $\nu^0:=\nu_{(x_0,t_0)}$.
In that case,
without mentioning,
we use the notation $\nu_t:=\nu_{(x_0,t_0);t}$ or $\nu^0_t:=\nu_{(x_0,t_0);t}$,
respectively.
\end{rem}

We define $\tau:=t_0-t$, which is called the \textit{parameter}.
The \textit{density} $f \in C^\infty ( M \times ( (-\infty,t_0)\cap I))$ is determined by
\begin{equation*}
G(x_0,t_0;\cdot,t) = (4\pi \tau)^{-n/2} e^{-f(\cdot,t)},
\end{equation*}
which can be written as
\begin{equation}\label{eq:potential}
f(x,t)=-\log G(x_0,t_0;x,t)-\frac{n}{2}\log \tau-\frac{n}{2}\log 4\pi.
\end{equation}
The density enjoys
\begin{equation}\label{eq:potential_enjoy}
- \partial_t f = \Delta f - |\nabla f|^2 + \HH - \frac{n}{2\tau}. 
\end{equation}
The following can be derived from \eqref{eq:conjugate integration} and integration by parts (cf. \cite[Lemma 4.4]{B3}):
\begin{equation}\label{eq:conjugate integration2}
\frac{d}{dt}\int_M\,u\,d\nu_{(x_0,t_0);t}=\int_M\,\square u\,d\nu_{(x_0,t_0);t},\quad \int_M \,\Div (V)\,d\nu_{(x_0,t_0);t}=\int_{M}\,\langle V,\nabla f \rangle\,d\nu_{(x_0,t_0);t}
\end{equation}
for all $u\in C^{\infty}(M\times I)$ and vector fields $V$.

\begin{rem}
We do not always introduce the notations of the parameter and the density when it is clear from the context.
\end{rem}

We define
\begin{equation}\label{eq:Harnack notation}
\Psi:=h + \nabla^2 f - \frac1{2\tau} g,\quad w:=\tau (2\Delta f - |\nabla f|^2 + \HH ) + f -n.
\end{equation}
We possess the following Perelman type Harnack estimate (see \cite{P}, \cite[Theorems 1.1, 1.2]{CGT}):
\begin{thm}[\cite{P}, \cite{CGT}]\label{thm:Harnack}
We set $u:=G(x_0,t_0; \cdot,\cdot)$.
Then we have
\begin{equation*}
\square^* (wu) = - 2 \tau \left| \Psi \right|^2 u-\tau \,u \,\mathcal{D}(\nabla f).
\end{equation*}
Moreover,
if $\mathcal{D}\geq 0$,
then $w\leq 0$.
\end{thm}

\subsection{Wasserstein distance}

Let $(X,d)$ be a complete separable metric space,
and let $\mu_1, \mu_2$ be two Borel probability measures.
We denote by $W_1(\mu_1, \mu_2)$ the \textit{$L^1$-Wasserstein distance} between them.
The canonical formulation of the Wasserstein distance is based on a coupling method.
We would rather make use of the following characterization called the \textit{Kantorovich-Rubinstein duality} (see e.g., \cite[Theorem 5.10]{V}):
\begin{equation}\label{eq:W1}
W_1(\mu_1, \mu_2)=\sup_{\phi}  \bigg( \int_X \phi \, d\mu_2  - \int_X \phi \, d\mu_1 \bigg), 
\end{equation}
where the supremum is taken over all bounded $1$-Lipschitz functions $\phi:X \to \mathbb{R}$.
We also denote by $\Var (\mu_1, \mu_2)$ the \textit{variance} defined by
\begin{equation*}
\Var (\mu_1, \mu_2) := \int_X \int_X d^2(x_1, x_2) \, d\mu_1(x_1) d\mu_2 (x_2).
\end{equation*}

We recall the following (see e.g., \cite[Lemma 2.8]{B2}):
\begin{lem}[\cite{B2}]\label{lem:W-V}
For any two Borel probability measures $\mu_1,\mu_2$ on $(X,d)$,
we have
\begin{equation*}
W_1(\mu_1, \mu_2) \leq \sqrt{\Var(\mu_1, \mu_2)}.
\end{equation*}
\end{lem}

\section{Super Ricci flow}\label{sec:SRF}
In this section,
we summarize several facts on super Ricci flow.

\subsection{Wasserstein monotonicity and concentration bounds}\label{sec:Concentration bounds}

We first recall the following monotonicity property for the Wasserstein distance (see \cite[Lemma 2.7]{B1}):
\begin{prop}[\cite{B1}]\label{prop:monotonicity_W1}
Let $(M,(g_t)_{t\in I})$ be a super Ricci flow.
Let $(x_0,t_0),(x_1,t_1)\in M\times I$.
Then $W_1(\nu_{(x_0,t_0); t}, \nu_{(x_1, t_1); t} )$ is non-decreasing in $t\in (-\infty,\min\{t_0,t_1\}]\cap I$.
\end{prop}

For $(x_0,t_0) \in M \times I$,
a point $(z,t) \in M \times ((-\infty,t_0] \cap I)$ is said to be a \textit{$($$\Cn$-$)$center} if
\begin{equation*}
\Cn:=\frac{(n-1)\pi^2}{2} + 4,\quad \Var(\delta_z, \nu_{(x_0,t_0);t}) \leq \Cn (t_0-t).
\end{equation*}

For the existence of centers,
we have the following (see \cite[Proposition 3.12]{B1}):
\begin{prop}[\cite{B1}]\label{prop:center}
Let $(M,(g_t)_{t\in I})$ be a super Ricci flow.
Let $(x_0,t_0) \in M \times I$.
Then for every $t \in (-\infty,t_0] \cap I$,
there is $z \in M$ such that $(z,t)$ is a center of $(x_0, t_0)$.
\end{prop}

We possess the following concentration bounds (see \cite[Proposition 3.13, Theorem 3.14]{B1}):
\begin{lem}[\cite{B1}]\label{lem:concentration1}
Let $(x_0,t_0) \in M \times I$,
and let $(z,t)$ be a center of $(x_0, t_0)$.
Then for every $R > 1$ we have
\begin{equation*}
\nu_{(x_0,t_0);t} \left(  B(z, t, \sqrt{R \Cn (t_0 - t)} ) \right) \geq 1- \frac{1}{R}.
\end{equation*}
\end{lem}

\begin{rem}
We also use Lemma \ref{lem:concentration1} in the following form:
For all $r>0$ we have
\begin{equation}\label{eq:elementary concentration}
\nu_{(x_0,t_0);t} \left( M \setminus B(z, t, r) \right)\leq \frac{\Var(\nu_{(x_0,t_0);t},\delta_z)}{r^2}\leq \frac{\Cn (t_0-t)}{r^2}.
\end{equation}
\end{rem}

\begin{prop}[\cite{B1}]\label{prop:concentration2}
Let $(M,(g_t)_{t\in I})$ be a super Ricci flow.
Let $(x_0,t_0) \in M \times I$,
and let $(z,t)$ be a center of $(x_0, t_0)$.
Then for all $r> 0$ we have
\begin{equation*} 
 \nu_{(x_0,t_0);t} \left( M \setminus B(z, t, r) \right)\leq 2 \exp \left( {- \frac{\left(r - \sqrt{2\Cn (t_0-t)} \right)_+^2}{8(t_0-t)}  }\right). 
\end{equation*}
\end{prop}


\subsection{Analytic bounds}\label{sec:Gradient estimates}

Let $\Phi : \IR \to (0,1)$ be a function defined by
\begin{equation}\label{eq:Gaussain}
\Phi(x):=\int^{x}_{-\infty}\,(4\pi)^{-1/2} e^{-y^2/4}\,dy.
\end{equation}
For $t>0$,
let $\Phi_t : \IR \to (0,1)$ stand for a function defined by $\Phi_t (x):=\Phi (t^{-1/2} x)$,
which is the solution to the one dimensional heat equation such that its initial condition is the Heaviside function.
We have the following sharp gradient estimate (see \cite[Theorem 4.1]{B1}):
\begin{thm}[\cite{B1}]\label{thm:gradient_estimate}
Let $(M,(g_t)_{t\in I})$ be a super Ricci flow.
For $[t_0, t_1] \subset I$,
let $u \in C^\infty (M \times  [t_0, t_1])$ be a solution to the heat equation.
Assume that
$u$ only takes values in $(0,1)$,
and $| \nabla (\Phi_T^{-1} ( u (\cdot , t_0) ))| \leq 1$ for $T>0$.
Then $| \nabla (\Phi^{-1}_{t - t_0+T} ( u(\cdot, t) ))| \leq 1$ for all $t \in [t_0, t_1]$.
\end{thm}

\begin{rem}
Due to Theorem \ref{thm:gradient_estimate},
every super Ricci flow can be regarded as a metric flow in the sense of \cite[Definition 3.1]{B2}.
Note that
Proposition \ref{prop:monotonicity_W1} holds in the framework of metric flows (see \cite[Proposition 3.16]{B2}).
Furthermore,
every super Ricci flow satisfies a certain concentration property (see \cite[Corollary 3.8]{B1}).
In particular,
it can be regarded as an $H_n$-concentrated metric flow in the sense of \cite[Definition 3.21]{B2},
which is a crucial object in the compactness theory of metric flows (see \cite[Theorem 7.4, Corollary 7.5]{B2}).
Proposition \ref{prop:center} immediately follows from the $H_n$-concentration property (see \cite[Proposition 3.25]{B2}).
\end{rem}

We next present the following heat kernel estimates (see \cite[Proposition 4.2]{B1}):
\begin{prop}[\cite{B1}]\label{prop:G_bound1}
Let $(M,(g_t)_{t\in I})$ be a super Ricci flow.
Let $x\in M$ and $[s,t] \subset I$.
Then for every $p\in [1,\infty)$ we have
\begin{equation}\label{eq:G_bound11}
 (t-s)^{p/2} \int_M \left( \frac{|\nabla_x G(x,t; \cdot, s)|}{G(x,t; \cdot, s)} \right)^p d\nu_{(x,t);s} \leq C_{p};
\end{equation}
moreover,
for every Borel subset $\Omega \subset M$ we have
\begin{equation}\label{eq:G_bound13}
 (t-s)^{p/2} \int_\Omega \left( \frac{|\nabla_x G(x,t; \cdot, s)|}{G(x,t; \cdot, s)} \right)^p d\nu_{(x,t);s} \leq C_{p}\, \nu_{(x,t);s}(\Omega)(-\log (\nu_{(x,t);s}(\Omega)/2))^{p/2}. 
\end{equation}
Also,
for every $v \in T_x M$ with $|v|_t = 1$,
\begin{equation}\label{eq:G_bound12}
(t-s) \int_M \left( \frac{\partial_{v} G(x,t; \cdot, s)}{G(x,t; \cdot, s)} \right)^2 d\nu_{(x,t);s} \leq \frac12.
\end{equation}
\end{prop}

\begin{rem}\label{rem:G_bound}
In \eqref{eq:G_bound11} and \eqref{eq:G_bound13}, one can choose $C_{2} = n/2$.
\end{rem}

We have the following Poincar\'e inequality (see \cite[Theorem 1.5]{HaN}, \cite[Theorem 11.1]{B1}, and also \cite[Theorem 1.10]{HN}):
\begin{prop}[\cite{HaN}, \cite{B1}]\label{prop:Poincare}
Let $(M,(g_t)_{t\in I})$ be a super Ricci flow.
Let $(x_0,t_0) \in M \times I$.
Let $\tau>0$ satisfy $[t_0 - \tau,t_0] \subset I$.
Then for any $u \in C^1(M)$ with $\int_M u \, d\nu_{(x_0,t_0);t_0-\tau}=0$ and $p \in [1,\infty)$,
\begin{equation*}
\int_M |u|^p d\nu_{(x_0,t_0);t_0-\tau} \leq C_p \,\tau^{p/2}\, \int_M |\nabla u|^p d\nu_{(x_0,t_0);t_0-\tau}.
\end{equation*}
\end{prop}
\begin{proof}
Bamler \cite{B1} has stated this assertion only for Ricci flow,
but the argument also works for super Ricci flow without any changes.
For $p=2$,
the desired inequality has been obtained by Haslhofer-Naber \cite{HaN} (see \cite[Theorem 1.5]{HaN}).
For $p=1$,
by noticing the following point,
one can prove the desired one along the lines of the proof of \cite[Theorem 11.1]{B1}:
If $\square u=0$,
then the Bochner formula and the Kato inequality imply
\begin{equation*}
\square |\nabla u|^2=2h(\nabla u,\nabla u)-2\Ric(\nabla u,\nabla u)-2|\nabla^2 u|^2\leq -\frac{|\nabla |\nabla u|^2|^2}{2|\nabla u|^2};
\end{equation*}
in particular,
$\square |\nabla u|\leq 0$.
For general $p$,
the desired one follows from the assertion for $p=1$ together with the same argument of the proof of \cite[Theorem 11.1]{B1}.
\end{proof}

\begin{rem}\label{rem:Poincare}
In Proposition \ref{prop:Poincare}, we may choose $C_1 = \sqrt{\pi}$ and $C_2 = 2$.
\end{rem}

\section{Nash entropy}\label{sec:Nash entropy}
In this section,
we examine basic properties of the Nash entropy.

\subsection{Entropy monotonicity}

Fix a base point $(x_0, t_0) \in M \times I$.
For $\tau > 0$ with $[t_0 - \tau, t_0] \subset I$,
the \textit{pointed Nash-entropy} is defined by
\begin{equation*}
\NN_{(x_0, t_0)} (\tau):= \int_M \,f\,d\nu_{(x_0,t_0);t_0-\tau} - \frac{n}2.
\end{equation*}
We set $\NN_{(x_0, t_0)} (0) := 0$ such that $\NN_{(x_0,t_0)}(\tau)$ is continuous in $\tau$ (cf. \cite[Proposition 5.2]{B1}, \cite{GPT}).
The \textit{pointed $\mathcal{W}$-entropy} is defined by
\begin{equation*}
\WW_{(x_0,t_0)}(\tau):=\int_M \left( \tau (|\nabla f|^2 + \HH) + f - n \right) d\nu_{(x_0,t_0);t_0-\tau}.
\end{equation*}
We recall the following monotonicity properties (cf. \cite[Proposition 5.2]{B1}, \cite[Theorem 3.1]{Hu}, \cite[Theorem 5.2]{GPT}, \cite[Lemma 3.1]{FZ}):
\begin{lem}\label{lem:Nash} 
Assume $\mathcal{D}\geq 0$.
Let $(x_0, t_0) \in M \times I$.
If $\tau > 0$ with $[t_0 - \tau, t_0] \subset I$, then
\begin{align}\label{eq:Nash2}
 &\frac{d}{d\tau} \left(\tau \NN_{(x_0, t_0)} (\tau)\right) = \WW_{(x_0,t_0)}(\tau) \leq 0,\\ \label{eq:Nash3}
 &\frac{d^2}{d\tau^2} \left( \tau \NN_{(x_0, t_0)} (\tau) \right) =  - \tau \int_M \left(2\left| \Psi \right|^2+\mathcal{D}(\nabla f) \right) d\nu_{(x_0,t_0);t_0 - \tau} \leq 0,\\ \label{eq:Nash3.5}
 &\frac{d}{d\tau}\NN_{(x_0, t_0)} (\tau) \leq 0,\\ \label{eq:Nash rem}
 &\WW_{(x_0,t_0)}(\tau) \leq \NN_{(x_0, t_0)} (\tau),
\end{align}
where $\Psi$ is defined as \eqref{eq:Harnack notation}.
\end{lem}
\begin{proof}
The formulas \eqref{eq:Nash2} and \eqref{eq:Nash3} are well-known (see e.g., \cite[Theorem 3.1]{Hu}, \cite[Theorem 5.2]{GPT}, \cite[Lemma 3.1]{FZ}).
Furthermore,
\eqref{eq:Nash3.5} and \eqref{eq:Nash rem} can be derived from the same calculation as in the proof of \cite[Proposition 5.2]{B1} together with \eqref{eq:Nash2} and \eqref{eq:Nash3}.
\end{proof}

We also see the following:
\begin{lem}\label{lem:Nash_more} 
Assume $\mathcal{D}\geq 0$.
Let $(x_0, t_0) \in M \times I$.
Let $\tau_1,\tau_2>0$ satisfy $\tau_1 \leq \tau_2$ and $[t_0 - \tau_2, t_0] \subset I$.
For $A>0$,
we assume $\HH(\cdot,t_0-\tau_2) \geq -A$.
Then we have
\begin{equation*}
 \NN_{(x_0, t_0)} (\tau_1)-\left((\tau_2-\tau_1)A+\frac{n}{2}\log\left( \frac{\tau_2}{\tau_1} \right)\right) \leq \NN_{(x_0, t_0)} (\tau_2). 
\end{equation*}
\end{lem}
\begin{proof}
From \eqref{eq:Nash2} we deduce
\begin{align*}
 \frac{d}{d\tau} \left(\tau \NN_{(x_0, t_0)}(\tau) \right) =\WW_{(x_0,t_0)}(\tau)
&= \int_M \tau \left( |\nabla f|^2 + \HH  \right)d\nu_{(x_0,t_0);t_0 - \tau} + \NN_{(x_0, t_0)}(\tau) - \frac{n}2 \\
&\geq -\tau A + \NN_{(x_0, t_0)}(\tau) - \frac{n}2,
\end{align*}
and hence
\begin{equation*}
\frac{d}{d\tau} \NN_{(x_0, t_0)} (\tau)\geq -\left(A+ \frac{n}{2\tau}\right).
\end{equation*}
This implies
\begin{equation*}
 \NN_{(x_0, t_0)} (\tau_2) - \NN_{(x_0, t_0)}(\tau_1)\geq -\int^{\tau_2}_{\tau_1} \,\left(A+ \frac{n}{2\tau}\right)\, d\tau= -\left((\tau_2-\tau_1)A+\frac{n}{2}\log\left( \frac{\tau_2}{\tau_1} \right)\right).
\end{equation*}
We complete the proof.
\end{proof}

Lemma \ref{lem:Nash} and Proposition \ref{prop:Poincare} yield the following (cf. \cite[Proposition 5.13]{B1}):
\begin{lem}\label{lem:Nash basic}
Let $(M,(g_t)_{t\in I})$ be a super Ricci flow with $\mathcal{D}\geq 0$.
Let $(x_0, t_0) \in M \times I$.
For $A>0$,
we assume $\HH(\cdot, t_0 - \tau) \geq -A$.
Then we have
\begin{align}\label{eq:Nash basic1}
 &\int_M \tau ( |\nabla f|^2 + \HH) d\nu_{(x_0,t_0);t_0-\tau} \leq \frac{n}2,\\ \label{eq:Nash basic2}
&\int_M  \left( f - \NN_{(x_0, t_0)} (\tau) - \frac{n}2 \right)^2 d\nu_{(x_0,t_0);t_0-\tau} \leq n+ 2 A\tau.
\end{align}
\end{lem}
\begin{proof}
By the same calculation as in the proof of \cite[Proposition 5.13]{B1},
and by \eqref{eq:Nash2}, \eqref{eq:Nash3.5},
\begin{align*}
 \int_M \tau ( |\nabla f|^2 + \HH ) d\nu_{(x_0,t_0);t_0 - \tau}&= \frac{n}2 + \tau  \frac{d}{d\tau} \NN_{(x_0, t_0)} (\tau) \leq \frac{n}2,
\end{align*}
which proves \eqref{eq:Nash basic1}.
Proposition \ref{prop:Poincare} with $p=2$ together with 
\begin{equation*}
\int_M  \left( f - \NN_{(x_0, t_0)} (\tau) - \frac{n}2 \right) d\nu_{(x_0,t_0);t_0-\tau} = 0
\end{equation*}
and
\begin{equation*}
2\tau \int_M |\nabla f|^2 d\nu_{(x_0,t_0);t_0 - \tau} \leq n - 2\tau \int_{M} \HH \, d\nu_{(x_0,t_0);t_0 - \tau}   \leq  n + 2A \tau
\end{equation*}
leads us to \eqref{eq:Nash basic2} (see also Remark \ref{rem:Poincare}).
We arrive at the desired estimates.
\end{proof}

\subsection{Derivative estimates}
For a fixed $s\in I$,
we define
\begin{equation*}
\NN_s (x, t) := \NN_{(x, t)} (t-s)
\end{equation*}
on $M \times ( (s, \infty) \cap I)$.
We have the following derivative estimates (cf. \cite[Theorem 5.9]{B1}):
\begin{lem}\label{lem:Nash thm}
Let $(M,(g_t)_{t\in I})$ be a super Ricci flow with $\mathcal{D}\geq 0$.
Let $s \in I$.
For $A>0$,
we assume $\HH (\cdot, s) \geq -A$.
Then on $M \times ( (s, \infty) \cap I)$ we have
\begin{align}\label{eq:Nash thm gradient}
 &|\nabla \NN_s |\leq \left( \frac{n}{2(t-s)}+A \right)^{1/2},\\ \label{eq:Nash thm heat}
 &- \frac{n}{2(t-s)} \leq  \square \NN_s \leq 0. 
\end{align}
\end{lem}
\begin{proof}
We may assume $s = 0$.
Fix a base point $(x,t) \in M \times ( (0, \infty) \cap I)$.
For $v \in T_x M$ with $|v|_t = 1$,
the same calculation as in the proof of \cite[Theorem 5.9]{B1} tells us that
\begin{align*}
 &\quad\,\, \partial_{v} \NN_0 (x,t)\\
&= \int_M  \left(\frac{\partial_{v}G(x,t;y,0)}{G(x,t;y,0)}\right) \left( f (y,0) - \NN_0 (x,t) - \frac{n}2 \right) d\nu_{(x,t);0} (y) \\
&\leq \left( \int_M \left(  \frac{\partial_{v} G(x,t;y,0)}{G(x,t;y,0)} \right)^2 d\nu_{(x,t);0}(y) \right)^{1/2} \left( \int_M  \left( f - \NN_0 (x,t) - \frac{n}2 \right)^2 d\nu_{(x,t);0}(y) \right)^{1/2};
\end{align*}
in particular,
\eqref{eq:G_bound12} and \eqref{eq:Nash basic2} imply
\begin{equation*}
|\nabla \NN_0|^2(x,t) \leq \frac{1}{2t} (n + 2A t),
\end{equation*}
which is the desired gradient estimate \eqref{eq:Nash thm gradient} (see also Remark \ref{rem:scal2}).
Similarly,
we possess
\begin{align*}
\square \NN_0 (x,t)= \int_M \left( \frac{|\nabla_x G(x,t;y,0)|}{G(x,t;y,0)}  \right)^2 d\nu_{(x,t);0}(y) - \frac{n}{2t},
\end{align*}
and \eqref{eq:G_bound11} with $p=2$ leads us to \eqref{eq:Nash thm heat} (see also Remark \ref{rem:G_bound}).
\end{proof}

We conclude the following (cf. \cite[Corollary 5.11]{B1}):
\begin{lem}\label{lem:Nash cor}
Let $(M,(g_t)_{t\in I})$ be a super Ricci flow with $\mathcal{D}\geq 0$.
For a fixed $s\in I$,
let $(x_1,t_1),(x_2,t_2)\in M\times ( (s, \infty) \cap I)$.
For $A>0$,
we assume $\HH (\cdot, s) \geq -A$.
Then for every $t\in (s,\min\{t_1,t_2\}]$ we have
\begin{equation*}
 \NN_s(x_1, t_1) - \NN_s(x_2, t_2)  \leq \left( \frac{n}{2(t-s)} +  A \right)^{1/2}   W_1(\nu_{(x_1, t_1);t}, \nu_{(x_2, t_2);t})    + \frac{n}2 \log \left( \frac{t_2-s}{t-s} \right) . 
\end{equation*}
\end{lem}
\begin{proof}
We may assume $s = 0$.
For $i=1,2$, we set $\nu^{i}:=\nu_{(x_i,t_i)}$.
In virtue of \eqref{eq:Nash thm heat},
we see
\begin{equation}\label{eq:Nash cor1}
 \NN_0 (x_i, t_i) \leq \int_M \NN_0 (\cdot, t) d\nu^i_{t} \leq \NN_0 (x_i, t_i) + \frac{n}2  \log \left( \frac{t_i}{t} \right). 
\end{equation}
On the other hand,
due to \eqref{eq:Nash thm gradient} and \eqref{eq:W1},
\begin{equation}\label{eq:Nash cor2}
 \left| \int_M \NN_0 (\cdot, t) d\nu^2_{t} - \int_M \NN_0 (\cdot, t) d\nu^1_{t} \right| \leq \left( \frac{n}{2t} + A \right)^{1/2} W_1( \nu^1_{t}, \nu^2_{t} ).
\end{equation}
The desired estimate follows from combining \eqref{eq:Nash cor1} and \eqref{eq:Nash cor2}.
\end{proof}

We will use Lemma \ref{lem:Nash cor} in the following form:
\begin{lem}\label{lem:Nash cor_rem}
Let $(M,(g_t)_{t\in I})$ be a super Ricci flow with $\mathcal{D}\geq 0$.
For a fixed $s\in I$,
let $(x_0,t_0)\in M\times ( (s, \infty) \cap I)$.
For $t\in  (s, t_0]$,
we assume that
$(z,t)$ is a center of $(x_0,t_0)$.
For $A>0$,
we further assume $\HH (\cdot, s) \geq -A$.
Then on $M$, we have
\begin{align*}
- \NN_s(\cdot, t)\leq -\NN_s(x_0,t_0)+\left( \frac{n}{2(t-s)} +  A\right)^{1/2}   \sqrt{H_n(t_0-t)}+\left( \frac{n}{2(t-s)}+A \right)^{1/2}d_{t}(z,\cdot).
\end{align*}
\end{lem}
\begin{proof}
We may assume $s=0$.
Lemmas \ref{lem:W-V} and \ref{lem:Nash cor} yield
\begin{align}\label{eq:Nash cor_rem1}
 \NN_0(x_0, t_0) - \NN_0(z,t)  &\leq \left( \frac{n}{2t} +  A \right)^{1/2}   W_1(\delta_{z},\nu_{(x_0, t_0);t})\\ \notag
 &\leq \left( \frac{n}{2t} +  A \right)^{1/2}   \sqrt{\Var(\delta_z, \nu_{(x_0, t_0);t})}\leq \left( \frac{n}{2t} +  A\right)^{1/2}   \sqrt{H_n(t_0-t)}.
\end{align}
Further,
\eqref{eq:Nash thm gradient} and \eqref{eq:Nash cor_rem1} lead us to 
\begin{align*}
- \NN_0(\cdot, t)&\leq -\NN_0(z,t)+\left( \frac{n}{2t}+A \right)^{1/2}d_{t}(z,\cdot)\\
                      &\leq -\NN_0(x_0,t_0)+\left( \frac{n}{2t} + A\right)^{1/2}   \sqrt{H_n(t_0-t)}+\left( \frac{n}{2t}+A \right)^{1/2}d_{t}(z,\cdot)
\end{align*}
on $M$.
This proves the lemma.
\end{proof}

We also use Lemma \ref{lem:Nash cor} in the following form:
\begin{lem}\label{lem:useful}
Let $(M,(g_t)_{t\in I})$ be a super Ricci flow with $\mathcal{D}\geq 0$.
Let $(x_0,t_0),(x_1,t_1)\in M\times I$ with $t_0\leq t_1$.
For $r>0,\alpha\in (0,1)$,
we assume $[t_0-2\alpha^{-1} r^2,t_0] \subset I$ and $0\leq t_1-t_0\leq \alpha^{-1} r^2$.
For $A>0$,
we assume $\HH (\cdot, t_0-2\alpha^{-1} r^2) \geq -A\,r^{-2}$.
For $D>0$,
we assume $W_1(\nu_{(x_0,t_0);s_0},\nu_{(x_1,t_1);s_0})\leq D r$
for some $s_0\in [t_0-\alpha^{-1} r^2,t_0-\alpha r^2]$.
Then we have
\begin{equation*}
\NN_{(x_1, t_1)} (r^2)\geq \NN_{(x_0,t_0)}(r^2)-C_{D,\alpha,A}.
\end{equation*}
\end{lem}
\begin{proof}
We may assume $r=1$.
By Lemma \ref{lem:Nash_more},
we have
\begin{equation}\label{eq:useful1}
\NN_{(x_0, t_0)} (2\alpha^{-1})\geq \NN_{(x_0, t_0)} (1)-C_{\alpha,A}. 
\end{equation}
Furthermore,
by Lemma \ref{lem:Nash cor} and Proposition \ref{prop:monotonicity_W1},
\begin{align}\label{eq:useful2}
&\quad \,\,\NN_{(x_0,t_0)}(2\alpha^{-1})-\NN_{(x_1,t_1)}(t_1-t_0+2\alpha^{-1})\\ \notag
&=\NN_{t_0-2\alpha^{-1}}(x_0,t_0)-\NN_{t_0-2\alpha^{-1}}(x_1,t_1)\\ \notag
&\leq \left( \frac{n}{2\alpha^{-1}} +A\right)^{1/2}W_1(\nu_{(x_0,t_0);t_0-\alpha^{-1}},\nu_{(x_1,t_1);t_0-\alpha^{-1}})+\frac{n}{2}\log 3\\ \notag
&\leq \left( \frac{n}{2\alpha^{-1}} +A\right)^{1/2}W_1(\nu_{(x_0,t_0);s_0},\nu_{(x_1,t_1);s_0})+\frac{n}{2}\log 3\leq C_{D,\alpha,A}.
\end{align}
From \eqref{eq:Nash3.5}, \eqref{eq:useful1} and \eqref{eq:useful2},
we derive
\begin{equation*}
\NN_{(x_1, t_1)} (1)\geq \NN_{(x_1,t_1)}(t_1-t_0+2\alpha^{-1}) \geq \NN_{(x_0,t_0)}(2\alpha^{-1})-C_{D,\alpha,A}\geq \NN_{(x_0,t_0)}(1)-C_{D,\alpha,A}.
\end{equation*}
This completes the proof.
\end{proof}

\section{Volume and heat kernel estimates}\label{sec:vol_heat}

In this section,
we derive several volume and heat kernel estimates.

\subsection{Lower volume estimates}

We begin with the following (cf. \cite[Theorem 6.2]{B1}):
\begin{prop}\label{prop:lower volume}
Let $(M,(g_t)_{t\in I})$ be a super Ricci flow with $\mathcal{D}\geq 0$.
Let $(x_0,t_0)\in M\times I$.
For $r>0$,
we assume $[t_0-r^2, t_0] \subset I$.
For $A>0$,
we assume $\HH(\cdot, t_0-r^2) \geq -A\,r^{-2}$.
Let $(z,t_0-r^2) \in M \times I$ be a center of $(x_0,t_0)$.
Then we have
\begin{equation*}
m_{t_0-r^2}\left(B(z,t_0-r^2, \sqrt{2 \Cn} r)\right) \geq C_{A}\, \exp ( \NN_{(x_0,t_0)}(r^2)) r^n.
\end{equation*}
\end{prop}
\begin{proof}
We may assume $t_0 = 1$ and $r=1$.
Let $\nu:=\nu_{(x_0,1)}$.
By virtue of Lemma \ref{lem:concentration1},
\begin{equation}\label{eq:lower volume1}
 \nu_{0} ( B)  \geq \frac12,
\end{equation}
where $B:=B(z,0, \sqrt{2\Cn} )$.
Lemma \ref{lem:Nash basic} tells us that
\begin{equation}\label{eq:lower volume2}
\int_M \left| f - \NN_{(x_0,1)}(1) - \frac{n}2 \right| d\nu_{0} \leq \left( \int_M \left( f - \NN_{(x_0,1)}(1) - \frac{n}2 \right)^2 d\nu_{0} \right)^{1/2}  \leq (n+2A)^{1/2}.
\end{equation}
Combining \eqref{eq:lower volume1} and \eqref{eq:lower volume2},
we conclude
\begin{align*}
\frac1{\nu_{0} (B)} \int_B f \, d\nu_{0} &\geq \NN_{(x_0,1)}(1) + \frac{n}2 - \frac1{\nu_{0} (B)} \int_B \left| f - \NN_{(x_0,1)}(1) - \frac{n}2 \right| d\nu_{0}\\
&\geq  \NN_{(x_0,1)}(1) + \frac{n}2 - 2(n+ 2 A)^{1/2}.
\end{align*}
We set $u := (4\pi )^{n/2} e^{-f} / \nu_{0} (B)$.
Since $\int_B u \, dm_{0} = 1$,
\eqref{eq:lower volume1} tells us that
\begin{align*}
 \int_B u \log u \, dm_{0} &= - \frac1{\nu_0 (B)} \int_B f \, d\nu_0 - \log \nu_0 (B) + \frac{n}2 \log 4\pi \\
  &\leq  -\NN_{(x_0,1)}(1) -  \frac{n}2 + 2(n+ 2 A)^{1/2} + \log 2 + \frac{n}2 \log (4\pi). 
\end{align*}
By the Jensen inequality,
\begin{equation*}
\log \left( \frac1{m_0(B)} \int_B  u \, dm_{0} \right) \frac1{m_0(B)} \int_B  u \, dm_{0} \leq \frac1{m_0(B)} \int_B  u \log u \, dm_{0},
\end{equation*}
and thus
\begin{equation*}
- \log  m_0(B) \leq  \int_B  u \log u \, dm_{0} \leq - \NN_{(x_0,1)}(1) + C + 2 (n+ 2 A)^{1/2}.
\end{equation*}
We complete the proof.
\end{proof}

\subsection{Heat kernel estimates}

We next show the following (cf. \cite[Theorem 7.1]{B1}):
\begin{prop}\label{prop:HK}
Let $(M,(g_t)_{t\in I})$ be a super Ricci flow with $\mathcal{D}\geq 0$.
Let $(x_0,t_0)\in M\times I$,
and let $[t,t_0] \subset I$.
For $A>0$,
we assume $\HH(\cdot,t) \geq -A(t_0-t)^{-1}$.
Then on $M$,
we have
\begin{equation*}
G(x_0,t_0;\cdot,t) \leq \frac{C_{A}}{(t_0-t)^{n/2}} \exp ( - \NN_{(x_0,t_0)}(t_0-t) ).
\end{equation*}
\end{prop}
\begin{proof}
We may assume $t = 0$ and $t_0 = 1$.
For a fixed $y \in M$,
we set $u := G(\cdot, \cdot; y,0)$.
By the same argument as in the proof of \cite[Theorem 7.1]{B1},
it suffices to show the following (see also \cite[Theorem 26.25]{CCGG}):
If $L \geq \underline{L}_{A}$,
and if for all $(x,s) \in M \times (0,1]$ we have
\begin{equation}\label{eq:GboundW}
 u(x,s) \leq \frac{L}{s^{n/2}} \exp ( - \NN_0 (x,s)),
\end{equation}
then for all $x\in M$ we have
\begin{equation}\label{eq:GboundW_desired}
 u(x,1) \leq \frac{L}{2} \exp ( - \NN_0 (x,1)).
\end{equation}

First,
assuming \eqref{eq:GboundW},
we derive a gradient bound
\begin{equation}\label{eq:HK3}
 |\nabla u|(x,s) \leq C_{A} L \exp \left({ -  \NN_0 (x,s ) } \right)
\end{equation}
for all $(x,s) \in M \times [3/4, 1]$.
We set
\begin{equation*}
v := \left( s - \frac{1}{2} \right) |\nabla u|^2 + u^2.
\end{equation*}
Due to the Bochner formula,
\begin{equation*}
\square v = -|\nabla u|^2 - 2\left( s - \frac{1}{2} \right) |\nabla^2 u|^2-2\left( s-\frac{1}{2} \right)(\Ric -h)(\nabla u,\nabla u) \leq 0
\end{equation*}
for every $s \in [1/2 ,1]$.
Using \eqref{eq:GboundW},
we have
\begin{align}\label{eq:HK1}
\left( s - \frac{1}{2} \right) |\nabla u|^2 (x,s) \leq v(x,s) &\leq \int_M v(\cdot, 1/2 )  d\nu_{(x,s);1/2}= \int_M u^2(\cdot , 1/2) d\nu_{(x,s);1/2}\\ \notag
&\leq 2^{n}L^2 \int_M \exp \left({ - 2 \NN_0 (\cdot,1/2 ) } \right) d\nu_{(x,s);1/2}
\end{align}
for every $(x,s) \in M \times (1/2, 1]$.
Let $(z,1/2)$ be a center of $(x,s)$ (see Proposition \ref{prop:center}).
In view of Lemma \ref{lem:Nash cor_rem},
we have
\begin{equation}\label{eq:HK2}
- \NN_0 (\cdot,1/2) \leq  -\NN_0 (x,s) + C_{A}\, (d_{1/2} (z, \cdot)+1)
\end{equation}
on $M$.
From \eqref{eq:HK2}, the co-area formula and Proposition \ref{prop:concentration2},
we derive
\begin{align*} 
&\quad\,\,  \int_M  \exp \left({ - 2 \NN_0 (\cdot, 1/2 ) } \right) d\nu_{(x,s);1/2}\\
 &\leq C_{A} \exp ({ - 2 \NN_0 (x,s ) } )  \int_M \exp (C_{A} d_{1/2} (z, \cdot)) d\nu_{(x,s);1/2} \\
&\leq C_{A} \exp ({ - 2 \NN_0 (x,s ) } ) \left( \int_0^\infty e^{C_{A} r}\,\nu_{(x,s);1/2}(M \setminus B(z,1/2,r))\, dr + 1 \right) \\
&\leq C_{A} \exp ({ - 2 \NN_0 (x,s ) } )\left( \int_0^\infty e^{C_{A} r} \exp \left({ - \frac{\left(r - \sqrt{2\Cn (t-1/2)} \right)_+^2}{8 (t - 1/2)}}\right) dr +1 \right) \\
&= C_{A} \exp ({ - 2 \NN_0 (x,s ) } ) \left( \sqrt{t- 1/2} \int_0^\infty e^{C_{A} \sqrt{t- 1/2} \, r} \exp \left({ - \frac{\left(r - \sqrt{2\Cn} \right)_+^2}{8}}\right) dr +1 \right)\\
&\leq C_{A} \exp ({ - 2 \NN_0 (x,s ) } ).
\end{align*}
Combining this with \eqref{eq:HK1} implies \eqref{eq:HK3}.

Let $(x,1) \in M \times I$,
and set $\nu:=\nu_{(x,1)}$.
For $\xi \in (0, 1/2]$,
we put $t_1 := 1- \xi^2$,
and let $(z_1, t_1)$ be a center of $(x,1)$.
We set $B := B(z_1, t_1, \sqrt{100 \Cn \xi^2})$.
By Lemma \ref{lem:Nash cor_rem},
if $\xi \leq \ov\xi_{A}$, then
\begin{equation}\label{eq:HK4}
  - \NN_0 (\cdot, t_1) \leq  - \NN_0 (x,1)+C_{A} d_{t_1} (z_1, \cdot) + C_{A} \sqrt{\Cn \xi^2} \leq - \NN_0 (x,1)+C_{A} + \log 2
\end{equation}
on $B$;
in particular,
\eqref{eq:HK3} tells us that
on $B$,
\begin{equation}\label{eq:HK3.5}
|\nabla u| (\cdot, t_1) \leq C_{A} L \exp ( - \NN_0(x,1) ).
\end{equation}
Now,
we write
\begin{equation}\label{eq:HK5}
u(x,1) = \int_B u \,d\nu_{t_1} + \int_{M \setminus B} u  \,d\nu_{t_1}
\end{equation}
with the help of \eqref{eq:semigroup}.
We first estimate the first term in \eqref{eq:HK5}.
Since
\begin{equation*}
\frac{d}{ds} \int_M u \, dm_s = - \int_M \HH u\, dm_s \leq  A\int_M u \, dm_s
\end{equation*}
for all $s \in (0,1]$, we possess
\begin{equation}\label{eq:int u}
 \int_M u \, dm_{t_1} \leq e^{A}.
\end{equation}
By Proposition \ref{prop:lower volume} and \eqref{eq:Nash3.5} we have
\begin{equation}\label{eq:HK6}
m_{t_1}(B) \geq C_{A} \exp (\NN_{(x,1)} ( \xi^2 ) ) \xi^n \geq C_{A} \exp ( \NN_0 (x,1) ) \xi^n. 
\end{equation}
In virtue of \eqref{eq:HK3.5},
for all $x_1, x_2 \in B$ we have
\begin{equation}\label{eq:HK6.5}
u (x_1, t_1)   \leq u(x_2, t_1) +  C_{A}  L  \exp ( - \NN_0 (x,1) ) \xi.
\end{equation}
Integrating \eqref{eq:HK6.5} over $B$ with respect to $x_2$,
and using \eqref{eq:int u} and \eqref{eq:HK6} imply
\begin{equation*}
 u(x_1, t_1)  \leq  \frac1{ m_{t_1}(B)} \int_B u \, dm_{t_1} + C_{A} L \exp ( - \NN_0 (x,1) ) \xi \leq C_{A}( C_{A} \xi^{-n} + L \xi ) \exp ( - \NN_0 (x,1)). 
\end{equation*}
It follows that
\begin{equation}\label{eq:HK7}
 \int_B u \, d\nu_{t_1} \leq C_{A}( C_{A} \xi^{-n} + L \xi ) \exp ( - \NN_0 (x,1)). 
\end{equation}

We next estimate the second term in \eqref{eq:HK5}.
By \eqref{eq:GboundW}, \eqref{eq:HK4}, an inequality $e^s \leq 1 + \xi e^{s/\xi}$, Lemma \ref{lem:concentration1}, Proposition \ref{prop:concentration2} and the co-area formula,
if $\xi \leq \ov\xi_{A}$, then
\begin{align}\label{eq:HK8}
 \int_{M \setminus B} u \, d\nu_{t_1} 
&\leq 2 L \int_{M \setminus B} \exp ( - \NN_0 (\cdot, t_1) ) d\nu_{t_1}\\ \notag
&\leq 4 L \exp ( - \NN_0 (x, 1) ) \int_{M \setminus B} \exp \left( C_{A} d_{t_1} (z_1, \cdot) \right) d\nu_{t_1} \\ \notag
&\leq 4 L \exp ( - \NN_0 (x, 1) )  \int_{M \setminus B} \left( 1 + \xi \exp \left( C_{A} \xi^{-1} d_{t_1} (z_1, \cdot) \right) \right) d\nu_{t_1} \\ \notag
&\leq 4L \exp ( - \NN_0 (x, 1) )\left( \frac{1}{100} +  \xi  \int_{M}\exp \left( C_{A}  \xi^{-1} d_{t_1} (z_1, \cdot) \right) d\nu_{t_1} \right) \\ \notag
&\leq 4L \exp ( - \NN_0 (x, 1) )\left( \frac{1}{100} +  \xi + C_{A}   \int_0^\infty e^{C_{A} r/\xi}  \nu_{t_1}(M \setminus B(z_1, t_1, r))\, dr \right)\\ \notag
&\leq 4L \exp ( - \NN_0 (x, 1) )\left( \frac{1}{40} + C_{A}   \int_0^\infty e^{C_{A} r/\xi}  \exp \left({- \frac{ ( r - \sqrt{2\Cn \xi^2})_+^2}{8 \xi^2} }\right)  dr \right) \\ \notag
&= 4L \exp ( - \NN_0 (x, 1) )\left( \frac{1}{40} + C_{A}  \xi  \int_0^\infty e^{C_{A} r}  \exp \left({- \frac{ ( r - \sqrt{2\Cn} )_+^2}{8 } }\right)  dr \right)  \\ \notag
&\leq \left( \frac1{10} + C_{A} \xi \right) L \exp ( - \NN_0 (x, 1) ).
\end{align}
Combining \eqref{eq:HK5}, \eqref{eq:HK7} and \eqref{eq:HK8} yields that
if $\xi \leq \ov\xi_{A}$ and $L \geq \underline{L}_{A,\xi}$, then
\begin{equation*}
u(x,1) \leq \left(C_{A}C_{A\xi^2}  \xi^{-n} + \frac{L}{10} +C_{A}L \xi  \right)  \exp ( - \NN_0 (x,1) )\leq \frac{L}{2}\exp ( - \NN_0 (x,1) ).
\end{equation*}
This proves the desired assertion \eqref{eq:GboundW_desired}.
\end{proof}

We also have the following gradient estimate (cf. \cite[Theorem 7.5]{B1}):
\begin{prop}\label{prop:HKK}
Let $(M,(g_t)_{t\in I})$ be a super Ricci flow with $\mathcal{D}\geq 0$.
Let $(x,t)\in M\times I$,
and let $[s,t] \subset I$.
For $A>0$,
we assume $\HH(\cdot,s) \geq -A(t-s)^{-1}$.
Then on $M$, we have
\begin{equation*}
\frac{|\nabla_{x} G|(x,t;\cdot,s)}{G(x,t;\cdot,s)} \leq \frac{C_{A}}{(t-s)^{1/2}} \sqrt{ \log \left( \frac{C_{A} \exp (- \NN_{(x,t)}(t-s))}{(t-s)^{n/2} G(x,t;\cdot,s)} \right) }. 
\end{equation*}
\end{prop}
\begin{proof}
We may assume $s = 0$ and $t = 1$.
Set $\nu:=\nu_{(x,1)}$.
Let $(z,1/2)$ be a center of $(x,1)$ (see Proposition \ref{prop:center}).
For a fixed $y\in M$,
we set $u := G(\cdot,1/2;y,0)$.
By Lemma \ref{lem:Nash cor_rem}, we have
\begin{equation*}
 - \NN_0 (\cdot , 1/2) \leq - \NN_0 (x,1) + C_{A} \left( d_{1/2} (z, \cdot) + 1 \right).
\end{equation*}
From Proposition \ref{prop:HK}, it follows that
\begin{equation*}
u  \leq C_{A} \exp ( - \NN_0 (\cdot, 1/2) ) \leq  C_{A} \exp ( - \NN_0 (x, 1) ) \exp \left( C_{A} d_{1/2} (z, \cdot) \right).
\end{equation*}
Using the co-area formula and Proposition \ref{prop:concentration2}, we obtain
\begin{align}\label{eq:HKK2}
 \int_M u^2 d\nu_{1/2} &\leq C_{A} \exp ( -2 \NN_0 (x, 1) ) \int_M  \exp \left( C_{A} d_{1/2} (z, \cdot) \right) d\nu_{1/2} \\ \notag
&\leq C_{A} \exp ( -2 \NN_0 (x, 1) ) \left( \int_0^\infty  e^{C_{A}r} \nu_{1/2}(M \setminus B(z,{1/2},r)) \, dr + 1 \right)\\ \notag
&\leq C_{A} \exp ( -2 \NN_0 (x, 1) ) \left( \int_0^\infty  e^{C_{A}r} \exp \left({- \frac{ (r - \sqrt{2\Cn} )_+^2}{8}  }\right) dr +1 \right) \\ \notag
&\leq C_{A} \exp ( -2 \NN_0 (x, 1) ).
\end{align}

By Proposition \ref{prop:HK} we have
\begin{equation*}
G(x,1;y,0) \leq \frac{C_{1,A}}{2} \exp ( -\NN_0 (x,1) ).
\end{equation*}
Put
\begin{equation*}
a :=\left( \frac{G(x,1;y,0)}{C_{1,A} \exp (-\NN_0 (x,1))}\right)^2. 
\end{equation*}
Let $b \geq 0$ satisfy the following property (cf. \cite[Claim 4.6]{B1}):
There is $\Omega \subset M$ with $\nu_1(\Omega)=a$ such that
\begin{equation*}
\left\{ \frac{|\nabla_x G|(x,1;\cdot, 1/2)}{G(x,1; \cdot, 1/2)} > b \right\} \subset \Omega \subset  \left\{ \frac{|\nabla_x G|(x,1;\cdot, 1/2)}{G(x,1; \cdot, 1/2)} \geq b \right\}.
\end{equation*}
By \eqref{eq:G_bound13} in Proposition \ref{prop:G_bound1} with $p=1$,
\begin{equation*}
a b \leq \int_{ \Omega }  \frac{|\nabla_x G|(x,1;\cdot, 1/2)}{G(x,1; \cdot, 1/2)} d\nu_{1/2} \leq C\nu_1(\Omega) (-\log \nu_1(\Omega))^{1/2} =  C a (-\log a)^{1/2} ,
\end{equation*}
and hence $b \leq C (-\log a)^{1/2}$.
By \eqref{eq:semigroup}, \eqref{eq:HKK2}, and \eqref{eq:G_bound13} with $p=2$,
we conclude
\begin{align*}
|\nabla_x G|(x,1;y,0) 
&\leq \int_M \frac{|\nabla_x G|(x,1;\cdot, 1/2)}{G(x,1; \cdot, 1/2)} u \, d\nu_{1/2} \\
&= \int_{\Omega} \frac{|\nabla_x G|(x,1;\cdot, 1/2)}{G(x,1; \cdot, 1/2)} u \, d\nu_{1/2} + \int_{M \setminus \Omega} \frac{|\nabla_x G|(x,1;\cdot, 1/2)}{G(x,1; \cdot, 1/2)} u\, d\nu_{1/2} \\
&\leq \left( \int_{\Omega} \left( \frac{|\nabla_x G|(x,1;\cdot, 1/2)}{G(x,1; \cdot, 1/2)} \right)^{2} d\nu_1 \right)^{1/2} \left( \int_{\Omega} u^2 d\nu_{1/2} \right)^{1/2} + b \int_M  u \, d\nu_{1/2} \\
&\leq C_{A}\exp(-\NN_0 (x,1)) \left( {- a \log a } \right)^{1/2}  + C(-\log a)^{1/2} G(x,1;y,0) \\
&\leq C_{A} (- \log a)^{1/2} G(x,1;y,0). 
\end{align*}
We arrive at the desired estimate.
\end{proof}

\subsection{Upper volume estimates}

We further present the following (cf. \cite[Theorem 8.1]{B1}):
\begin{prop}\label{prop:upper volume}
Let $(M,(g_t)_{t\in I})$ be a super Ricci flow with $\mathcal{D}\geq 0$.
Let $(x_0,t_0)\in M\times I$.
For $r>0$,
we assume $[t_0 -r^2, t_0] \subset I$.
For $A>0$,
we assume $\HH(\cdot, t_0-r^2) \geq -Ar^{-2}$.
Then for every $R \geq 1$ we have
\begin{equation*}
m_{t_0}(B(x_0,t_0,Rr)) \leq C_{A} \exp ( \NN_{(x_0,t_0)}(r^2))  \exp ( C_{A}  R^2 ) r^n.
\end{equation*}
\end{prop}
\begin{proof}
We may assume $t_0=0$ and $r =1$.
By \eqref{eq:potential} we see
\begin{equation*}
\NN_{(x_0,0)}(1) =  - \int_M \left( \log G(x_0,0;\cdot, -1) \right) G(x_0,0; \cdot, -1) dm_{-1}  - \frac{n}2 \log 4\pi - \frac{n}2.
\end{equation*}
Hence,
there exists $y \in M$ such that
\begin{equation*}
\log G(x_0,0;y, -1) \geq - \NN_{(x_0,0)}(1) - \frac{n}2 \log 4\pi- \frac{n}2;
\end{equation*}
in particular,
\begin{equation}\label{eq:upper volume3}
G(x_0,0; y, -1) \geq C \exp ( - \NN_{(x_0,0)}(1)  ). 
\end{equation}
Set $u := G(\cdot, 0; y;-1)$.
By Lemma \ref{lem:Nash thm} we have
\begin{equation}\label{eq:upper volume2}
 - \NN_{(\cdot, 0)} (1)=-\NN_{-1}(\cdot,0)\leq -\NN_{-1}(x_0,0)+\left( \frac{n}{2}+A  \right)^{1/2}d_0(x_0,\cdot) \leq - \NN_{(x_0,0)} (1) + C_{1,A} R
\end{equation}
on $B(x_0,0,R)$.
By Proposition \ref{prop:HKK} and \eqref{eq:upper volume2},
\begin{equation}\label{eq:upper volume2.5}
\frac{|\nabla u|}{u} \leq C_{A} \sqrt{ \log \left( \frac{C_{1,A} \exp ( - \NN_{(\cdot,0)}(1))}{u} \right)}\leq C_{A} \sqrt{ \log \left( \frac{C_{1,A} \exp ( - \NN_{(x_0,0)}(1) + C_{1,A} R)}{u} \right)}
\end{equation}
on $B(x_0,0,R)$.
We define
\begin{equation*}
v := \sqrt{ \log \left( \frac{C_{1,A} \exp ( - \NN_{(x_0,0)}(1) + C_{1,A} R)}{u} \right)}.
\end{equation*}
By \eqref{eq:upper volume3} and \eqref{eq:upper volume2.5},
we obtain
\begin{equation*}
v(x_0) \leq C_{A}\sqrt{R},\quad |\nabla v| \leq C_{A}
\end{equation*}
on $B(x_0,0,R)$.
This implies $v \leq C_{A}\sqrt{R} + C_{A} R \leq C_{A}R$ on $B(x_0,0,R)$, and hence
\begin{equation}\label{eq:upper volume5}
 u \geq C_{A}  \exp (- C_{A} R^2) \exp ( - \NN_{(x_0,0)}(1))
\end{equation}
on $B(x_0,0,R)$.
Since
\begin{equation*}
\frac{d}{dt} \int_M G(\cdot,t;y, -1) dm_t = - \int_M \HH \, G(\cdot,t;y, -1) dm_t \leq A\int_M G(\cdot,t;y, -1) dm_t,
\end{equation*}
the inequality \eqref{eq:upper volume5} leads us to
\begin{equation*} 
C_{A} \exp (- C_{A} R^2) \exp ( - \NN_{(x_0,0)}(1)) m_0(B(x_0,0,R)) \leq \int_{B(x_0,0,R)} u \, dm_0 \leq e^{A}. 
\end{equation*}
This completes the proof.
\end{proof}

\section{Non-collapsed case}\label{sec:noncollapse}
Here,
we show some estimates under a non-collapsed condition for the Nash entropy.

\subsection{Heat kernel measure comparison}

We first produce the following heat kernel measure comparison between different base points (cf. \cite[Proposition 8.1]{B3}):
\begin{prop}\label{prop:inheriting}
For $A>0$,
there is $\mathfrak{C}_{A}>0$ such that the following holds:
Let $(M,(g_t)_{t\in I})$ be a super Ricci flow with $\mathcal{D}\geq 0$.
Let $(x_0,t_0),(x_1,t_1)\in M\times I$,
and let $s, t \in I$ satisfy $s < t \leq \min\{t_0,t_1\}$.
For $\kappa>0$, we assume $\NN_{(x_0, t_0)}(t_0 - s) \geq - \kappa$.
For $D>0$ and $\theta_0,\theta_1 \in (-\infty,1)$ with $\theta_0<\theta_1$,
we assume
\begin{align*}
&\HH (\cdot, s) \geq - A (t - s)^{-1},\quad W_1(\nu_{(x_0, t_0);t}, \nu_{(x_1,t_1);t} ) \leq D \sqrt{t - s},\\ \notag
&t_0 - t \leq \mathfrak{C}_{A} \, \frac{\theta_1-\theta_0}{1-\theta_1}( t - s),\quad   t_1 - t \leq D^2 (t - s). 
\end{align*}
Then we have
\begin{equation*}
e^{\theta_0 f_0}\nu_{(x_0,t_0); s} \leq C_{\kappa,D,\theta_0,\theta_1, A}\, e^{\theta_1 f_1} \nu_{(x_1, t_1); s}.
\end{equation*}
\end{prop}
\begin{proof}
We may assume $s = 0$ and $t = 1$.
For $i =0,1$,
let $\nu^i := \nu_{(x_i, t_i)}$.
In view of Lemma \ref{lem:scal},
\begin{equation}\label{eq:inheriting1}
\HH \geq - A,\quad W_1(\nu^0_{1}, \nu^1_{1} ) \leq D,\quad t_0-1 \leq \mathfrak{C}_{A} \frac{\theta_1-\theta_0}{1-\theta_1},\quad  t_1-1 \leq D^2
\end{equation}
on $M \times (I \cap [0, \infty))$ (see also Remark \ref{rem:scal2}).
For a fixed $y \in M$,
we set $u := G(\cdot, 1; y,0)$.
By Propositions \ref{prop:HK} and \ref{prop:HKK},
we possess
\begin{equation}\label{eq:inheriting51}
 u \leq \frac{C_{1,A}}{10} \exp ( - \NN_0 (\cdot,1) ), \quad \frac{|\nabla u|}{u}  \leq C_{A} \sqrt{ \log \left( \frac{C_{1,A} \exp (- \NN_0 (\cdot, 1))}{ u} \right)}.
\end{equation}
We further set $v :=C^{-1}_{1,A} u \, \exp ( \NN_0 (\cdot, 1))$.
Note that
\eqref{eq:Nash thm gradient} and \eqref{eq:inheriting51} lead to
\begin{equation*}\label{eq:inheriting6}
 \frac{|\nabla v|}{v} \leq \frac{|\nabla u|}{u} + |\nabla \NN_0 (\cdot, 1)|\leq C_{A} \sqrt{ - \log v} + C_{A} \leq C_{A} \sqrt{- \log v};
\end{equation*}
in particular,
for $\varphi:=\sqrt{- \log v}$,
we have
\begin{equation}\label{eq:inheriting7}
 |\nabla \varphi| \leq C_{2,A}.
\end{equation}
We define $\mathfrak{C}_{A}:=(8C_{2,A})^{-2}$.
For $i =0,1$,
let $(z_i,1)$ be a center of $(x_i, t_i)$ (Proposition \ref{prop:center}).
Lemma \ref{lem:W-V} and \eqref{eq:inheriting1} imply
\begin{align}\label{eq:inheriting3}
 d_1 (z_0, z_1) &\leq W_1(\delta_{z_0}, \nu^0_1) + W_1(\nu^0_1, \nu^1_1) + W_1(\nu^1_1, \delta_{z_1})\\ \notag
&\leq \sqrt{\Cn (t_0 - 1)} + D + \sqrt{\Cn (t_1 - 1)} \leq C_{D}. 
\end{align}
Also,
due to Lemma \ref{lem:concentration1},
we possess
\begin{equation}\label{eq:inheriting12}
\nu^1_1 ( B) \geq \frac{1}{2},
\end{equation}
where $B:= B(z_1, 1, \sqrt{2\Cn (t_1 - 1)})$.

We now write $\lambda :=(1-\theta_1)/(1-\theta_0)$.
For a fixed $x \in B$,
\eqref{eq:inheriting7} and \eqref{eq:inheriting3} tell us that
\begin{align*}
 \varphi(\cdot)&\geq \varphi(x)  - C_{2,A} d_1 (x,\cdot)\\
 &\geq   \varphi(x) - C_{2,A} d_1 (z_0,\cdot) - C_{2,A} d_1 (z_0, z_1) - C_{2,A} \sqrt{2 \Cn (t_1 - 1)} \\
&\geq \varphi(x) - C_{2,A} (d_1 (z_0,\cdot)+C_{D}) - C_{D,A}
\end{align*}
on $M$,
and we obtain
\begin{align}\label{eq:inheriting10}
 &\quad\,\, \int_M v  \exp ( C_{1,A} d_1 (z_0, \cdot) ) \, d\nu^0_1\\ \notag
&= \sum_{j=1}^\infty \int_{B(z_0, 1,  j) \setminus B(z_0, 1,  j-1)} (\phi \circ \varphi)  \exp ( C_{1,A} d_1 (z_0, \cdot) ) \, d\nu^0_1 \\ \notag
&\leq \sum_{j=1}^\infty \phi \left( (\varphi(x) - C_{2,A} j - C_{D,A})_+ \right) \exp (C_{1,A} j) \nu^0_1 ( M \setminus B(z_0, 1, j-1)),
\end{align}
where $\phi(\zeta):=\exp (- \zeta^2)$.
By Proposition \ref{prop:concentration2}, \eqref{eq:inheriting1} and the Young inequality,
for every $r>0$,
\begin{align}\label{eq:inheriting100}
\nu^0_1 \left( M \setminus B(z_0, 1, r) \right) &\leq 2 \exp \left( - \frac{\left(r - \sqrt{2 \Cn (t_0 - 1)} \right)_+^2}{8(t_0 - 1)}   \right) \\ \notag
&\leq C \exp \left( - \frac{r^2}{16(t_0 - 1)}  \right)\leq C \exp \left({ - \frac{C^2_{2,A} \la}{1-\la}  r^2 } \right).
\end{align}
By \eqref{eq:inheriting10}, \eqref{eq:inheriting100}, 
the same calculation as in the proof of \cite[Proposition 8.1]{B3} tells us that
\begin{equation}\label{eq:inheriting9}
 \int_M v \exp ( C_{1,A} d_1 (z_0, \cdot) )  \, d\nu^0_1 \leq C_{D,\theta_0,\theta_1,A} \, v(x)^\la.
\end{equation}

Combining \eqref{eq:inheriting12} and \eqref{eq:inheriting9},
we obtain
\begin{align*}
 \int_M v \exp ( C_{1,A} d_1 (z_0, \cdot) )  \, d\nu^0_1 &\leq \frac{C_{D,\theta_0,\theta_1,A}}{\nu^1_1(B)} \int_{B}\, v^\la d\nu^{1}_1\\
 &\leq C_{D,\theta_0,\theta_1,A} \int_{M}\, v^\la d\nu^{1}_1\leq C_{D,\theta_0,\theta_1,A} \left(\int_{M}\, v d\nu^{1}_1 \right)^{\lambda}.
\end{align*}
By Lemma \ref{lem:Nash cor_rem},
we see
\begin{align*}\label{eq:inheriting2}
0 \leq -\NN_0 (\cdot,1) &\leq -\NN_0 (x_0,t_0)+\left( \frac{n}{2} + A\right)^{1/2}   \sqrt{H_n(t_0-1)}+\left( \frac{n}{2} + A\right)^{1/2}d_1(z_0,\cdot)\\
&\leq C_{\kappa,\theta_0,\theta_1,A}(d_1(z_0,\cdot)+1).
\end{align*}
It follows that
\begin{equation*}
 \int_M v \exp ( - \NN_0 (\cdot, 1) ) \, d\nu^0_1 \leq C_{\kappa,D,\theta_0,\theta_1, A} \left( \int_M v \exp ( - \NN_0 (\cdot, 1) )  \, d\nu^1_1 \right)^{\lambda};
\end{equation*}
in particular, \eqref{eq:semigroup} leads us to
\begin{equation*}
\int_M u \, d\nu^0_1 \leq C_{\kappa,D,\theta_0,\theta_1, A} \left( \int_M u \, d\nu^1_1 \right)^{\lambda},\quad G (x_0, t_0; y,0) \leq C_{\kappa,D,\theta_0,\theta_1, A} \left( G(x_1, t_1; y,0) \right)^{\lambda}.
\end{equation*}
Thus,
we complete the proof.
\end{proof}

\subsection{Integral estimates}\label{sec:Integral bounds}
In this subsection,
we deduce an integral estimate.
We first notice that
densities are bounded from below under a non-collapsed condition.
\begin{lem}\label{lem:lower poteintial}
Let $(M,(g_t)_{t\in I})$ be a super Ricci flow with $\mathcal{D}\geq 0$.
Let $(x_0,t_0)\in M\times I$.
For $\alpha\in (0,1)$ and $A>0$,
we assume $\HH(\cdot,t_0- \alpha^{-1} r^2)\geq -A\,r^{-2}$.
For $\kappa>0$,
we also assume $\NN_{(x_0,t_0)} (r^2) \geq - \kappa$.
Then on $M\times [t_0-\alpha^{-1} r^2,t_0)$,
we have
\begin{equation*}
f\geq -C_{\kappa,\alpha,A},\quad f^2 \leq C_{\kappa,\alpha,A} + C \tau^2 ( |\nabla^2 f|^2 + |\nabla f|^4 + |h|^2). 
\end{equation*}
\end{lem}
\begin{proof}
We may assume $t_0=0$ and $r = 1$.
In view of Lemma \ref{lem:scal},
we possess $\HH(\cdot,t)\geq -(A\alpha^{-1})\tau^{-1}$ for $t\in [-\alpha^{-1},0)$.
Hence,
by Proposition \ref{prop:HK},
we have
\begin{equation*}
G(x_0,0;\cdot,t) \leq \frac{C_{\alpha,A}}{\tau^{n/2}} \exp ( - \NN_{(x_0,0)}(\tau) ).
\end{equation*}
Therefore,
from \eqref{eq:potential} we conclude
\begin{equation*}
f(\cdot,t)=-\log G(x_0,0;\cdot,t)-\frac{n}{2}\log \tau-\frac{n}{2}\log 4\pi\geq -\log C_{A}+\NN_{(x_0,0)}(\tau)-\frac{n}{2}\log 4\pi.
\end{equation*}
By Lemma \ref{lem:Nash_more} and \eqref{eq:Nash3.5}, it holds that $\NN_{(x_0,0)}(\tau)\geq -C_{\kappa,\alpha,A}$;
in particular,
the lower bound of $f$ follows.
Due to Theorem \ref{thm:Harnack},
we obtain
\begin{equation*}
 - C_{\kappa,\alpha,A} \leq f = w - \tau (2 \Delta f - |\nabla f|^2 + \HH) +n\leq   - \tau (2 \Delta f - |\nabla f|^2 + \HH) +n, 
\end{equation*}
where $w$ is defined as \eqref{eq:Harnack notation}.
Therefore,
\begin{equation*}\label{eq:integral15}
 f^2 \leq C_{\kappa,\alpha,A} + C \tau^2 ( |\nabla^2 f|^2 + |\nabla f|^4 + \HH^2)\leq C_{\kappa,\alpha,A} + C \tau^2 ( |\nabla^2 f|^2 + |\nabla f|^4 + |h|^2). 
\end{equation*}
We complete the proof.
\end{proof}

We also verify the following (cf. \cite[Proposition 6.5]{B3}):
\begin{lem}\label{lem:integral2}
Let $(M,(g_t)_{t\in I})$ satisfy $\mathcal{D}\geq 0$.
Let $(x_0,t_0)\in M\times I$.
For $A>0$,
we assume $\HH \geq -A$.
Then for all $t \in I \cap (-\infty, t_0)$ and $\theta \in [0, 1/2]$ we have
\begin{equation*}
\int_M e^{\theta f} d\nu_{(x_0,t_0);t} \leq e^{(n+\tau A) \theta}.
\end{equation*}
\end{lem}
\begin{proof}
Theorem \ref{thm:Harnack} yields
\begin{align*}
 \frac{d}{d\theta} \int_M e^{\theta f} d\nu_{(x_0,t_0);t} &= \int_M f e^{\theta f} d\nu_{(x_0,t_0);t}\\
 &\leq  \int_M \left( \tau (-2\Delta f + |\nabla f|^2 - \HH ) + n \right) e^{\theta f} d\nu_{(x_0,t_0);t} \\
&\leq  \int_M \left(  \tau (-2\Delta f + |\nabla f|^2  ) + n + \tau A  \right) (4\pi \tau)^{-n/2} e^{-(1-\theta) f} dm_t \\
&=  \int_M \left(  \tau (2\theta-1) |\nabla f|^2   + n + \tau A  \right) (4\pi \tau)^{-n/2} e^{-(1-\theta) f} dm_t  \\
&\leq (n+\tau A)  \int_M e^{\theta f} d\nu_{(x_0,t_0);t}.
\end{align*}
Integrating this over $\theta$ implies the lemma.
\end{proof}

We further see the following:
\begin{lem}\label{lem:integral3}
Let $(x_0,t_0)\in M\times I$.
Then for all $t \in I \cap (-\infty, t_0)$ and $\theta \in [0,1/4]$ we have
\begin{equation*}\label{eq:integral14}
 \int_M |\nabla f|^4 e^{\theta f} d\nu_{(x_0,t_0);t} \leq C  \int_M | \nabla^2 f |^2  e^{\theta f} d\nu_{(x_0,t_0);t}. 
\end{equation*}
\end{lem}
\begin{proof}
We have
\begin{align*}
(1-\theta) \int_M |\nabla f|^4 e^{\theta f} d\nu_{(x_0,t_0);t} 
&=(1-\theta) (4\pi \tau)^{-n/2} \int_M |\nabla f|^2 \langle \nabla f, \nabla f e^{-(1-\theta ) f} \rangle dm_t \\
&= (4\pi \tau)^{-n/2} \int_M \left(  2 \langle \nabla^2 f, d f \otimes d f\rangle  + |\nabla f|^2 \Delta f \right) e^{-(1-\theta ) f} dm_t \\
&\leq C (4\pi \tau)^{-n/2} \int_M | \nabla^2 f | |\nabla f|^2  e^{-(1-\theta ) f} dm_t \\
&\leq C  \int_M | \nabla^2 f |^2  e^{\theta f} d\nu_{(x_0,t_0);t} +\frac1{2}  \int_M  |\nabla f|^4  e^{\theta f} d\nu_{(x_0,t_0);t}.
\end{align*}
This yields the desired one.
\end{proof}

Based on the above lemmas,
we prove the following (cf. \cite[Proposition 6.2]{B3}):
\begin{prop}\label{prop:integral1}
If $\theta \in [0, \ov\theta]$,
then the following holds:
Let $(M,(g_t)_{t\in I})$ be a super Ricci flow with $\mathcal{D}\geq 0$.
Let $(x_0,t_0)\in M\times I$.
For $\alpha \in (0,1)$ and $A>0$,
we assume $\HH(\cdot,t_0-\alpha^{-1} r^2)\geq -A\,r^{-2}$.
For $\kappa>0$,
we assume $\NN_{(x_0,t_0)} (r^2) \geq - \kappa$.
Then we have
\begin{equation*}\label{eq:integral11}
 \int_{t_0-\alpha^{-1} r^2}^{t_0-\alpha r^2} \int_M \left(\tau( |h|^2 + |\nabla^2 f|^2+  |\nabla f|^4) +\tau^{-1} f^2 \right)  e^{2\theta f} d\nu_{(x_0,t_0);t} dt \leq C_{\kappa,\alpha,A}.
\end{equation*}
\end{prop}
\begin{proof}
We may assume $t_0=0$ and $r = 1$. 
Set $\nu:=\nu_{(x_0,0)}$ and $u:=G(x_0,0;\cdot,\cdot)$.
We define $\Psi$ and $w$ as \eqref{eq:Harnack notation}.
In virtue of \eqref{eq:conjugate integration}, \eqref{eq:potential_enjoy} and Theorem \ref{thm:Harnack},
it holds that
\begin{align*}
&\quad\,\, \frac{d}{dt} \int_M w e^{\theta f} d\nu_t= \int_M \left\{ ( \square e^{\theta f}) w u - e^{\theta f} \square^* (wu) \right\} dm_t \\ \notag
&=  \int_M \left(2\tau  \left|\Psi \right|^2+\tau \mathcal{D}(\nabla f)-\theta  \left\{ \tau^{-1} \left( w-f+\frac{n}{2}   \right) + \theta |\nabla f|^2 \right\} w\right) e^{\theta f} d\nu_t \\ \notag
&\geq  \int_M \left(   2\tau  \left|\Psi \right|^2-\theta \tau^{-1} \left( w-f+\frac{n}{2}   \right)w   \right) e^{\theta f} d\nu_t \\ \notag
&\geq  \int_M   2\tau  \left|\Psi \right|^2  e^{\theta f} d\nu_t-\theta \tau^{-1} \int_M \left| w-f+\frac{n}{2}   \right| |w| e^{\theta f} d\nu_t \\ \notag
&\geq  \int_M   2\tau  \left|\Psi \right|^2  e^{\theta f} d\nu_t-C\theta \tau^{-1} \int_M (w^2+f^2+1) e^{\theta f} d\nu_t \\ \notag
&\geq  \int_M   2\tau  \left|\Psi \right|^2  e^{\theta f} d\nu_t-C\theta \tau^{-1} \int_M (\tau^2((\Delta f)^2+|\nabla f|^4+\HH^2)+f^2+1) e^{\theta f} d\nu_t\\ \notag
&\geq  \int_M   2\tau  \left|\Psi \right|^2  e^{\theta f} d\nu_t-C\theta \tau^{-1} \int_M (\tau^2(|\nabla^2 f|^2+|\nabla f|^4+|h|^2)+f^2+1) e^{\theta f} d\nu_t.
\end{align*}
Also, if $\theta \in [0,1]$,
then \eqref{eq:conjugate integration2} and \eqref{eq:potential_enjoy} yield
\begin{align*}
&\quad\,\, \frac{d}{dt} \int_M \tau \HH \,e^{\theta f} d\nu_t=\int_M \square(\tau \HH e^{\theta f}) \, d\nu_t\\
&=\int_M (2 \tau |h|^2+\tau \mathcal{D}(0) - \HH) e^{\theta f}d\nu_t-\theta\int_M  \HH\left(w-f+\frac{n}{2} \right)  e^{\theta f}d\nu_t\\
&\quad +\theta (\theta-2) \tau\int_M  \HH|\nabla f|^2 e^{\theta f}d\nu_t+2\theta \tau\int_M  \HH \Delta f  e^{\theta f}d\nu_t\\
&\geq \int_M (2 \tau |h|^2 - \HH) e^{\theta f}d\nu_t-C\theta \int_M  \left\{ \tau \HH^2+\tau^{-1}\left(w^2+f^2+1 \right) \right\} e^{\theta f}d\nu_t\\
&\quad -C|\theta-2| \theta \tau\int_M  (H^2 +|\nabla f|^4) e^{\theta f}d\nu_t-C\theta \tau\int_M ( \HH^2 +(\Delta f)^2 ) e^{\theta f}d\nu_t\\
&\geq \int_M (2 \tau |h|^2 - \HH) e^{\theta f}d\nu_t-C\theta \int_M  \left\{ \tau \HH^2+\tau^{-1}\left(w^2+f^2+1 \right) \right\} e^{\theta f}d\nu_t\\
&\quad -C|\theta-2| \theta \tau\int_M  (\HH^2 +|\nabla f|^4) e^{\theta f}d\nu_t-C\theta \tau\int_M ( \HH^2 +(\Delta f)^2 ) e^{\theta f}d\nu_t\\
&\geq \int_M (2 \tau |h|^2 - \HH) e^{\theta f}d\nu_t-C\theta \tau^{-1}  \int_M (\tau^2((\Delta f)^2+|\nabla f|^4+\HH^2)+f^2+1) e^{\theta f} d\nu_t\\
&\geq \int_M (2 \tau |h|^2 - \HH) e^{\theta f}d\nu_t-C\theta \tau^{-1}  \int_M (\tau^2(|\nabla^2 f|^2+|\nabla f|^4+|h|^2)+f^2+1) e^{\theta f} d\nu_t.
\end{align*}
Therefore,
if $\theta \in [0,1]$,
then
\begin{align}\label{eq:integral17}
\frac{d}{dt} \int_M (w+\tau \HH) e^{\theta f} d\nu_t &\geq   \int_M \left(2\tau  \left|\Psi \right|^2+2 \tau |h|^2 - \HH \right) e^{\theta f} d\nu_t\\ \notag
&\quad-C\theta \tau^{-1}  \int_M (\tau^2(|\nabla^2 f|^2+|\nabla f|^4+|h|^2)+f^2+1) e^{\theta f} d\nu_t.
\end{align}
Note that
\begin{align*}
&\left|\nabla^2 f-\frac{1}{2\tau}g\right|^2=|\Psi-h|^2\leq 4|\Psi|^2+2|h|^2,\\
&|\HH|\leq \sqrt{n}|h|\leq \frac{\tau}{2}|h|^2+\frac{n}{2\tau},\\
&|\nabla^2 f|^2\leq 2\left|\nabla^2 f-\frac{1}{2\tau}g\right|^2+\frac{n}{2\tau^2};
\end{align*}
in particular,
\begin{align}\label{eq:integral25}
2\tau  \left|\Psi \right|^2+2 \tau |h|^2 - \HH 
&\geq 2\tau \left( \frac{1}{4}\left|\nabla^2 f-\frac{n}{2\tau}\right|^2-\frac{1}{2}|h|^2  \right)+2 \tau |h|^2 - \left(\frac{\tau}{2}|h|^2+\frac{n}{2\tau}\right)\\ \notag
&\geq 2\tau \left(\frac{1}{8}|\nabla^2 f|^2-\frac{1}{2}|h|^2-\frac{n}{16 \tau^2}\right)+2\tau|h|^2-\left(\frac{\tau}{2}|h|^2+\frac{n}{2\tau}\right)\\ \notag
&= \frac{\tau}{4}|\nabla^2 f|^2+\frac{\tau}{2}|h|^2-\frac{5n}{8 \tau}.
\end{align}
Combining \eqref{eq:integral17}, \eqref{eq:integral25} and Lemmas \ref{lem:lower poteintial}, \ref{lem:integral3} yields that for $\theta \in  [0,1/4]$,
\begin{align*}
&\quad\,\, \frac{d}{dt} \int_M (w+\tau \HH) e^{\theta f} d\nu_t\\
&\geq  \int_M \left\{ \tau \left( \frac{1}{4}-C\theta \right)|\nabla^2 f|^2+\tau \left( \frac{1}{2}-C \theta  \right)|h|^2 -\tau^{-1}\left(  \frac{5n}{8}+C_{\kappa,\alpha,A}\theta  \right) \right\}e^{\theta f} d\nu_t.
\end{align*}
In particular,
Lemma \ref{lem:integral2} implies that for $\theta \in [0,\ov\theta]$,
\begin{align*}
\frac{d}{dt} \int_M (w+\tau \HH) e^{\theta f} d\nu_t&\geq   \frac{\tau}{8}\int_M  (|\nabla^2 f|^2+|h|^2) e^{\theta f} d\nu_t-C_{\kappa,\alpha,A} \tau^{-1}  \int_M e^{\theta f} d\nu_t\\
&\geq   \frac{\tau}{8}\int_M  (|\nabla^2 f|^2+|h|^2) e^{\theta f} d\nu_t-C_{\kappa,\alpha,A} \tau^{-1}  e^{(n+\tau\, A)\theta}.
\end{align*}
Once we obtain this estimate,
we can prove the desired one by the same cutoff argument on time as in the proof of \cite[Proposition 6.2]{B3} together with Lemmas \ref{lem:lower poteintial} and \ref{lem:integral3}.
\end{proof}

\section{Almost selfsimilar points}\label{sec:almost_ss}

For $\eps\in (0,1),r>0$,
we say that
a point $(x_0,t_0)\in M \times I$ is \textit{$(\eps, r)$-selfsimilar} if we have:
\begin{enumerate}\setlength{\itemsep}{+1.0mm}
\item $[t_0 - \eps^{-1} r^2, t_0 ] \subset I$;
\item we have
\begin{equation}\label{eq:almost selfsimilar1}
 \int_{t_0 - \eps^{-1} r^2}^{t_0 - \eps r^2} \int_M \tau \left( \left| h + \nabla^2 f - \frac1{2\tau} g \right|^2+\frac{1}{2}\mathcal{D}(\nabla f) \right) d\nu_{(x_0, t_0); t} dt \leq \eps;
\end{equation}
\item for all $t \in [t_0 - \eps^{-1} r^2, t_0 - \eps r^2]$, we have
\begin{equation}\label{eq:almost selfsimilar2}
\int_M \left| \tau (2\Delta f - |\nabla f|^2 + \HH) + f - n - \mathrm{N} \right| d\nu_{(x_0, t_0); t} \leq \eps
\end{equation}
for $\mathrm{N}:=\NN_{(x_0, t_0)} (r^2)$;
\item on $M \times [t_0 - \eps^{-1} r^2, t_0 - \eps r^2]$, we have
\begin{equation}\label{eq:almost selfsimilar3}
\HH \geq -\eps r^{-2}.
\end{equation}
\end{enumerate}
In this section,
we investigate various properties of almost selfsimilar points.

\begin{rem}\label{rem:example}
We discuss the validity of the formulation of the almost selfsimilarity by taking Ricci flow coupled with harmonic map heat flow as an example.
Let $(M,(g_t)_{t\in I},(\varphi_t)_{t\in I})$ be a compact manifold equipped with time-dependent Riemannian metric and map from $M$ to another Riemannian manifold.
Such a time-dependent manifold is called \textit{Ricci flow coupled with harmonic map heat flow} or \textit{M\"uller flow} when
\begin{equation}\label{eq:Muller}
\begin{cases}
  \partial_{t}g=-2\Ric+4 d\varphi\otimes d\varphi,\\
 \partial_t \varphi=\mathsf{t}( \varphi),
                                                   \end{cases}
\end{equation}
which has been introduced by M\"uller \cite{M2}.
Here $\mathsf{t}( \varphi)$ is the tension field of $\varphi$.
In the special case where $(\varphi_t)_{t\in I}$ is function,
\eqref{eq:Muller} is called \textit{Ricci flow coupled with heat equation} or \textit{List flow} (\cite{L}).
For this flow,
we have the following (see e.g., \cite[Section 2]{M1}):
\begin{equation}\label{eq:List property}
h=\Ric-2d\varphi\otimes d\varphi,\quad \mathcal{D}(V)=4\left|\mathsf{t}(\varphi)-d\varphi(V) \right|^2.
\end{equation}
Also,
the soliton identity on $(M,g,\varphi)$, which characterizes selfsimilar solutions of \eqref{eq:Muller}, is given as follows (see \cite[Lemma 2.2]{M2}):
\begin{equation}\label{eq:soliton}
\begin{cases}
  (\Ric-2d\varphi\otimes d\varphi)+\nabla^2 f=Kg,\\
 \mathsf{t}(\varphi)-d\varphi(\nabla f)=0 
                                                   \end{cases}
\end{equation}
for some $f\in C^{\infty}(M)$ and $K\in \mathbb{R}$.
The first equation in \eqref{eq:soliton} corresponds to the selfsimilarity of metric $(g_t)_{t\in I}$,
and the second one does that of map $(\varphi_t)_{t\in I}$.
In view of \eqref{eq:List property} and \eqref{eq:soliton},
for the flow \eqref{eq:Muller},
the first term of the integrand in \eqref{eq:almost selfsimilar1} expresses the almost selfsimilarity of metric,
and the second term does that of map (see also \cite[(2.5)--(2.9)]{M2}, \cite[Subsection 4.7]{B3} for \eqref{eq:almost selfsimilar2}).
\end{rem}

\subsection{Characterization}\label{sec:Integral estimates}
In this subsection,
we prove that
the almost selfsimilarity can be characterized by the almost constancy of the Nash entropy.
We first prepare the following lemma (cf. \cite[Lemma 7.10]{B3}):
\begin{lem}\label{lem:integral soliton}
For $\kappa>0,\eps \in (0,1)$,
if $\zeta \leq \ov\zeta_{\kappa, \eps}$, then the following holds:
Let $(M,(g_t)_{t\in I})$ be a super Ricci flow with $\mathcal{D}\geq 0$.
Let $(x_0,t_0) \in M \times I$.
For $r>0$,
we assume $[t_0 - \zeta^{-1} r^2, t_0] \subset I$.
We assume $\NN_{(x_0, t_0)} (r^2) \geq - \kappa$.
If
\begin{equation}\label{eq:integral soliton_ass}
| \WW_{(x_0, t_0)} (\tau) - \mathrm{N} | \leq \zeta
\end{equation}
for all $\tau \in [\zeta r^2, \zeta^{-1} r^2]$,
then
\begin{equation*}
\int_M \left| w - \mathrm{N} \right| d\nu_{(x_0,t_0);t} \leq \eps
\end{equation*}
for all $t\in [t_0-\eps^{-1}r^2,t_0-\eps^{-1} r^2]$, where $w$ is defined as \eqref{eq:Harnack notation} and $\mathrm{N}:=\NN_{(x_0,t_0)} (r^2)$.
\end{lem}
\begin{proof}
We may assume $t_0=0$ and $r = 1$.
We set $\nu:=\nu_{(x_0,0)}$ and $u :=G(x_0,0;\cdot,\cdot)$.
Let $v \in C^\infty (M \times [-\zeta^{-1}, 0])$ be a solution to the heat equation with $|v(\cdot, -\zeta^{-1}) | \leq 1$.
In view of the maximum principle,
we see $|v| \leq 1$ on $M \times [-\zeta^{-1}, 0]$.
Due to \eqref{eq:conjugate integration} and Theorem \ref{thm:Harnack},
\begin{align*}
 \frac{d}{dt} \int_M v (w - \mathrm{N} ) u \, dm_t&=  -\int_M v \square^* (w u) dm_t \\
&\geq \int_M  \square^* (w u) dm_t = -\frac{d}{dt} \int_M w u \, dm_t = \frac{d}{dt} \WW_{(x_0, 0)} (|t|).
\end{align*}
From \eqref{eq:integral soliton_ass},
it follows that
\begin{align*}
 \int_M  v (w - \mathrm{N} )\, d\nu_{-\tau_1} &\leq \int_M v (w - \mathrm{N} ) \, d\nu_{-\tau_2} - \WW_{(x_0, 0)} (\tau_2) +  \WW_{(x_0, 0)} (\tau_1)\\ \notag
& \leq  \int_M v (w - \mathrm{N} ) d\nu_{-\tau_2} + 2\zeta
\end{align*}
for every $\tau_1,\tau_2 \in [\zeta, \zeta^{-1}]$,
and hence
\begin{equation}\label{eq:integral soliton2}
\left|  \int_M  v (w - \mathrm{N} )\, d\nu_{-\tau_1} - \int_M  v (w - \mathrm{N} )\, d\nu_{-\tau_2}    \right|\leq 2\zeta.
\end{equation}

From Lemma \ref{lem:lower poteintial} and Proposition \ref{prop:integral1},
we deduce
\begin{equation*}
\int_{-2\zeta}^{-\zeta} \int_M \left( \tau^2 ( |h|^2+|\nabla^2 f|^2 +  |\nabla f|^4) + f^2 \right) d\nu_t dt \leq C_\kappa \zeta;
\end{equation*}
in particular,
there exists $\tau_0 \in [\zeta, 2\zeta]$ such that
\begin{equation*}
\int_M \left( \tau_0^2 ( |h|^2+|\nabla^2 f|^2 +  |\nabla f|^4 ) + f^2 \right) d\nu_{-\tau_0} \leq C_\kappa.
\end{equation*}
It follows that
\begin{equation}\label{eq:integral soliton4}
 \int_{M} (w-\mathrm{N} )^2 d\nu_{-\tau_0}\leq C_\kappa + \int_{M} \left( \tau_0^2 ( |\nabla^2 f|^2 +  |\nabla f|^4 +  \HH^2 )+ f^2 \right) d\nu_{-\tau_0} \leq C_\kappa. 
 \end{equation}
Thanks to Proposition \ref{prop:Poincare} with $p=2$ and Theorem \ref{thm:gradient_estimate},
we also possess
\begin{equation}\label{eq:integral soliton3}
\int_M |v - a|^2 \, d\nu_{-\tau_0} \leq 2 \tau_0 \int_M |\nabla v|^2 d\nu_{-\tau_0} \leq C \tau_0 \leq C \zeta,
\end{equation}
where $a\in [-1,1]$ is defined as $a := \int_M v \, d\nu_{-\tau_0}$.
 Therefore, \eqref{eq:integral soliton3} and \eqref{eq:integral soliton4} lead us to
 \begin{align*}
\left| \int_{M} v (w-\mathrm{N} ) d\nu_{-\tau_0} \right| &\leq \left| a \int_M (w-\mathrm{N} ) d\nu_{-\tau_0} \right|  + \int_{M} |v - a | |w-\mathrm{N} |   d\nu_{-\tau_0} \\
&\leq  \left| \int_M w\, d\nu_{-\tau_0} - \mathrm{N}  \right| + \left( \int_{M} |v - a|^2  d\nu_{-\tau_0} \right)^{1/2} \left( \int_{M} |w-\mathrm{N} |^2   d\nu_{-\tau_0} \right)^{1/2} \\
&\leq | \WW_{(x_0, 0)} (\tau_0) - \mathrm{N}  | + C_\kappa \,\zeta^{1/2}
\leq \zeta + C_\kappa \zeta^{1/2}.
\end{align*}
Combining this with \eqref{eq:integral soliton2} tells us that
if $\zeta \leq \ov\zeta_{\kappa,\eps}$,
then for all $\tau \in [\eps,\eps^{-1}]$ we have
\begin{equation*}
 \left|\int_M  v (w - \mathrm{N}) \,d\nu_{-\tau} \right| \leq \eps.
\end{equation*}
Since $v$ is arbitrarily,
we complete the proof.
\end{proof}

We provide the following characterization (cf. \cite[Proposition 7.1]{B3}):
\begin{prop}\label{prop:Nash almost}
For $\kappa>0,\eps \in (0,1)$,
if $\delta \leq \ov\delta_{\kappa, \eps}$, then the following holds:
Let $(M,(g_t)_{t\in I})$ be a super Ricci flow with $\mathcal{D}\geq 0$.
Let $(x_0, t_0) \in M \times I$.
For $r > 0$, we assume $[t_0 - \delta^{-1} r^2, t_0] \subset I$ and $\NN_{(x_0, t_0)} ( r^2) \geq - \kappa$.
If
\begin{equation}\label{eq:Nash almost first}
\NN_{(x_0, t_0)} (\delta^{-1} r^2) \geq \NN_{(x_0, t_0)} (\delta r^2) - \delta,
\end{equation}
then $(x_0, t_0)$ is $(\eps, r)$-selfsimilar.
Vice versa, if $(x_0, t_0)$ is $(\delta, r)$-selfsimilar, then for all $\tau_1, \tau_2 \in [ \eps r^2,  \eps^{-1} r^2]$ we have
\begin{equation}\label{eq:Nash almost}
 | \NN_{(x_0, t_0)} (\tau_1) - \NN_{(x_0, t_0)} (\tau_2) | \leq \eps. 
\end{equation}
\end{prop}
\begin{proof}
We may assume $t_0 = 0$ and $r = 1$.
Set $\mathrm{N}:=\NN_{(x_0,0)} (1)$.
We first assume \eqref{eq:Nash almost first}.
By \eqref{eq:Nash3.5}, \eqref{eq:Nash rem} and \eqref{eq:Nash almost first},
\begin{equation}\label{eq:Nash almost1}
|\NN_{(x_0, 0)} (\tau) - \mathrm{N}| \leq \delta,\quad  \WW_{(x_0, 0)} (\tau)  \leq \NN_{(x_0, 0)}(\tau) \leq \mathrm{N} + \delta
\end{equation}
for all $\tau \in [\delta,\delta^{-1}]$.
Let $\zeta \in ( \delta, 1)$.
By \eqref{eq:Nash3}, \eqref{eq:Nash2} and \eqref{eq:Nash almost1},
if $\delta \leq \ov\delta_{\kappa,\zeta}$,
then 
\begin{align}\label{eq:Nash almost2}
 \WW_{(x_0, 0)} (\tau) &\geq \frac1{\delta^{-1} - \tau} \int^{\delta^{-1}}_{\tau} \WW_{(x_0, 0)} (\sigma) d\sigma=  \frac1{\delta^{-1} - \tau} \left( \delta^{-1} \NN_{(x_0, 0)} (\delta^{-1}) - \tau \NN_{(x_0, 0)} (\tau) \right) \\ \notag
&\geq  \frac{\delta^{-1}}{\delta^{-1} - \tau}  \NN_{(x_0, 0)} (\delta^{-1})\geq \frac{\delta^{-1}}{\delta^{-1} - \zeta^{-1}}  (\mathrm{N} - \delta)\geq \mathrm{N} - \zeta
\end{align}
for all $\tau \in [\zeta,\zeta^{-1}]$;
in particular, $| \WW_{(x_0, 0)} (\tau)-\mathrm{N}|\leq \zeta$ for all $\tau \in [\zeta,\zeta^{-1}]$.
By Lemma \ref{lem:integral soliton},
if $\zeta \leq \ov\zeta_{\kappa, \eps}$,
then we arrive at \eqref{eq:almost selfsimilar2}.
Secondly,
\eqref{eq:Nash3} yields
\begin{equation*}
\int_{ - \eps^{-1}}^{- \eps } \int_M \left(2\tau \left| \Psi \right|^2+\tau \mathcal{D}(\nabla f) \right)d\nu_{(x_0,0);t} dt= \WW_{(x_0, 0)} (\eps) -  \WW_{(x_0, 0)} (\eps^{-1}) \leq \eps,
\end{equation*}
where $\Psi$ is defined as \eqref{eq:Harnack notation}.
Hence, \eqref{eq:almost selfsimilar1}.
Finally,
\eqref{eq:almost selfsimilar3} is a consequence of Lemma \ref{lem:scal}.

We next show \eqref{eq:Nash almost}.
Assume that $(x_0, 0)$ is $(\delta, 1)$-selfsimilar.
Integrating \eqref{eq:almost selfsimilar2},
we have
\begin{equation}\label{eq:Nash almost4}
 |\WW_{(x_0, 0)} (\tau) - \mathrm{N} | \leq \delta
\end{equation}
for all $\tau \in [\delta, \delta^{-1}]$.
By \eqref{eq:Nash rem} and \eqref{eq:Nash almost4}, we see
\begin{equation*}
\NN_{(x_0, 0)} (\tau) \geq \WW_{(x_0, 0)} (\tau) \geq \mathrm{N} - \delta
\end{equation*}
for all $\tau \in [\delta, \delta^{-1}]$.
Furthermore,
\eqref{eq:Nash3}, \eqref{eq:Nash2} and \eqref{eq:Nash almost4},
\begin{align*}
\mathrm{N}+\delta \geq \WW_{(x_0, 0)} (\delta) &\geq \frac1{ \tau -\delta} \int_{\delta}^{\tau} \WW_{(x_0, 0)} (\sigma) d\sigma\\
&=  \frac{1}{\tau - \delta}  \left(\tau\NN_{(x_0, 0)} (\tau)-\delta \NN_{(x_0,0)}(\delta) \right)\geq \frac{\eps}{\eps - \delta}\NN_{(x_0, 0)} (\tau)
\end{align*}
for all $\tau \in [\eps,\eps^{-1}]$.
Therefore,
if $\delta \leq \ov\delta_{\kappa, \eps}$,
then $| \NN_{(x_0,0)} (\tau) - \mathrm{N} | \leq \eps/2$ for all $\tau \in [\eps,\eps^{-1}]$,
and this proves \eqref{eq:Nash almost}.
\end{proof}

\subsection{Improved selfsimilarity}\label{sec:Integral estimates}

We have the following (cf. \cite[Proposition 7.3]{B3}):
\begin{prop}\label{prop:almost soliton estimate}
For $\kappa>0,\eps \in (0,1)$,
if $\theta \in [0, \ov\theta],\delta \leq \ov\delta_{\kappa, \eps}$,
then the following holds:
Let $(M,(g_t)_{t\in I})$ be a super Ricci flow with $\mathcal{D}\geq 0$.
Let $(x_0, t_0) \in M \times I$.
For $r > 0$,
we assume that $(x_0, t_0)$ is $(\delta, r)$-selfsimilar, and $\NN_{(x_0, t_0)} ( r^2) \geq - \kappa$.
Then we have
\begin{align*}
& \int_{t_0 - \eps^{-1} r^2}^{t_0 - \eps r^2} \int_M \tau \left| h + \nabla^2 f - \frac1{2\tau} g \right|^2 e^{\theta f} d\nu_{(x_0,t_0);t} dt \leq \eps, \\ 
&r^{-2} \int_{t_0 - \eps^{-1} r^2}^{t_0 - \eps r^2} \int_M  \left| \tau ( - |\nabla f|^2 + \Delta f) + f - \frac{n}2 - \mathrm{N}  \right| e^{\theta f} d\nu_{(x_0,t_0);t} dt \leq \eps, \\ 
&r^{-2} \int_{t_0 - \eps^{-1} r^2}^{t_0 - \eps r^2} \int_M  \left| \square (\tau f) + \frac{n}2 + \mathrm{N}  \right| e^{\theta f} d\nu_{(x_0,t_0);t} dt \leq \eps,\\
&r^{-2} \int_{t_0 - \eps^{-1} r^2}^{t_0 - \eps r^2} \int_M  \left|{ - \tau ( |\nabla f|^2 +\HH) + f  - \mathrm{N} } \right| e^{\theta f} d\nu_{(x_0,t_0);t} dt \leq \eps,
\end{align*}
where $\mathrm{N}  := \NN_{(x_0,t_0)} (r^2)$.
\end{prop}
\begin{proof}
We can prove this assertion by using the calculation technique stated in the proof of \cite[Proposition 7.3]{B3} together with Proposition \ref{prop:integral1}.
\end{proof}

\subsection{Almost constancy}\label{sec:almost_monotone}

In this subsection,
we prove an almost constancy property for an integral quantity.
To do so,
let us show the following formula:
\begin{lem}\label{lem:formula}
Let $(x_0, t_0) \in M \times I$.
Then for all $t \in I \cap (-\infty, t_0)$ we have
\begin{align*}
\frac{d}{dt} \int_M \tau \HH \, d\nu_{(x_0,t_0);t} &=2 \int_M \tau \left\langle \Psi, h+\nabla^2 f- df \otimes df \right\rangle d\nu_{(x_0,t_0);t}\\
&\quad +2 \int_M  \tau (\tr \Psi)(|\nabla f|^2-\Delta f)  d\nu_{(x_0,t_0);t}+\int_M \tau \mathcal{D}(\nabla f)  d\nu_{(x_0,t_0);t},
\end{align*}
where $\Psi$ is defined as \eqref{eq:Harnack notation}.
\end{lem}
\begin{proof}
We set $\nu:=\nu_{(x_0,t_0)}$.
From \eqref{eq:conjugate integration2} we deduce
\begin{align}\label{eq:direct0}
\frac{d}{dt} \int_M \tau \HH \, d\nu_t&=\int_M \square(\tau \HH) \, d\nu_t=\int_M \tau\square \HH-\HH \, d\nu_t\\ \notag
&=\int_M (2 \tau |h|^2+\tau \mathcal{D}(0) - \HH) d\nu_t\\ \notag
&= \int_M \left( 2 \tau \left\langle \Psi,h  \right\rangle- 2 \tau \langle h, \nabla^2 f\rangle +\tau \mathcal{D}(0) \right) d\nu_t.
\end{align}
By direct calculations,
we see
\begin{equation}\label{eq:direct}
\Div (h(\nabla f))=(\Div h)(\nabla f)+\langle h,\nabla^2f \rangle,\quad \Div (\Psi(\nabla f))=(\Div \Psi)(\nabla f)+\langle \Psi,\nabla^2f \rangle.
\end{equation}
The first one in \eqref{eq:direct} and \eqref{eq:conjugate integration2} yield
\begin{align}\label{eq:direct1}
\int_M \langle h, \nabla^2 f\rangle  \, d\nu_t 
&= \int_M \left( h (\nabla f, \nabla f)- (\Div h)(\nabla f) \right) d\nu_t  \\ \notag
&=  \int_M \left( h (\nabla f, \nabla f)- \frac{1}{2}\langle \nabla \HH,\nabla f \rangle  \right) d\nu_t  \\ \notag
&\quad + \int_M \left( \frac{1}{2}(\Ric-h)(\nabla f,\nabla f)-\frac{1}{4}\mathcal{D}(\nabla f)+\frac{1}{4}\mathcal{D}(0) \right)  d\nu_t.
\end{align}
The second one in \eqref{eq:direct}, \eqref{eq:conjugate integration2} and the Bochner formula (see e.g., \cite[Lemma 2.1]{PW}) imply
\begin{align}\label{eq:direct2}
\int_M \langle \Psi, \nabla^2 f\rangle  \, d\nu_t &= \int_M \left( \Psi (\nabla f, \nabla f)- (\Div \Psi)(\nabla f) \right) d\nu_t  \\  \notag
&= \int_M \left( \langle \Psi,df \otimes df\rangle- (\Div h)(\nabla f)-(\Div \nabla^2 f)(\nabla f) \right) d\nu_t  \\  \notag
&= \int_M  \langle \Psi,df \otimes df \rangle d\nu_t  \\ \notag
&\quad + \int_M \left( - \frac{1}{2}\langle \nabla \HH,\nabla f \rangle+\frac{1}{2}(\Ric-h)(\nabla f,\nabla f)-\frac{1}{4}\mathcal{D}(\nabla f)+\frac{1}{4}\mathcal{D}(0) \right) d\nu_t  \\ \notag
&\quad + \int_M \left( -\langle \nabla \Delta f,\nabla f\rangle -\Ric(\nabla f,\nabla f) \right) d\nu_t.
\end{align}
Combining \eqref{eq:direct1} and \eqref{eq:direct2} tells us that
\begin{align}\label{eq:direct3}
&\quad\,\, \int_M \left(\langle h, \nabla^2 f\rangle+\langle \Psi, \nabla^2 f\rangle \right)  \, d\nu_t \\ \notag
&=\int_M \left(\langle \Psi,df \otimes df \rangle - \langle \nabla (\tr \Psi),\nabla f \rangle-\frac{1}{2}\mathcal{D}(\nabla f)+\frac{1}{2}\mathcal{D}(0) \right) d\nu_t.
\end{align}
By substituting \eqref{eq:direct3} into \eqref{eq:direct0},
we obtain
\begin{align*}
\frac{d}{dt} \int_M \tau \HH \, d\nu_t&=2\tau \int_M \left\langle \Psi, h+\nabla^2 f- df \otimes df \right\rangle d\nu_t+2\tau \int_M \langle \nabla (\tr \Psi),\nabla f \rangle  d\nu_t\\
&\quad +\tau \int_M \mathcal{D}(\nabla f)  d\nu_t.
\end{align*}
From integration by parts \eqref{eq:conjugate integration2},
we conclude the desired equation.
\end{proof}

We are now in a position to conclude the following result,
which is new even for Ricci flow.
\begin{prop}\label{prop:almost soliton scal}
For $\kappa>0,\eps \in (0,1)$,
if $\delta \leq \ov\delta_{\kappa, \eps}$, then the following holds:
Let $(M,(g_t)_{t\in I})$ be a super Ricci flow with $\mathcal{D}\geq 0$.
Let $(x_0, t_0) \in M \times I$.
For $r > 0$,
we assume that $(x_0,t_0)$ is $(\delta, r)$-selfsimilar, and $\NN_{(x_0, t_0)} ( r^2) \geq - \kappa$.
Then for all $t_1, t_2 \in [t_0 - \eps^{-1} r^2,t_0 - \eps r^2]$, we have
\begin{equation*}
\left|\int_M \tau \HH\, d\nu_{(x_0,t_0);t_1}-\int_M \tau \HH \, d\nu_{(x_0,t_0);t_2}\right| \leq \eps.
\end{equation*}
\end{prop}
\begin{proof}
We may assume $t_0 = 0$ and $r=1$.
We set $\nu:=\nu_{(x_0,0)}$.
We may also assume $t_1\leq t_2$.
From Lemma \ref{lem:formula},
it follows that
\begin{align*}
\left|\int_M \tau \HH d\nu_{t_1}-\int_M \tau \HH d\nu_{t_2}\right|&\leq  2\left|\int^{t_2}_{t_1} \int_M \tau\left\langle \Psi, h+\nabla^2 f- df \otimes df \right\rangle d\nu_t dt \right| \\
&\quad +2\left|\int^{t_2}_{t_1}\,\int_M \tau  (\tr \Psi)(|\nabla f|^2-\Delta f) \, d\nu_t dt \right|\\
&\quad +\int^{t_2}_{t_1}\int_M\,\tau\,\mathcal{D}(\nabla f)\,d\nu_t dt.
\end{align*}
By the Cauchy-Schwarz inequality,
we have
\begin{align*}
&\quad\,\, \left|\int^{t_2}_{t_1} \int_M \tau\left\langle \Psi, h+\nabla^2 f- df \otimes df \right\rangle d\nu_t dt \right|\\
&\leq \left(\int^{t_2}_{t_1} \int_M\, \tau  |\Psi|^2 \,d\nu_t dt \right)^{1/2}\left(C\int^{t_2}_{t_1} \int_M\, \tau  (|h|^2+|\nabla^2 f|^2+ |\nabla f|^4) \,d\nu_t dt \right)^{1/2}\leq C_{\kappa,\eps} \,\delta^{1/2}.
\end{align*}
In the same manner,
we obtain
\begin{align*}
&\quad\,\, \left|\int^{t_2}_{t_1}\,\int_M \tau  (\tr \Psi)(|\nabla f|^2-\Delta f) \, d\nu_t dt \right|\\
&\leq \left(\int^{t_2}_{t_1} \int_M\, \tau  (\tr \Psi)^2 \,d\nu_t dt \right)^{1/2}\left(\int^{t_2}_{t_1} \int_M\, \tau  (|\nabla f|^2-\Delta f)^2 \,d\nu_t dt \right)^{1/2}\\
&\leq \left(C\int^{t_2}_{t_1} \int_M\, \tau |\Psi|^2 \,d\nu_t dt \right)^{1/2}\left(C \int^{t_2}_{t_1} \int_M\, \tau  (|\nabla f|^4+|\nabla^2 f|^2) \,d\nu_t dt \right)^{1/2}\leq C_{\kappa,\eps} \,\delta^{1/2}.
\end{align*}
Furthermore,
by the almost selfsimilarity \eqref{eq:almost selfsimilar1},
\begin{equation*}
\int^{t_2}_{t_1}\int_M\,\tau\,\mathcal{D}(\nabla f)\,d\nu_t dt\leq \int^{-\delta}_{-\delta^{-1}}\int_M\,\tau\,\mathcal{D}(\nabla f)\,d\nu_t dt\leq 2\delta.
\end{equation*}
Hence,
we arrive at 
\begin{equation*}
\left|\int_M \tau \HH d\nu_{t_1}-\int_M \tau \HH d\nu_{t_2}\right|\leq C_{\kappa,\eps}\, \delta^{1/2}+2\delta.
\end{equation*}
This completes the proof.
\end{proof}

\begin{rem}
Bamler \cite{B3} has proved an almost monotonicity of this integral quantity along Ricci flow (see \cite[Proposition 7.9]{B3}).
Proposition \ref{prop:almost soliton scal} asserts that
it may be extended to our general setting,
and the reverse direction also holds.
\end{rem}

\subsection{Distance expansion estimate}\label{sec:dist_expansion}

We close this section with the following distance expansion estimate (cf. \cite[Proposition 9.1]{B3}):
\begin{prop}\label{prop:dist expansion} 
For $\kappa, D>0,\alpha \in (0,1)$,
if $\delta \leq \ov\delta_{\kappa,D, \alpha}$,
then the following holds:
Let $(M,(g_t)_{t\in I})$ be a super Ricci flow with $\mathcal{D}\geq 0$.
Let $(x_0, t_0), (x_1, t_1) \in M \times I$.
For $r>0$,
we assume that $(x_0, t_0)$ is $(\delta, r)$-selfsimilar, $\NN_{(x_0, t_0)} ( r^2)\geq - \kappa$ and $0\leq t_1 - t_0 \leq \alpha^{-1} r^2$.
Assume
\begin{equation}\label{eq:dist expansion1}
 W_1(\nu_{(x_0, t_0);s_0}, \nu_{(x_1,t_1);s_0} ) \leq Dr
\end{equation}
for some $s_0 \in [ t_{0}- \alpha^{-1} r^2, t_{0} - \alpha r^2 ]$.
Then we have
\begin{equation*}
W_1(\nu_{(x_0, t_0);s_0+\alpha r^2/4}, \nu_{(x_1,t_1);s_0+\alpha r^2/4}) \leq C_{\kappa, D, \alpha} r.
\end{equation*}
\end{prop}
\begin{proof}
We may assume $r = 1$.
For $i=0,1$,
we set $\nu^i := \nu_{(x_i, t_i)}$.
Let $\delta \leq \alpha/2$.
By Lemma \ref{lem:useful} we have $\NN_{(x_1, t_1)} (1)\geq -C_{\kappa,D,\alpha}$.
From Lemma \ref{lem:lower poteintial},
we conclude
\begin{equation}\label{eq:dist expansion2}
 f, f_1 \geq - C_{\kappa, D, \alpha}
\end{equation}
on $[t_{0} - \alpha^{-1}, t_{0})$,
where $f,f_1$ are the densities for $(x_0,t_0),(x_1,t_1)$, respectively.
Moreover,
we will denote by $\tau,\tau_1$ the parameters for $(x_0,t_0),(x_1,t_1)$, respectively.

Set $s_1 := s_0 -  \alpha/4$.
We first prove that
there exists $\Omega \subset M$ such that
\begin{equation}\label{eq:dist expansion15}
f(\cdot,s_1),f_1(\cdot,s_1)\leq C_{\kappa,D,\alpha},\quad  \nu^1_{s_1}(\Omega) \geq C_{\kappa,D,\alpha}
\end{equation}
on $\Omega$.
Let $(z,s_1)$ be a center of $(x_0, t_0)$ (see Proposition \ref{prop:center}).
By Lemma \ref{lem:concentration1} and Proposition \ref{prop:upper volume},
\begin{equation}\label{eq:dist expansion3}
\nu^0_{s_1 } (B) \geq \frac12,\quad m_{s_1}(B)\leq C_{\alpha} (t_0 - s_1)^{n/2},
\end{equation}
where $B:=B(z,s_1,\sqrt{2\Cn(t_0 - s_1)}$.
For $a >0$, we define $\Omega_0 := \{ f(\cdot, s_1) \leq a \} \cap B$.
In virtue of \eqref{eq:dist expansion3},
we possess
\begin{equation*}\label{eq:dist expansion5}
\nu^0_{s_1}(\Omega_0) \geq \frac12 -  \int_{B \setminus \Omega_0} (4\pi (t_0 - s_1 ))^{-n/2} e^{-f} dm_{s_1}\geq \frac12 - C e^{-a} (t_0-s_1)^{-n/2} m_{s_1}(B) \geq \frac12 - C_{\alpha} e^{-a};
\end{equation*}
in particular,
if $a \geq \underline{a}_{\alpha}$,
then \eqref{eq:dist expansion2} implies
\begin{equation}\label{eq:dist expansion6}
\nu^0_{s_1}(\Omega_0) \geq \frac14.
\end{equation}
Let us verify
\begin{equation}\label{eq:dist expansion7}
 \nu^1_{s_1 }(\Omega_0)\geq C_{\kappa,D,\alpha}. 
\end{equation}
Define a function $\phi : M \times [s_1,s_0] \to [0,1]$ by $\phi(y,s):= \nu_{(y, s); s_1}(\Omega_0)$,
and set $\psi:=\phi(\cdot,s_0)$.
Now, \eqref{eq:semigroup} and \eqref{eq:dist expansion6} imply
\begin{equation}\label{eq:dist expansion7_1}
\int_M \psi \, d\nu^0_{s_0}=\nu^0_{s_1}(\Omega_0)  \geq \frac14.
\end{equation}
Let $(y_0, s_0), (y_1, s_0)$ be centers of $(x_0,t_0), (x_1, t_1)$, respectively.
Lemma \ref{lem:W-V} implies
\begin{align}\label{eq:inheriting10000}
 d_{s_0} (y_0, y_1) &\leq W_1(\delta_{y_0}, \nu^0_{s_0}) + W_1(\nu^0_{s_0}, \nu^1_{s_0}) + W_1(\nu^1_{s_0}, \delta_{y_1})\\ \notag
&\leq \sqrt{\Cn (t_0 - s_0)} + D + \sqrt{\Cn (t_1 - s_0)} \leq C_{D,\alpha}. 
\end{align}
Furthermore,
by Lemma \ref{lem:concentration1},
\begin{equation}\label{eq:conc3}
\nu^0_{s_0} \left( M\setminus B(y_0, s_0, \sqrt{8 H_n  (t_0-s_0)   } ) \right)\leq \frac{1}{8},\quad \nu^1_{s_0} \left( B(y_1, s_0, \sqrt{2 H_n  (t_1-s_0)   } ) \right)\geq \frac{1}{2}.
\end{equation}
By \eqref{eq:dist expansion7_1}, \eqref{eq:conc3} and $\psi \leq 1$,
we see
\begin{equation*}
\int_{B(y_0, s_0,\sqrt{8 H_n  (t_0-s_0)   } )}  \psi \, d\nu^0_{s_0} \geq \frac{1}{4} - \nu^0_{s_0} \left( M \setminus   B(y_0, s_0, \sqrt{8 H_n  (t_0-s_0)   } ) \right) \geq \frac{1}{8};
\end{equation*}
in particular,
$\psi \geq 1/8$ at a point in $B(y_0, s_0, \sqrt{8 H_n  (t_0-s_0)   } )$.
By Theorem \ref{thm:gradient_estimate},
the function $\Phi^{-1}_{\alpha/ 4}\circ \psi$ is $1$-Lipschitz,
here $\Phi$ is defined as \eqref{eq:Gaussain};
in particular,
\eqref{eq:inheriting10000} yields $\psi \geq C_{\kappa,D,\alpha}$ on $B(y_1, s_0, \sqrt{2\Cn (t_1 - s_0)})$. 
Now,
\eqref{eq:semigroup} and \eqref{eq:conc3} lead us to
\begin{equation*}
\nu^1_{s_1}(\Omega_0) =\int_M \psi \, d\nu^1_{s_0} \geq \int_{B(y_1, s_0, \sqrt{2\Cn (t_1 - s_0)})} \psi \, d\nu^1_{s_0} \geq C_{\kappa,D,\alpha}.
\end{equation*}
This proves \eqref{eq:dist expansion7}.
We define $\Omega := \{ f_1(\cdot, s_1) \leq a \} \cap \Omega_0$.
With the help of \eqref{eq:dist expansion7} and \eqref{eq:dist expansion3},
\begin{align*}
\nu^1_{s_1}(\Omega) &\geq C_{\kappa,D,\alpha} -  \int_{\Omega_0 \setminus \Omega} (4\pi (t_1 - s_1 ))^{-n/2} e^{-f_1} dm_{s_1}\\
&\geq C_{\kappa,D,\alpha} - C e^{-a}  (t_0-s_1)^{-n/2}m_{ s_1 }(B) \geq C_{\kappa,D,\alpha} - C_{\alpha} e^{-a};
\end{align*}
in particular,
if $a \geq \underline{a}_{\kappa,D, \alpha}$, then we conclude \eqref{eq:dist expansion15}.

For $s:=s_0+\alpha/4$ and $s_2 := s_1 - \alpha/4$,
we define $u \in C^\infty (M \times [s_1, s])$ by
\begin{equation*}
u := \frac{1}{(4\pi (t - s_2))^{n/2}} \exp \left( {-  \frac{\tau (f  -\mathrm{N} )}{t-s_2}} \right),
\end{equation*}
where $\mathrm{N}  := \NN_{(x_0, t_0)}(1)$.
Note that
for every $t \in [s_1, s]$, we have $\tau\in [3\alpha/4,\alpha^{-1}+\alpha/4],t-s_2\in [\alpha/4,3\alpha/4]$, and
\begin{equation*}
1\leq \frac{\tau}{t-s_2}\leq \frac{\alpha^{-1} + \alpha/4}{\alpha/4}.
\end{equation*}
By \eqref{eq:dist expansion2} and \eqref{eq:dist expansion15},
we also see
\begin{equation}\label{eq:dist expansion9}
 \int_M u\, d\nu^1_{s_1} \geq \int_{\Omega} u \, d\nu^1_{s_1}  \geq C_{\kappa,D, \alpha},\quad u \leq C_{\kappa, D,\alpha} (4\pi \tau)^{-n/2} e^{-f}.
\end{equation}
By direct calculations,
\begin{align*}
 \square u &= -\left\{ \frac{1}{t - s_2} \left( \square (\tau f)  + \frac{n}2 + \mathrm{N}  \right)- \frac{\tau}{(t-s_2)^2} \left( - \tau( |\nabla  f|^2 +  \HH) +f  -  \mathrm{N}  \right) - \frac{\tau^2}{(t-s_2)^2}  \HH \right \} u \\
 &\geq - C_{\kappa,D, \alpha}  \left( \left|  \square (\tau f)  + \frac{n}2 + \mathrm{N}  \right| + \tau \left| - \tau( |\nabla  f|^2 +  \HH)  +f - \mathrm{N}   \right|+ \delta \right)u.
\end{align*}
This together with \eqref{eq:conjugate integration2}, \eqref{eq:dist expansion2} and \eqref{eq:dist expansion9} implies
\begin{align*}
&\quad\,\, \frac{d}{dt} \int_M  u \, d\nu^1_t = \int_M \square u  \, d\nu^1_t \\
&\geq - C_{\kappa, D,\alpha} \int_M   \left( \left|  \square (\tau f)  + \frac{n}2 + \mathrm{N}  \right| + \tau \left| - \tau( |\nabla  f|^2 +  \HH)  +f - \mathrm{N}   \right| + \delta \right)  (4\pi \tau)^{-n/2} e^{-f} d\nu^1_t \\
&= - C_{\kappa,D, \alpha} \int_M  \left( \left|  \square (\tau f)  + \frac{n}2 + \mathrm{N}  \right| + \tau \left| - \tau( |\nabla  f|^2 +  \HH)  +f - \mathrm{N}   \right| + \delta \right)   (4\pi \tau_1)^{-n/2} e^{-f_1} d\nu^0_t \\
&\geq - C_{\kappa, D,\alpha} \int_M  \left( \left|  \square (\tau f)  + \frac{n}2 + \mathrm{N}  \right| + \tau \left| - \tau( |\nabla  f|^2 +  \HH)  +f - \mathrm{N}   \right|  + \delta \right)  d\nu^0_t.
\end{align*}
Fix $\eta \in (0,1)$.
By Proposition \ref{prop:almost soliton estimate} and \eqref{eq:dist expansion9},
if $\delta \leq \ov\delta_{\kappa, D,\alpha,\eta}$ and $\eta \leq \ov{\eta}_{\kappa,D,\alpha}$, then
\begin{equation}\label{eq:dist expansion20}
\int_M u \, d\nu^1_{s} \geq \int_M u \, d\nu^1_{s_1 } -\eta\geq C_{\kappa,D,\alpha}.
\end{equation}
From \eqref{eq:dist expansion9} and \eqref{eq:dist expansion20},
it follows that
\begin{equation}\label{eq:dist expansion12}
 \int_M (4\pi)^{-n} (\tau \tau_1)^{-n/2} e^{-f-f_1} dm_{s}= \int_M (4\pi \tau)^{-n/2} e^{-f} d\nu^1_{s}  \geq C_{\kappa,D,\alpha} \int_M u \, d\nu^1_{s}  \geq C_{\kappa,D,\alpha}. 
\end{equation}

Let $(z_0, s),(z_1, s)$ be centers of $(x_0, t_0),(x_1, t_1)$, respectively.
We put $d:=d_{s} (z_0, z_1)$.
From \eqref{eq:dist expansion2} and \eqref{eq:dist expansion12},
we deduce
\begin{align*}
&\quad\,\, \nu^0_{s} ( M \setminus  B(z_0, s, d/2) ) + \nu^1_{s} ( M \setminus B(z_1, s, d/2))\\ \notag
&\geq C_{\kappa,D,\alpha} \left( \int_{M \setminus B(z_0, s, d/2)}  (4\pi)^{-n} (\tau \tau_1)^{-n/2} e^{-f-f_1} dm_{s } + \int_{M \setminus B(z_1, s, d/2)}  (4\pi)^{-n} (\tau \tau_1)^{-n/2} e^{-f-f_1} dm_{s }\right)\\ \notag
&\geq C_{\kappa,D,\alpha}\int_M (4\pi)^{-n} (\tau \tau_1)^{-n/2} e^{-f-f_1} dm_{s}\geq C_{\kappa,D,\alpha}. 
\end{align*}
This and \eqref{eq:elementary concentration} tell us that
\begin{align*}
C_{\kappa,D,\alpha}&\leq  \nu^0_{s} ( M \setminus  B(z_0, s, d/2) ) + \nu^1_{s} ( M \setminus B(z_1, s, d/2))\leq  \frac{\Cn(t_0 - s) +  \Cn(t_1 - s)}{(d/2)^2} \leq \frac{C_\alpha}{d^2};
\end{align*}
in particular,
$d \leq C_{\kappa,D, \alpha}$.
Lemma \ref{lem:W-V} implies
\begin{equation*}
W_1(\nu^0_{s}, \nu^1_{s})\leq W_1(\nu^0_{s}, \delta_{z_0} ) + d + W_1(\delta_{z_1}, \nu^1_{s} )\leq \sqrt{\Cn (t_0 - s)} + d+ \sqrt{\Cn (t_1 - s)} \leq C_{\kappa,D,\alpha}.
\end{equation*}
Thus,
we complete the proof.
\end{proof}

\section{Almost static points}\label{sec:almost_st}

For $\eps\in (0,1),r>0$,
a point $(x_0,t_0)\in M \times I$ is called \textit{$(\eps, r)$-static} if the following holds:
\begin{enumerate}\setlength{\itemsep}{+1.0mm}
\item $[t_0 - \eps^{-1} r^2, t_0 ] \subset I$;
\item we have
\begin{equation*}
r^2 \int_{t_0 - \eps^{-1} r^2}^{t_0 - \eps r^2} \int_M   |h|^2  d\nu_{(x_0, t_0); t} dt \leq \eps;
\end{equation*}
\item for all $t \in [t_0 - \eps^{-1} r^2, t_0 - \eps r^2]$, we have
\begin{equation*}
r^2\int_M \HH  \, d\nu_{(x_0, t_0);t} \leq \eps;
\end{equation*}
\item $\HH \geq -\eps r^{-2}$ on $M \times [t_0 - \eps^{-1} r^2, t_0 - \eps r^2]$.
\end{enumerate}
Our first main result is the following almost static cone splitting theorem,
which has been formulated by Bamler \cite{B3} for Ricci flow (cf. \cite[Proposition 10.1]{B3}):
\begin{thm}\label{thm:almost static}
For $\kappa, D>0,\alpha,\eps \in (0,1)$,
if $\delta \leq \ov\delta_{\kappa, D, \alpha, \eps}$, then the following holds:
Let $(M,(g_t)_{t\in I})$ be a super Ricci flow with $\mathcal{D}\geq 0$.
Let $(x_0, t_0), (x_1, t_1)  \in M \times I$.
For $r>0$,
we assume that $(x_0, t_0), (x_1, t_1)$ are $(\delta, r)$-selfsimilar, $\NN_{(x_0, t_0)}(r^2) \geq - \kappa$ and $\alpha r^2 \leq t_1 - t_0 \leq \alpha^{-1} r^2$.
If there exists $s_0 \in [t_0 - \alpha^{-1} r^2, t_0-\alpha r^2]$ such that
\begin{equation*}
W_1(\nu_{(x_0, t_0);s_0}, \nu_{(x_1,t_1);s_0}) \leq Dr,
\end{equation*}
then $(x_0,t_0)$ is $(\eps, r)$-static.
\end{thm}
\begin{proof}
We may assume $t_1 = 0$ and $r=1$.
For $i=0,1$,
set $\nu^i:=\nu_{(x_i, t_i)}$.
We also set $\mathrm{N}_1 := \NN_{(x_1, t_1)}(1)$.
Lemma \ref{lem:useful} implies $\mathrm{N}_1 \geq - C_{\kappa, D, \alpha}$.
Let $f,f_1$ be the densities for $(x_0,t_0),(x_1,t_1)$, respectively.
Moreover,
let $\tau,\tau_1$ be the parameters for $(x_0,t_0),(x_1,t_1)$, respectively.
Let $\theta \in (0,\ov\theta]$ be a constant obtained in Proposition \ref{prop:almost soliton estimate}.

We fix $\zeta \in (0,1)$.
By iterating Propositions \ref{prop:monotonicity_W1} and \ref{prop:dist expansion},
if $\delta  \leq \ov\delta_{\kappa, D,\alpha,\zeta}$,
then it holds that $W_1(\nu^0_{t_0-\zeta}, \nu^1_{t_0-\zeta} ) \leq C_{\kappa,D, \alpha,\zeta}$.
We also fix $\xi \in (0,1)$.
If $\zeta \leq \min\{\xi/2, (\mathfrak{C} \theta (1-\theta)^{-1}\xi)/2\}$,
then for every $t\in [t_0-1,t_0-\xi]$, we have $(t_0-\zeta)-t \geq \xi/2$ and 
\begin{align*}
&\HH(\cdot,t)\geq -\delta \geq -\frac{\xi}{2}\geq -((t_0-\zeta)-t)^{-1},\\
&W_1(\nu_{(x_0, t_0);t_0-\zeta}, \nu_{(x_1,t_1);t_0-\zeta} ) \leq  \frac{\sqrt{2}C_{\kappa,D,\alpha,\zeta}}{\sqrt{\xi}}\sqrt{(t_0-\zeta)-t}\leq C_{1,\kappa,D,\alpha,\zeta,\xi}\sqrt{(t_0-\zeta)-t},\\
&t_0-(t_0-\zeta)=\zeta\leq \frac{2\zeta}{\xi}((t_0-\zeta)-t)\leq \mathfrak{C}\frac{\theta}{1-\theta}((t_0-\zeta)-t),\\
&-(t_0-\zeta)\leq \alpha^{-1}+1\leq \frac{2(\alpha^{-1}+1)}{\xi}((t_0-\zeta)-t)\leq C^2_{1,\kappa,D,\alpha,\zeta,\xi}((t_0-\zeta)-t),
\end{align*}
where $\mathfrak{C}$ is a constant obtained in Proposition \ref{prop:inheriting}.
By Proposition \ref{prop:inheriting},
if $\zeta \leq \ov\zeta_{\xi}$, then
\begin{equation}\label{eq:almost static1}
\nu^0_t \leq C_{\kappa,D, \alpha,\zeta,\xi} e^{\theta f_1}\nu^{1}_t,\quad f_1(\cdot, t_0-1)\leq C_{\kappa,D, \alpha}+(1-\theta)^{-1}f(\cdot,t_0-1)
\end{equation}
for every $t\in [t_0-1,t_0-\xi]$.
Fix $\eta \in (0,1)$.
By Proposition \ref{prop:almost soliton estimate},
if $\delta \leq \ov\delta_{\kappa,D,\eta}$,
then
\begin{align}\label{eq:almost static21}
&\int_{t_1 - \eta^{-1}}^{t_1 - \eta} \int_M  \left| \square (\tau_1 f_1) + \frac{n}2 + \mathrm{N}_1 \right| e^{\theta f_1} d\nu^1_{t} dt \leq \eta,\\ \label{eq:almost static22}
&\int_{t_1 - \eta^{-1}}^{t_1 - \eta} \int_M  \left|{ - \tau_1 (\HH+ |\nabla f_1|^2) + f_1  - \mathrm{N}_1} \right| e^{\theta f_1} d\nu^1_{t} dt \leq \eta.
\end{align}

By \eqref{eq:almost static1} and \eqref{eq:almost static22},
if $\eta \leq \ov\eta_{\kappa,D,\alpha,\zeta,\xi}$,
then
\begin{align*}
&\quad\,\, \int_{t_0 - 2\xi}^{t_0 - \xi} \int_M  \left|{ - \tau_1 (\HH+ |\nabla f_1|^2) + f_1  - \mathrm{N}_1} \right|  d\nu^0_{t} dt\\
&\leq C_{\kappa,D,\alpha,\zeta,\xi} \,\int_{t_0 - 2\xi}^{t_0 - \xi} \int_M  \left|{ - \tau_1 (\HH+ |\nabla f_1|^2) + f_1  - \mathrm{N}_1} \right|  e^{\theta f_1} d\nu^1_{t} dt\\
&\leq C_{\kappa,D,\alpha,\zeta,\xi} \,\int_{t_1 - \eta^{-1}}^{t_1 - \eta} \int_M  \left|{ - \tau_1 (\HH+ |\nabla f_1|^2) + f_1  - \mathrm{N}_1} \right|  e^{\theta f_1} d\nu^1_{t} dt\leq C_{\kappa,D,\alpha,\zeta,\xi} \,\eta \leq \xi;
\end{align*}
in particular,
\begin{align}\label{eq:almost static24}
\int_{t_0 - 2\xi}^{t_0 - \xi} \int_M \tau_1  \HH\,d\nu^0_t dt&\leq \xi-\int_{t_0 - 2\xi}^{t_0 - \xi} \int_M (\tau_1|\nabla f_1|^2-f_1+\mathrm{N}_1)\,d\nu^0_t dt\\ \notag
&\leq C_{\kappa,D,\alpha}\, \xi+\int_{t_0 - 2\xi}^{t_0 - \xi} \int_M f_1\,d\nu^0_t dt.
\end{align}
Fix $s\in [t_0-2\xi,t_0-\xi]$.
By \eqref{eq:almost static1} and \eqref{eq:almost static21},
if $\eta \leq \ov\eta_{\kappa,D,\alpha,\zeta,\xi}$, then
\begin{align}\label{eq:almost static25}
&\quad\,\, \int^{s}_{t_0 - 1} \int_M  \left|\square (\tau_1 f_1) + \frac{n}2 + \mathrm{N}_1 \right|  d\nu^0_{t} dt\\ \notag
&\leq C_{\kappa,D,\alpha,\zeta,\xi} \,\int^{s}_{t_0 - 1} \int_M  \left|\square (\tau_1 f_1) + \frac{n}2 + \mathrm{N}_1\right|  e^{\theta f_1} d\nu^1_{t} dt\\ \notag
&\leq C_{\kappa,D,\alpha,\zeta,\xi} \,\int_{t_1 - \eta^{-1}}^{t_1 - \eta} \int_M  \left| \square (\tau_1 f_1) + \frac{n}2 + \mathrm{N}_1 \right|  e^{\theta f_1} d\nu^1_{t} dt\leq C_{\kappa,D,\alpha,\zeta,\xi}\, \eta \leq \xi.
\end{align}
Hence,
\eqref{eq:almost static1} and \eqref{eq:almost static25} together with \eqref{eq:conjugate integration2} imply
\begin{align*}
\int_M \tau_1 f_1 d\nu^0_{s}&=\int_M \tau_1 f_1 d\nu^0_{t_0 -1}+\int^{s}_{t_0 -1} \int_M  \square(\tau_1f_1) d\nu^0_{t} dt\\
&\leq (1-t_0)\left(C_{\kappa,D,\alpha}+(1-\theta)^{-1}\int_M  f d\nu^0_{t_0 -1} \right)+\int^{s}_{t_0 -1} \int_M  \square(\tau_1f_1) d\nu^0_{t} dt\\
&\leq (1+\alpha^{-1})\left(C_{\kappa,D,\alpha}+(1-\theta)^{-1}\frac{n}{2} \right)+\xi-\int^{s}_{t_0 - 1} \int_M  \left( \frac{n}2 + \mathrm{N}_1 \right)  d\nu^0_{t} dt\\
&\leq (1+\alpha^{-1})\left(C_{\kappa,D,\alpha}+(1-\theta)^{-1}\frac{n}{2} \right)+\xi+C_{\kappa,D,\alpha}(s-t_0+1)\leq C_{\kappa,D,\alpha},
\end{align*}
which leads us to
\begin{equation}\label{eq:almost static26}
\int_M f_1 d\nu^0_{s}\leq \frac{C_{\kappa,D,\alpha}}{\tau_1}\leq C_{\kappa,D,\alpha}.
\end{equation}
From \eqref{eq:almost static24} and \eqref{eq:almost static26},
we conclude
\begin{equation*}
\int_{t_0 - 2\xi}^{t_0 - \xi} \int_M \tau_1  \HH\,d\nu^0_t dt\leq C_{\kappa,D,\alpha} \,\xi;
\end{equation*}
in particular,
there exists $t_2 \in [t_0 - 2\xi, t_0 - \xi]$ such that
\begin{equation*}
\int_M \tau_1 \HH \, d\nu^0_{ t_2}  \leq C_{\kappa,D,\alpha}.
\end{equation*}
Note that $\tau \leq 2\xi$ and $\tau_1 \geq \alpha$ at $t = t_2$.
By Proposition \ref{prop:almost soliton scal}, if $\delta \leq \ov\delta_{\kappa, \eps, \xi}$, then
\begin{equation*}\label{eq:almost static6}
  \int_M  \tau \HH \, d\nu^0_t  \leq \int_M  \tau \HH \, d\nu^0_{t_2} + \xi= \frac{\tau}{\tau_1}\int_M  \tau_1 \HH \, d\nu^0_{t_2} + \xi \leq C_{\kappa,D,\alpha}\, \xi
\end{equation*}
for all $t \in [t_0 - \eps^{-1}, t_2]$.
Using this bound for $t = t_0 - \eps$ and $H\geq -\delta$ leads us to
\begin{equation*}\label{eq:almost static5}
 2 \int_{t_0 - \eps^{-1}}^{t_0 - \eps} \int_M |h|^2 d\nu^0_t dt = \int_M \HH \, d\nu^0_{t_0 - \eps} -\int_M \HH \, d\nu^0_{t_0 - \eps^{-1}}- \int_{t_0 - \eps^{-1}}^{t_0 - \eps} \int_M \mathcal{D}(0) d\nu^0_t dt\leq   \frac{C_{\kappa,D,\alpha}\, \xi}{\eps} + \delta. 
\end{equation*}
If $\xi \leq \ov\xi_{\kappa, D, \alpha, \eps}$ and $\delta \leq \ov\delta_{\eps}$,
then we complete the proof.
\end{proof}

\section{Almost splitting}\label{sec:almost_split}

For $\eps\in (0,1),r>0$,
a \textit{$(k,\eps,r)$-splitting map at $(x_0, t_0)\in M\times I$} is a map $\vec y = (y_1, \ldots, y_k) : M \times [t_0 - \eps^{-1} r^2, t_0 - \eps r^2] \to \IR^k$ with the following properties for all $i,j = 1, \ldots, k$:
\begin{enumerate}\setlength{\itemsep}{+1.0mm}
\item $[t_0 - \eps^{-1} r^2, t_0 ] \subset I$;
\item  we have
\begin{equation*}
 r^{-1}\int_{t_0 - \eps^{-1} r^2}^{t_0 - \eps r^2} \int_M |\square y_i  | d\nu_{(x_0, t_0);t} dt \leq \eps;
 \end{equation*}
\item we have
\begin{equation*}
r^{-2}\int_{t_0 - \eps^{-1} r^2}^{t_0 - \eps r^2} \int_M \left|\langle \nabla y_i, \nabla y_j\rangle  - \delta_{ij} \right| d\nu_{(x_0, t_0);t} dt \leq \eps. 
\end{equation*}
\end{enumerate}
We say that $(x_0, t_0)$ is \textit{$(k,  \eps, r)$-split} if
there is a $(k,\eps,r)$-splitting map.
The aim of this section is prove the following almost splitting theorem,
which has been obtained by Bamler \cite{B3} for Ricci flow (cf. \cite[Proposition 10.8]{B3}):
\begin{thm}\label{thm:main almost splitting}
For $\kappa, D>0,\eps,\xi \in (0,1)$,
if $\beta\leq \ov\beta,\mathfrak{D} \geq \underline{\mathfrak{D}}_{\kappa,D}, \mathfrak{N} \geq \underline{\mathfrak{N}}_{\kappa, D}, \delta \leq \ov\delta_{\kappa, D, \eps,\xi}$,
then the following holds:
Let $(M,(g_t)_{t\in I})$ be a super Ricci flow with $\mathcal{D}\geq 0$.
Let $\{(x_i, t_i)\}^{N-1}_{i=0}\subset M \times I$ with $t_0 \leq \cdots \leq t_{N-1}$ and $N \geq \mathfrak{N}\, \xi^{-k}$.
For $r>0$,
we assume $\{(x_i, t_i)\}^{N-1}_{i=0}$ are $(\delta, r)$-selfsimilar, $\NN_{(x_0, t_0)} (r^2) \geq - \kappa$ and $0\leq t_i-t_0 \leq \beta\,\xi^2 r^2$ for all $i$.
Assume the following:
\begin{enumerate}\setlength{\itemsep}{+2.0mm}
\item $W_1(\nu_{(x_0, t_0);t_0 - r^2},\nu_{(x_i, t_i); t_0- r^2}) \leq Dr$ for all $i$;
\item $W_1(\nu_{(x_i, t_i);t_0 - 2\xi^2 r^2}, \nu_{(x_j, t_j); t_0 - 2\xi^2 r^2}) \geq  \mathfrak{D} \xi r$ for all $i\neq j$.
\end{enumerate}
Then $(x_0, t_0)$ is $(k+1, \eps, r)$-split.
\end{thm}

\subsection{Construction of coordinate functions}

We begin with the following:
\begin{lem}\label{lem:splitting lem1}
For $\kappa, D>0,\eps,\xi \in (0,1)$,
if $\delta \leq \ov\delta_{\kappa, D,\eps,\xi}$, then the following holds:
Let $(M,(g_t)_{t\in I})$ be a super Ricci flow with $\mathcal{D}\geq 0$.
Let $(x_0, t_0), (x_1, t_1)  \in M \times I$.
For $r>0$,
we assume that $(x_0, t_0), (x_1, t_1)$ are $(\delta, r)$-selfsimilar, $\NN_{(x_0, t_0)}(r^2) \geq - \kappa$ and $0 \leq t_1 - t_0 \leq \xi^2 r^2$.
We assume
\begin{equation}\label{eq:W1assumption}
W_1(\nu_{(x_0, t_0);t_0-r^2}, \nu_{(x_1,t_1);t_0-r^2}) \leq D r.
\end{equation}
Let $\tau,\tau_1$ be the parameters for $(x_0,t_0),(x_1,t_1)$, respectively.
Moreover,
let $f,f_1$ be the densities for $(x_0,t_0),(x_1,t_1)$, respectively.
We define a function
\begin{equation}\label{eq:linear_static}
u := r^{-1}(\tau_1 f_1  - \tau f  - (\mathrm{N}_1 - \mathrm{N} ) \tau), 
\end{equation}
where $\mathrm{N}:=\NN_{(x_0, t_0)}(r^2)$ and $\mathrm{N}_1:=\NN_{(x_1,t_1)}(r^2)$.
Then we have
\begin{equation}\label{eq:splitting lem1_goal}
\int_{t_0-\eps^{-1}r^2}^{t_0-\xi^2 r^2} \int_M \tau^{-1/2}|\square u|   d\nu_{(x_0,t_0);t} dt \leq \eps,\quad \int_{t_0-\eps^{-1}r^2}^{t_0-\xi^2 r^2} \int_M   |\nabla^2 u|^2 d\nu_{(x_0,t_0);t} dt \leq \eps.
\end{equation}
\end{lem}
\begin{proof}
We may assume $t_0 = 0$ and $r=1$.
For $i=0,1$,
we set $\nu^i:=\nu_{(x_i, t_i)}$.
Let $\theta \in (0,\ov\theta]$ be a constant obtained in Proposition \ref{prop:almost soliton estimate}.

Fix $\zeta \in (0,1)$.
By iterating Propositions \ref{prop:monotonicity_W1} and \ref{prop:dist expansion},
if $\delta  \leq \ov\delta_{\kappa, D,\zeta}$,
then $W_1(\nu^0_{-\zeta}, \nu^1_{-\zeta} ) \leq C_{\kappa,D,\zeta}$.
If $\zeta \leq \ov\zeta_\xi$,
then for every $t\in [-\eps^{-1},-\xi^2]$ we see $-\zeta-t \geq \xi^2-\zeta$, and 
\begin{align*}
&\HH(\cdot,-\zeta)\geq -\delta \geq -(-\zeta-t)^{-1},\\
&W_1(\nu^0_{-\zeta}, \nu^1_{-\zeta} ) \leq  \frac{C_{\kappa,D, \zeta}}{\sqrt{\xi^2-\zeta}}\sqrt{-\zeta-t}\leq C_{1,\kappa,D,\xi,\zeta}\sqrt{-\zeta-t},\\
&\zeta\leq \frac{\zeta}{\xi^2-\zeta}(-\zeta-t)\leq \mathfrak{C}\frac{\theta}{1-\theta}(-\zeta-t),\\
&t_1-(-\zeta)\leq \xi^2+\zeta\leq \frac{\xi^2+\zeta}{\xi^2-\zeta}(-\zeta-t)\leq C^2_{1,\kappa,D,\xi,\zeta}(-\zeta-t),
\end{align*}
Therefore,
by Proposition \ref{prop:inheriting},
if $\zeta \leq \ov\zeta_\xi$, then we have
\begin{equation}\label{eq:splitting lem11}
\nu^0_t \leq C_{\kappa,D, \xi,\zeta} e^{\theta f_1}\nu^{1}_t
\end{equation}
for every $t\in [-\eps^{-1},-\xi^2]$.
We fix $\eta \in (0,1)$.
By Proposition \ref{prop:almost soliton estimate},
if $\delta \leq \ov\delta_{\kappa,D,\eta}$,
then
\begin{align}\label{eq:splitting lem12}
&\int_{- \eta^{-1}}^{- \eta} \int_M  \tau \left| h + \nabla^2 f - \frac1{2\tau} g \right|^2 d\nu^0_{t} dt \leq \eta,\\ \label{eq:splitting lem13}
&\int_{ - \eta^{-1}}^{- \eta} \int_M  \left| \square (\tau f) + \frac{n}2 + \mathrm{N} \right| d\nu^0_{t} dt \leq \eta,\\ \label{eq:splitting lem14}
&\int_{t_1- \eta^{-1}}^{t_1- \eta} \int_M  \tau_1 \left| h + \nabla^2 f_1 - \frac1{2\tau_1} g \right|^2 e^{\theta f_1} d\nu^1_{t} dt \leq \eta,\\ \label{eq:splitting lem15}
&\int_{t_1 - \eta^{-1}}^{t_1 - \eta} \int_M  \left| \square (\tau_1 f_1) + \frac{n}2 + \mathrm{N}_1 \right| e^{\theta f_1} d\nu^1_{t} dt \leq \eta.
\end{align}
By \eqref{eq:splitting lem11}, \eqref{eq:splitting lem13}, \eqref{eq:splitting lem15},
if $\eta \leq \ov\eta_{\kappa,D,\eps,\xi,\zeta}$,
then we obtain
\begin{align*}
 \int_{-\eps^{-1}}^{-\xi^2} \int_M \tau^{-1/2}|\square u|   d\nu^0_t dt &\leq C_{\kappa,D, \xi,\zeta}  \int_{-\eps^{-1}}^{-\xi^2} \int_M \tau^{-1/2}\left| \square (\tau_1 f_1 ) + \frac{n}2 + \mathrm{N}_1 \right|  e^{\theta f_1}  d\nu^1_t dt  \\
&\quad  + \int_{-\eps^{-1}}^{-\xi^2} \int_M \tau^{-1/2}\left| \square (\tau f ) + \frac{n}2 + \mathrm{N} \right|  d\nu^0_t dt\leq  C_{\kappa,D, \xi,\zeta}\, \eta.
\end{align*}
This proves the first estimate in \eqref{eq:splitting lem1_goal}.

We next show the second estimate in \eqref{eq:splitting lem1_goal}.
If $\eta \leq \ov{\eta}_{\xi}$, then
\begin{align*}
&\quad\,\,  \int_{-\eps^{-1}}^{-\xi^2} \int_M    |\nabla^2 u|^2 d\nu^0_t dt \\
&= \int_{-\eps^{-1}}^{-\xi^2} \int_M   \left| \tau_1 \left( h+\nabla^2 f_1 - \frac1{2\tau_1} g \right) -  \tau \left( h+\nabla^2 f - \frac1{2\tau} g \right) + (\tau - \tau_1) h  \right|^2 d\nu^0_t dt \\
&\leq C_{\kappa,D, \xi,\zeta} \int_{-\eps^{-1}}^{-\xi^2} \int_M  \tau_1^2  \left| h+\nabla^2 f_1 - \frac1{2\tau_1} g \right|^2 e^{\theta f_1} d\nu^1_t dt \\
&\quad +  C\int_{-\eps^{-1}}^{-\xi^2} \int_M  \tau^2  \left|h+ \nabla^2 f - \frac1{2\tau} g \right|^2  d\nu^0_t dt + C t_1^2 \int_{-\eps^{-1}}^{-\xi^2} \int_M |h|^2 d\nu^0_t dt \\
&\leq C_{\kappa,D, \xi,\zeta} \,\eta + C t^2_1\int_{-\eps^{-1}}^{-\xi^2} \int_M |h|^2 d\nu^0_t dt.
\end{align*}
For a fixed $\alpha \in (0,1)$,
we first consider the case of $t_1\in [0,\alpha]$.
Proposition \ref{prop:integral1} tells us that if $\alpha \leq \ov{\alpha}_{\kappa,\eps,\xi}$ and $\eta \leq \ov{\eta}_{\kappa,D,\eps,\xi,\zeta}$, then
\begin{equation*}
\int_{-\eps^{-1}}^{-\xi^2} \int_M   |\nabla^2 u|^2 d\nu^0_t dt \leq C_{\kappa,D, \xi,\zeta} \,\eta + C_{\kappa,  \xi} \alpha^2\leq \eps.
\end{equation*}
On the other hand,
in the case of $t_1 \geq \alpha$, Theorem \ref{thm:almost static} tells us that if $\delta\leq \ov{\delta}_{\kappa,D,\alpha,\eta}$, then
\begin{equation*}
\int_{-\eta^{-1}}^{-\eta} \int_M |h|^2 d\nu^0_t dt \leq \eta.
\end{equation*}
Therefore,
if $\eta \leq \ov{\eta}_{\kappa,D, \eps,\xi,\zeta}$, then
\begin{equation*}
\int_{-\eps^{-1}}^{-\xi^2} \int_M    |\nabla^2 u|^2 d\nu^0_t dt \leq C_{\kappa,D,\xi,\zeta}\,\eta \leq \eps.
\end{equation*}
This completes the proof of the second estimate in \eqref{eq:splitting lem1_goal}.
\end{proof}

\begin{rem}
Along the line of the above proof,
by using Proposition \ref{prop:integral1} with \eqref{eq:splitting lem11},
we also obtain the following (cf. \cite[Claim 10.26]{B3}):
Under the same setting as in Lemma \ref{lem:splitting lem1},
\begin{equation}\label{eq:uniform}
\int_{t_0-\eps^{-1} r^2}^{t_0-\xi^2 r^2} \int_M \left(\tau^{-5}u^8+\tau^{-1}|\nabla u|^4+(\partial_t u)^2 \right) d\nu_{(x_0,t_0);t} dt \leq C_{\kappa,D,\eps,\xi}.
\end{equation}
\end{rem}

For functions in Lemma \ref{lem:splitting lem1},
we further see the following (cf. \cite[Claim 10.39]{B3}):
\begin{lem}\label{lem:splitting lem2}
For $\kappa, D>0,\xi \in (0,1)$,
if $\delta \leq \ov\delta_{\kappa, D}$, then the following holds:
Let $(M,(g_t)_{t\in I})$ be a super Ricci flow with $\mathcal{D}\geq 0$.
Let $(x_0, t_0), (x_1, t_1)  \in M \times I$.
For $r>0$,
we assume that $(x_0, t_0), (x_1, t_1)$ are $(\delta, r)$-selfsimilar, $\NN_{(x_0, t_0)}(r^2) \geq - \kappa$ and $0 \leq t_1 - t_0 \leq \xi^2 r^2$.
We further assume \eqref{eq:W1assumption}.
Let $u$ be a function defined as \eqref{eq:linear_static}.
Then we have
\begin{equation*}
\int_{t_0-2\xi^2 r^2}^{t_0-\xi^2 r^2} \int_M \tau^{-1}|\nabla u|^2 d\nu_{(x_0,t_0);t} dt \leq C_{\kappa,D}.
\end{equation*}
\end{lem}
\begin{proof}
We may assume $t_0 = 0$ and $r=1$.
Set $\nu^i:=\nu_{(x_i, t_i)}$ for $i=0,1$.
Let $\theta \in (0,\ov\theta]$ be a constant obtained in Proposition \ref{prop:integral1}.
Fix $\zeta \in (0,1)$.
Using Proposition \ref{prop:dist expansion},
if $\delta  \leq \ov\delta_{\kappa, D,\zeta}$,
then $W_1(\nu^0_{-\zeta \xi^2}, \nu^1_{-\zeta \xi^2} ) \leq C_{\kappa,D,\zeta}\xi$.
By Proposition \ref{prop:inheriting},
if $\zeta \leq \ov\zeta$, then $\nu^0_t \leq C_{\kappa,D, \zeta} e^{\theta f_1}\nu^{1}_t$ for every $t\in [-2\xi^2,-\xi^2]$.
Proposition \ref{prop:integral1} yields
\begin{align*}
 \int_{-2\xi^2}^{-\xi^2} \int_M \tau^{-1} |\nabla u  |^2 d\nu^0_t dt &\leq  C \int_{-2\xi^2}^{-\xi^2} \int_M \tau^{-1}\tau^2_1 |\nabla f_{1}  |^2 d\nu^0_t dt + C \int_{-2\xi^2}^{-\xi^2} \int_M \tau\,|\nabla f  |^2 d\nu^0_t dt \\
 &\leq  C_{\kappa,D,\zeta} \int_{-2\xi^2}^{-\xi^2} \int_M  |\nabla f_{1}  |^2 e^{\theta f_1} d\nu^1_t dt  + C \int_{-2\xi^2}^{-\xi^2} \int_M  |\nabla f  |^2 d\nu^0_t dt \leq C_{\kappa, D,\zeta}.
\end{align*}
This completes the proof.
\end{proof}

\subsection{Proof of Theorem \ref{thm:main almost splitting}}

We have the following (cf. \cite[Claim 10.31]{B3}):
\begin{lem}\label{lem:splitting lem3}
For $\kappa, D>0,\eps,\xi \in (0,1)$,
if $\delta \leq \ov\delta_{\kappa, D,\eps,\xi}$, then the following holds:
Let $(M,(g_t)_{t\in I})$ be a super Ricci flow with $\mathcal{D}\geq 0$.
Let $(x_0, t_0), (x_1, t_1),(x_2,t_2)  \in M \times I$ with $t_0\leq t_1\leq t_2$.
For $r>0$,
we assume that $(x_0, t_0), (x_1, t_1),(x_2,t_2)$ are $(\delta, r)$-selfsimilar, $\NN_{(x_0, t_0)}(r^2) \geq - \kappa$ and $0 \leq t_i - t_0 \leq \xi^2 r^2$ for $i=1,2$.
For $i=1,2$,
we further assume
\begin{equation*}
W_1(\nu_{(x_0, t_0);t_0- r^2}, \nu_{(x_i,t_i);t_0- r^2}) \leq D r.
\end{equation*}
Let $\tau,\tau_1,\tau_2$ be the parameters for $(x_0,t_0),(x_1,t_1),(x_2,t_2)$, respectively.
Moreover,
let $f,f_1,f_2$ be the densities for $(x_0,t_0),(x_1,t_1),(x_2,t_2)$, respectively.
For $i=1,2$,
we define
\begin{equation}\label{eq:linear_static2}
u_i := r^{-1}(\tau_i f_i  - \tau f  - (\mathrm{N}_i - \mathrm{N} ) \tau),
\end{equation}
where $\mathrm{N}:= \NN_{(x_0, t_0)}(r^2), \,\mathrm{N}_{i}:=\NN_{(x_i,t_i)}(r^2)$.
Then there are constants $\mathrm{c}_1,\mathrm{c}_2,\mathrm{c}_3\in \mathbb{R}$ such that
\begin{align}\label{eq:almost constant1}
\int_{t_0-\eps^{-1}r^2}^{t_0-\xi^2 r^2} \int_M \tau^{-1} \left| |\nabla (u_1+u_2)|^2 - \mathrm{c}_1 \right|   d\nu_{(x_0,t_0);t} dt \leq \eps,\\ \label{eq:almost constant2}
\int_{t_0-\eps^{-1}r^2}^{t_0-\xi^2 r^2} \int_M  \tau^{-1}\left| |\nabla (u_1-u_2)|^2 - \mathrm{c}_2 \right|   d\nu_{(x_0,t_0);t} dt \leq \eps,\\ \label{eq:almost constant3}
\int_{t_0-\eps^{-1}r^2}^{t_0-\xi^2 r^2} \int_M \tau^{-1}\left| \langle \nabla u_1,\nabla u_2 \rangle - \mathrm{c}_3 \right|   d\nu_{(x_0,t_0);t} dt \leq \eps.
\end{align}
\end{lem}
\begin{proof}
We may assume $t_0 = 0$ and $r=1$.
We set $\nu:=\nu_{(x_0,0)}$.
Let us prove \eqref{eq:almost constant1}.
We define $u:=u_1+u_2$ and fix $\eta \in (0,1)$.
By Propositions \ref{prop:Nash almost}, \ref{prop:almost soliton estimate} and Lemma \ref{lem:splitting lem1},
if $\delta \leq \ov\delta_{\kappa,D,\eta,\xi}$,
then
\begin{align}\label{eq:almost constant11}
&\int_{- \eta^{-1}}^{- \xi^2} \int_M \left|\tau (2\Delta f - |\nabla f|^2 + \HH )+f-n-\mathrm{N} \right| \,  d\nu_t dt\leq \eta,\\ \label{eq:almost constant12}
&\int_{- \eta^{-1}}^{- \xi^2} \int_M \tau\left|\HH + \Delta f -  \frac{n}{2\tau} \right| \, d\nu_t\leq \eta,\\ \label{eq:almost constant13}
&\int_{- \eta^{-1}}^{- \xi^2} \int_M \tau^{-1/2}|\square u|  \,  d\nu_t dt\leq \eta,\\ \label{eq:almost constant14}
&\int_{- \eta^{-1}}^{- \xi^2} \int_M  |\nabla^2 u|^2    d\nu_t dt\leq \eta,\\ \label{eq:almost constant15}
&\left|\mathcal{N}_{(x_0,0)}(\tau)-\mathrm{N}  \right|\leq \eta
\end{align}
for all $t\in [-\eta^{-1},-\eta]$.

We first prove that
there is a constant $\mathrm{c}\in \mathbb{R}$ such that
for all $t\in [-\eps^{-1},-\xi^2]$,
\begin{equation}\label{eq:almost constant_step1}
\left| \int_M \tau^{-1} u^2 \left( f- \frac{n}2 - \mathrm{N} \right) d\nu_{t}-\mathrm{c} \right|\leq C_{\kappa,D,\eps,\xi} \,\eta^{1/4}.
\end{equation}
Using \eqref{eq:conjugate integration2} and \eqref{eq:potential_enjoy},
we obtain the following (cf. \cite[Lemma 10.10]{B3}):
\begin{align*}
\frac{d}{dt}  \int_M  \tau^{-1} u^2 \left( f- \frac{n}2 - \mathrm{N} \right) d\nu_t&= \int_M \tau^{-2} u^2  \left(\tau (2\Delta f - |\nabla f|^2 + \HH )+f-n-\mathrm{N}  \right) d\nu_t  \\
&\quad - 2\int_M \tau^{-1} u^2  \left(\HH + \Delta f -  \frac{n}{2\tau} \right) d\nu_t  \\ 
&\quad +2 \int_M \tau^{-1} u \square u \left( f  - \frac{n}{2} - \mathrm{N}   \right) d\nu_t \\
 &\quad - 2 \int_M \tau^{-1} \left( |\nabla u|^2 - \int_M |\nabla u|^2 d\nu_t \right) \left( f  - \frac{n}{2} - \mathrm{N}   \right) d\nu_t \\
 &\quad - 2\tau^{-1}  \left( \int_M |\nabla u|^2 d\nu_t \right) \int_M  \left( f  - \frac{n}{2} - \mathrm{N}   \right) d\nu_t.
\end{align*}
For $[s_1, s_2] \subset [- \eps^{-1}, - \xi^2 ]$,
it holds that
\begin{align}\label{eq:each term}
 &\quad\,\, \left| \int_M \tau^{-1}u^2 \left( f- \frac{n}2 - \mathrm{N}  \right) d\nu_{s_2}-\int_M \tau^{-1} u^2 \left( f- \frac{n}2 - \mathrm{N}  \right) d\nu_{s_1} \right| \\ \notag
 &\leq  \xi^{-4} \int^{-\xi^2}_{-\eps^{-1}} \int_M u^2  \left|\tau (2\Delta f - |\nabla f|^2 + \HH )+f-n-\mathrm{N}  \right| d\nu_t  dt \\ \notag
 &\quad + 2\xi^{-4}\int^{-\xi^2}_{-\eps^{-1}} \int_M u^2  \tau \left|\HH + \Delta f -  \frac{n}{2\tau} \right| d\nu_t dt\\ \notag
 &\quad +2\xi^{-1} \int^{-\xi^2}_{-\eps^{-1}}\int_M \tau^{-1/2}|\square u| |u|  \left| f  - \frac{n}{2} - \mathrm{N}   \right| d\nu_t dt \\ \notag
 &\quad +2\xi^{-2}\int^{-\xi^2}_{-\eps^{-1}} \int_M \left| |\nabla u|^2 - \int_M |\nabla u|^2 d\nu_t \right| \left| f  - \frac{n}{2} - \mathrm{N}   \right| d\nu_t dt \\ \notag
 &\quad  + 2 \xi^{-2} \int^{-\xi^2}_{-\eps^{-1}} \left( \int_M |\nabla u|^2 d\nu_t \right) \left|\mathcal{N}_{(x_0,0)}(\tau)-\mathrm{N}  \right| dt.
\end{align}
Let us estimate each term in the right hand side of \eqref{eq:each term}.
Let $a>0$.
For the first term,
\eqref{eq:almost constant11}, Proposition \ref{prop:integral1} and \eqref{eq:uniform} imply
\begin{align}\label{eq:first term}
&\quad\,\,  \int_{- \eps^{-1}}^{- \xi^2} \int_M u^2  \left|\tau (2\Delta f - |\nabla f|^2 + \HH )+f-n-\mathrm{N}  \right| \, d\nu_t dt\\ \notag
&\leq a^{-1} \int_{- \eps^{-1}}^{- \xi^2} \int_M \left|\tau (2\Delta f - |\nabla f|^2 + \HH )+f-n-\mathrm{N}  \right| \,  d\nu_t dt\\ \notag
&\quad  + a \int_{- \eps^{-1}}^{- \xi^2} \int_M \left(\tau (2\Delta f - |\nabla f|^2 + \HH )+f-n-\mathrm{N}  \right)^2 \,  d\nu_t dt + a \int_{- \eps^{-1}}^{- \xi^2} \int_M  u^8\, d\nu_t dt \\ \notag
&\leq a^{-1} \eta + a \,C_{\kappa,D,\eps,\xi}.
\end{align}
For the second term,
\eqref{eq:almost constant12}, Proposition \ref{prop:integral1} and \eqref{eq:uniform} lead us to
\begin{align}\label{eq:second term}
\int_{- \eps^{-1}}^{- \xi^2} \int_M u^2  \tau\left|\HH + \Delta f -  \frac{n}{2\tau} \right| \, d\nu_t dt &\leq a^{-1} \int_{- \eps^{-1}}^{- \xi^2} \int_M \tau \left|\HH + \Delta f -  \frac{n}{2\tau} \right| \,  d\nu_t dt\\ \notag
&\quad  + a \int_{- \eps^{-1}}^{- \xi^2} \int_M \tau^2\left(\HH + \Delta f -  \frac{n}{2\tau} \right)^2 \,  d\nu_t dt\\ \notag
&\quad + a \int_{- \eps^{-1}}^{- \xi^2} \int_M  u^8\, d\nu_t dt \\ \notag
&\leq a^{-1} \eta + a \,C_{\kappa,D,\eps,\xi}.
\end{align}
For the third term,
\eqref{eq:almost constant13}, Proposition \ref{prop:integral1} and \eqref{eq:uniform} yield
\begin{align*}
\int_{- \eps^{-1}}^{- \xi^2} \int_M \tau^{-1/2} |\square u| |u|\left| f  - \frac{n}{2} - \mathrm{N}   \right| \, d\nu_t dt &\leq a^{-1} \int_{- \eps^{-1}}^{- \xi^2} \int_M \tau^{-1/2}|\square u|  \,  d\nu_t dt\\
&\quad  + a \int_{- \eps^{-1}}^{- \xi^2} \int_M  \tau^{-1} (\square u)^2\,  d\nu_t dt\\
&\quad + a \int_{- \eps^{-1}}^{- \xi^2} \int_M  u^4 \left( f  - \frac{n}{2} - \mathrm{N}   \right)^4 d\nu_t dt \\
&\leq a^{-1} \eta + a \,C_{\kappa,D,\eps,\xi}.
\end{align*}
For the fourth term,
\eqref{eq:almost constant14}, the Kato inequality, Propositions \ref{prop:integral1} and \ref{prop:Poincare} with $p=1$ and \eqref{eq:uniform} imply
\begin{align*}
&\quad\,\, \int^{-\xi^2}_{-\eps^{-1}} \int_M \left| |\nabla u|^2 - \int_M |\nabla u|^2 d\nu_t \right| \left| f  - \frac{n}{2} - \mathrm{N}   \right| d\nu_t dt\\
&\leq a^{-1}\int_{- \eps^{-1}}^{- \xi^2} \int_M  \left| |\nabla u|^2 - \int_M |\nabla u|^2 d\nu_t \right|   d\nu_t dt\\
&\quad +a\int_{- \eps^{-1}}^{- \xi^2} \int_M  \left( |\nabla u|^2 - \int_M |\nabla u|^2 d\nu_t \right)^2   d\nu_t dt +a\int_{- \eps^{-1}}^{- \xi^2} \int_M  \left( f  - \frac{n}{2} - \mathrm{N}   \right)^4   d\nu_t dt\\
&\leq C a^{-1}\int_{- \eps^{-1}}^{- \xi^2} \int_M  |\nabla |\nabla u|^2 |     d\nu_t dt+a\, C_{\kappa,D,\eps,\xi}\\
&\leq C a^{-1}\int_{- \eps^{-1}}^{- \xi^2} \int_M  |\nabla u||\nabla^2 u|    d\nu_t dt+a\, C_{\kappa,D,\eps,\xi}\\
&\leq C a^{-1}\left(\int_{- \eps^{-1}}^{- \xi^2} \int_M  |\nabla u|^2    d\nu_t dt \right)^{1/2}\left(\int_{- \eps^{-1}}^{- \xi^2} \int_M  |\nabla^2 u|^2    d\nu_t dt \right)^{1/2}+a \,C_{\kappa,D,\eps,\xi}\\
&\leq a^{-1}C_{\kappa,D,\eps,\xi} \eta^{1/2} + a \,C_{\kappa,D,\eps,\xi}.
\end{align*}
For the last term,
\eqref{eq:almost constant15}, Proposition \ref{prop:integral1} and \eqref{eq:uniform} tell us that
\begin{equation*}
\int^{-\xi^2}_{-\eps^{-1}} \left( \int_M |\nabla u|^2 d\nu_t \right) \left|\mathcal{N}_{(x_0,0)}(\tau)-\mathrm{N} \right| dt\leq C_{\kappa,D,\eps,\xi} \,\eta.
\end{equation*}
Combining them with \eqref{eq:each term},
and choosing $a=\eta^{1/4}$,
we obtain
\begin{align*}
&\quad \,\,\left| \int_M \tau^{-1} u^2 \left( f- \frac{n}2 - \mathrm{N}  \right) d\nu_{s_2}-\int_M \tau^{-1} u^2 \left( f- \frac{n}2 - \mathrm{N}  \right) d\nu_{s_1} \right|\\
&\leq C_{\kappa,D,\eps,\xi}\left(a^{-1} \eta+ a^{-1} \eta^{1/2}+a  +\eta \right)\leq C_{\kappa,D,\eps,\xi}\, \eta^{1/4}.
\end{align*}
We arrive at \eqref{eq:almost constant_step1}.

Using \eqref{eq:conjugate integration2},
for every $t \in [- \eps^{-1}, - \xi^2 ]$ we also see the following (cf. \cite[Lemma 10.10]{B3}):
\begin{align*}
&\quad\,\, \int_M \left\{   |\nabla u|^2-\frac{1}{2}\tau^{-1}  u^2 \left(f- \frac{n}2 - \mathrm{N}  \right)  \right\} d\nu_t\\
&= -  \frac{1}{2}\int_M \tau^{-1} u^2 \left\{\tau (2\Delta f - |\nabla f|^2 + \HH )+f-n-\mathrm{N}  \right\} d\nu_t  \\ 
 &\quad +  \frac{1}{2}\int_M u^2 \left( \HH + \Delta f - \frac{n}{2\tau} \right) d\nu_t -\int_M u \Delta u \, d\nu_t.
\end{align*}
Let $b>0$.
By \eqref{eq:first term}, \eqref{eq:second term} and \eqref{eq:almost constant14},
we have
\begin{align*}
&\quad \,\,\int^{-\xi^2}_{-\eps^{-1}}\,\left| \int_M \,\left\{   |\nabla u|^2-\frac{1}{2}\tau^{-1}  u^2 \left(f- \frac{n}2 - \mathrm{N}  \right)  \right\} d\nu_t \right|\,dt\\
&\leq C \xi^{-2}\int^{-\xi^2}_{-\eps^{-1}}\,\int_M u^2 \left|\tau (2\Delta f - |\nabla f|^2 + \HH )+f-n-\mathrm{N}  \right| d\nu_t \, dt\\ 
&\quad + C\xi^{-2} \int^{-\xi^2}_{-\eps^{-1}}\,\int_M u^2\tau \left| \HH + \Delta f - \frac{n}{2\tau} \right| d\nu_t \,dt+ \int^{-\xi^2}_{-\eps^{-1}}\,\int_M |u| |\Delta u| \, d\nu_t \,dt\\
&\leq C_{\kappa,D,\eps,\xi} \left(b^{-1} \eta +b\right)+C \left(\int^{-\xi^2}_{-\eps^{-1}}\,\int_M\, u^2 \, d\nu_t \,dt \right)^{1/2}\left(\int^{-\xi^2}_{-\eps^{-1}}\,\int_M |\nabla^2 u|^2 \, d\nu_t \,dt \right)^{1/2}\\
&\leq C_{\kappa,D,\eps,\xi} \left(b^{-1} \eta +b+\eta^{1/2} \right)\leq C_{\kappa,D,\eps,\xi} \,\eta^{1/2},
\end{align*}
where we choose $b=\eta^{1/2}$.
It follows that
\begin{align*}
\int_{-\eps^{-1}}^{-\xi^2} \int_M \tau^{-1} \left| |\nabla u|^2 - \frac{\mathrm{c}}{2} \right|   d\nu_{t} dt
&\leq \int_{-\eps^{-1}}^{-\xi^2} \int_M \tau^{-1} \left| |\nabla u|^2 - \int_{M}\,|\nabla u|^2\,d\nu_t \right|   d\nu_{t} dt\\
&\quad +\int^{-\xi^2}_{-\eps^{-1}}\,\tau^{-1}\left| \int_M \,\left\{   |\nabla u|^2-\frac{1}{2}\tau^{-1}  u^2 \left(f- \frac{n}2 - \mathrm{N} \right)  \right\} d\nu_t \right|\,dt\\
&\quad +\frac{1}{2}\int^{-\xi^2}_{-\eps^{-1}}\,\tau^{-1}\left| \int_M \,\left\{ \tau^{-1}  u^2 \left(f- \frac{n}2 - \mathrm{N} \right)-\mathrm{c}  \right\} d\nu_t \right|\,dt\\
&\leq C_{\kappa,\eps,\xi} \,\eta^{1/4}.
\end{align*}
This proves \eqref{eq:almost constant1}.
Once we obtain \eqref{eq:almost constant1},
the estimate \eqref{eq:almost constant2} can be derived from the same calculation.
Moreover,
\eqref{eq:almost constant3} follows from \eqref{eq:almost constant1} and \eqref{eq:almost constant2}.
We complete the proof.
\end{proof}

We also prove the following (cf. \cite[Claim 10.32]{B3}):
\begin{lem}\label{lem:splitting lem4}
For $\kappa, D>0,\xi\in (0,1)$,
if $\beta \leq \ov{\beta},\mathfrak{D} \geq \underline{\mathfrak{D}}_{\kappa,D},\delta \leq \ov\delta_{\kappa, D,\xi}$,
then the following holds:
Let $(M,(g_t)_{t\in I})$ be a super Ricci flow with $\mathcal{D}\geq 0$.
Let $(x_0, t_0), (x_1, t_1),(x_2,t_2)  \in M \times I$ with $t_0\leq t_1\leq t_2$.
For $r>0$,
we assume that $(x_0, t_0), (x_1, t_1),(x_2,t_2)$ are $(\delta, r)$-selfsimilar, $\NN_{(x_0, t_0)}(r^2) \geq - \kappa$ and $0 \leq t_i - t_0 \leq \beta \xi^2 r^2$ for $i=1,2$.
Suppose the following:
\begin{enumerate}\setlength{\itemsep}{+2.0mm}
\item $W_1(\nu_{(x_0, t_0);t_0 - r^2},\nu_{(x_i, t_i); t_0- r^2}) \leq D r$ for $i=1,2$;
\item $W_1(\nu_{(x_1, t_1);t_0 - 2\xi^2 r^2}, \nu_{(x_2, t_2); t_0 - 2\xi^2 r^2}) \geq  \mathfrak{D} \xi r$.
\end{enumerate}
For $i=1,2$,
let $u_i$ be a function defined as \eqref{eq:linear_static2}.
Then we have
\begin{equation*}
\int^{t_0-\xi^2 r^2}_{t_0-2\xi^2 r^2} \int_M \tau^{-1} |\nabla (u_1 - u_2) |^2 d\nu_{(x_0,t_0);t} dt \geq  \xi^2.
\end{equation*}
\end{lem}
\begin{proof}
We may assume $t_0=0$ and $r=1$.
Set $\nu:=\nu_{(x_0,0)}$ and $\nu^i:=\nu_{(x_i,t_i)}$ for $i=1,2$.
Fix $\eta \in (0,1)$.
By Proposition \ref{prop:almost soliton estimate} and Lemma \ref{lem:splitting lem3},
if $\delta \leq \ov\delta_{\kappa, D,\eta,\xi}$,
then there exists a constant $\mathrm{c}\in \mathbb{R}$ such that
\begin{equation*}
\int^{-\xi^2}_{-2\xi^2} \left\{ \int_M \tau^{-1}\left| |\nabla (u_1 - u_2) |^2 -\mathrm{c} \right| d\nu_t+\sum^2_{i=1}\int_M \left| - \tau_i (  |\nabla f_i|^2 + \HH) + f_i  - \mathrm{N}_i \right|  d\nu^i_{t}       \right\} dt  \leq 3 \eta;
\end{equation*}
in particular,
there exists $s \in [-2\xi^2, -\xi^2]$ such that
\begin{align}\label{eq:main almost splitting7.100}
&\quad\,\, \int_M  \tau^{-1}\left| | \nabla ( \tau_1 f_{1} - \tau_2 f_{2}) |^2 -\mathrm{c} \right| d\nu_{s}+\sum^{2}_{i=1} \int_M \left| - \tau_i (  |\nabla f_i|^2 + \HH) + f_i  - \mathrm{N}_i \right|  d\nu^i_{s}\\ \notag
&=\int_M \tau^{-1} \left|  |\nabla (u_1 - u_2) |^2 -\mathrm{c} \right| d\nu_{s}+\sum^{2}_{i=1}\int_M \left| - \tau_i (  |\nabla f_i|^2 + \HH) + f_i  - \mathrm{N} _i \right|  d\nu^i_{s}\leq C\eta \,\xi^{-2}.
\end{align}
Let $\zeta\in (0,1)$.
Due to Propositions \ref{prop:monotonicity_W1} and \ref{prop:dist expansion},
if $\delta  \leq \ov\delta_{\kappa, D,\xi,\zeta}$,
then $W_1(\nu^j_{-\zeta \xi^2}, \nu^2_{-\zeta \xi^2} ) \leq C_{\kappa,D, \xi,\zeta}$ for $j=0,1$. 
If $\beta\leq (\mathfrak{C} \theta)/2$ and $\zeta \leq \ov{\zeta}_{\beta}$,
then
\begin{align*}
&\HH(\cdot,s)\geq -\delta \geq -(1-\zeta)^{-1} \xi^{-2}\geq -(-\zeta \xi^2-s)^{-1},\\
&W_1(\nu^j_{-\zeta \xi^2}, \nu^2_{-\zeta \xi^2} ) \leq C_{1,\kappa,D,\xi,\zeta}\sqrt{-\zeta \xi^2-s},\\
&t_2+\zeta \xi^2 \leq (\beta+\zeta) \xi^2 \leq \mathfrak{C}\theta (1-\zeta) \xi^2 \leq \mathfrak{C}\theta (-\zeta \xi^2-s),\\
&\zeta \xi^2 \leq t_1+\zeta \xi^2\leq (\beta+\zeta) \xi^2\leq C^2_{1,\kappa,D,\xi,\zeta}(1-\zeta) \xi^2,
\end{align*}
where $\theta \in (0,\ov\theta]$ and $\mathfrak{C}$ are constants obtained in Propositions \ref{prop:almost soliton estimate} and \ref{prop:inheriting},
respectively.
Proposition \ref{prop:inheriting} implies $e^{-\theta f_2}\nu^2_s \leq C_{\kappa,D,\xi} \,\nu^j_s$ for $j=0,1$;
in particular, \eqref{eq:main almost splitting7.100} yields
\begin{equation}\label{eq:main almost splitting7.4}
\int_M \left( \tau^{-1}\left| | \nabla ( \tau_1 f_{1} - \tau_2 f_{2}) |^2 -\mathrm{c} \right|+\sum^{2}_{i=1} \left| - \tau_i (  |\nabla f_i|^2 + \HH) + f_i  - \mathrm{N}_i \right|  \right)\,e^{-\theta f_2}\,d\nu^2_{s}\leq  C_{\kappa,D,\xi}\,\eta.
\end{equation}

For each $i= 1,2$,
let $(y_i, s)$ be a center of $(x_i, t_i)$ (see Proposition \ref{prop:center}).
By Lemma \ref{lem:W-V} and Proposition \ref{prop:monotonicity_W1},
if $\mathfrak{D} \geq \underline{\mathfrak{D}}$, then
\begin{align}\label{eq:main almost splitting7.5}
 d_s (y_1 , y_2)& \geq W_1(\nu^1_{-2\xi^2}, \nu^2_{-2\xi^2}) -\sqrt{\Var(\delta_{y_1}, \nu^1_{s})} - \sqrt{\Var (\delta_{y_2}, \nu^2_{s})} \\ \notag
                        &\geq \left(\mathfrak{D} - 2 \sqrt{2\Cn} \right) \xi \geq \frac12 \mathfrak{D} \xi.
\end{align}
From \eqref{eq:main almost splitting7.5}, Lemma \ref{lem:concentration1} and \eqref{eq:elementary concentration},
it follows that
\begin{equation}\label{eq:main almost splitting9}
 \nu^2_{s} (  B ) \geq \frac12,\quad \nu^1_{s} (  B ) \leq \nu^1_{s} \left( M \setminus B(y_1,s, \mathfrak{D} \xi/4) \right) \leq   \frac{\Var(\nu^1_{s}, \delta_{y_1})}{(\mathfrak{D} \xi/4)^2}  \leq   \frac{32 \Cn}{\mathfrak{D}^2},
\end{equation}
where $B := B(y_2,s,2 \sqrt{ \Cn } \xi)$.
For $a>0$, we set $\Omega_{0} := \{ f_2(\cdot,s) \leq a \} \cap B$.
By Proposition \ref{prop:upper volume},
if $a \geq \underline{a}$,
then
\begin{equation}\label{eq:main almost splitting12}
\nu^2_{s} ( B \setminus \Omega_{0} ) \leq C \xi^{-n} e^{-a} m_{s}(B)\leq C  e^{-a} \leq \frac18.
\end{equation}
By \eqref{eq:main almost splitting7.4},
if $\eta\leq \ov{\eta}_{\kappa,D,\xi}$,
then there is $\Omega \subset \Omega_{0}$ with $\nu^2_{s} \left( \Omega_{0} \setminus \Omega \right) \leq 1/8$ such that
for $i =1,2$,
\begin{equation*}
\tau^{-1}\left| | \nabla ( \tau_1 f_{1} - \tau_2 f_{2}) |^2 -\mathrm{c} \right|+\sum^{2}_{i=1} \left| - \tau_i (  |\nabla f_i|^2 + \HH) + f_i  - \mathrm{N}_i \right| \leq 1
\end{equation*}
on $\Omega$ at time $s$.
Note that $\tau \in [\xi^2,2\xi^2],\tau_i \in [\xi^2,3\xi^2]$.
We also notice that
if $\mathrm{c}<3\xi^2$,
then
\begin{align}\label{eq:main almost splitting12}
&\tau_2 |\nabla f_2 |^2 \leq  \left|  \tau_2 (  |\nabla f_2|^2 + \HH) - f_2  + \mathrm{N}_2  \right|-\tau_2 \HH+ f_2-\mathrm{N}_2 \leq C_{\kappa, D,a},\\ \notag
&\tau_1 |\nabla f_1 |^2 \leq  2 \xi^{-2} \left(  | \nabla ( \tau_1 f_{1} - \tau_2 f_{2})|^2+\tau_2^2 |\nabla f_2 |^2 \right) \leq 2  \xi^{-2} \left( \mathrm{c}+\tau+C_{\kappa,D,a} \xi^2 \right) \leq C_{\kappa,D,a},\\ \notag
&\left| \tau_1 |\nabla f_1| - \tau_2 |\nabla f_2| \right| \leq (\mathrm{c}+\tau)^{1/2}\leq C\xi
\end{align}
on $\Omega$.
In view of \eqref{eq:main almost splitting9} and \eqref{eq:main almost splitting12},
we also possess $\nu^2_{s} (\Omega)\geq 1/4$.

Let us show $\mathrm{c}\geq 3\xi^2$.
We suppose $\mathrm{c} < 3\xi^2$.
From \eqref{eq:main almost splitting12},
we derive
\begin{align*}
\tau_2 f_2 - \tau_1 f_1 &\geq (\tau_2^2 - \tau_1^2) \HH -\sqrt{2} \left| \tau_1 |\nabla f_1| - \tau_2 |\nabla f_2| \right|\left(\sum^{2}_{i=1}\tau^2_i |\nabla f_i|^2\right)^{1/2}-\sum^2_{i=1}\tau_i|\mathrm{N}_i|\\
&\quad\,\, -\sum^2_{i=1}\tau_i \left| - \tau_i (  |\nabla f_i|^2 + \HH) + f_i  - \mathrm{N}_i  \right|\geq  -C_{\kappa,D,a}\, \xi^2
\end{align*}
on $\Omega$.
It follows that $f_1 \leq C_{\kappa,D,a}$ on $\Omega$.
Therefore, \eqref{eq:main almost splitting9} leads us to
\begin{equation*}
\frac{1}{4} \leq \nu^2_{s} (\Omega) =  \int_{\Omega} (4\pi \tau_2)^{-n/2} e^{-f_2} dm_{s}\leq  C_{\kappa,D,a} \int_{\Omega} (4 \pi \tau_1)^{-n/2} e^{-f_1} dm_{s}= C_{\kappa,D,a}\, \nu^1_{s} (B) \leq \frac{C_{\kappa,D,a} }{\mathfrak{D}^2}.
\end{equation*}
If $\mathfrak{D} \geq \underline{\mathfrak{D}}_{\kappa,D,a}$,
then this yields the contradiction,
and hence $\mathrm{c} \geq 3\xi^2$.
We conclude
\begin{align*}
\int^{-\xi^2}_{-2\xi^2} \int_M \tau^{-1} |\nabla (u_1 - u_2) |^2 d\nu_t dt &\geq \frac{\mathrm{c}}{2}-\int^{-\xi^2}_{-2\xi^2} \int_M \tau^{-1}\left| |\nabla (u_1 - u_2) |^2 -\mathrm{c} \right| d\nu_t dt\geq \xi^2
\end{align*}
if $\eta \leq \xi^2/2$.
We complete the proof.
\end{proof}

We are now in a position to conclude Theorem \ref{thm:main almost splitting}:
\begin{proof}[Proof of Theorem \ref{thm:main almost splitting}]
Once we obtain Lemmas \ref{lem:splitting lem1}, \ref{lem:splitting lem2}, \ref{lem:splitting lem3}, \ref{lem:splitting lem4}, 
Theorem \ref{thm:main almost splitting} follows from the same argument as in the proof of \cite[Proposition 10.8]{B3} together with \cite[Lemma 10.23]{B3}.
Thus,
we complete the proof.
\end{proof}

\section{Quantitative stratification}\label{sec:stratification}

\subsection{Parabolic balls}\label{sec:parabolic_nbd}

For $(x_0, t_0) \in M \times I$ and $r>0$ (with $t_0-r^2\in I$),
the \textit{parabolic ball} is defined by
\begin{equation*}
P(x_0, t_0; r):=\left\{\, (x,t)\in M\times I \,\mid\, t \in [t_0-r^2, t_0+r^2],~ W_1(\nu_{(x_0, t_0); t_0 - r^2}, \nu_{(x,t);  t_0 - r^2}) < r \, \right\}.
\end{equation*}

We recall the following basic property (see \cite[Proposition 9.4]{B1}, \cite[Proposition 4.25]{B3}):
\begin{lem}[\cite{B1}]\label{lem:parab_nbhd}
For $(x_0, t_0), (x_1, t_1) \in M \times I$ and $r>0$,
if $P (x_0, t_0;r) \cap P (x_1, t_1; r) \neq \emptyset$, then $P (x_0, t_0; r) \subset P (x_1, t_1; 3r)$.
\end{lem}

We first show the following volume estimate for time-slices (cf. \cite[Theorem 9.8]{B1}):
\begin{lem}\label{lem:time_slice}
Let $(M,(g_t)_{t\in I})$ be a super Ricci flow with $\mathcal{D}\geq 0$.
Let $(x_0,t_0)\in M\times I$.
For $r,\rho,R>0$,
we assume $[t_0 - (R^2+\rho) r^2, t_0] \subset I$.
Then for every $t \in [t_0 - R^2 r^2, t_0 + R^2r^2]$,
\begin{equation*}
 m_t(P (x_0, t_0; R r) \cap \left( M \times \{ t \} \right)) \leq C_{\rho,R} \exp ( \NN_{(x_0, t_0)} (\rho r^2/2) ) r^n.
\end{equation*}
\end{lem}
\begin{proof}
We may assume $t_0=R^2$ and $r=1$.
Lemma \ref{lem:scal} implies $\HH(\cdot,-\rho/2)\geq -n/\rho$ (see Remark \ref{rem:scal2}).
Let $(z_0,0)$ be a center of $(x_0,R^2)$ (see Proposition \ref{prop:center}).
For $t \in [0,2R^2]$ and $x \in S_t$,
let $(z, 0)$ be a center of $(x,t)$,
where $S_t := P (x_0, t_0; R) \cap \left( M \times \{ t \} \right)$. 
By Lemma \ref{lem:W-V},
\begin{align*}
 d_0 (z_0,z)&\leq W_1(\delta_{z_0}, \nu_{(x_0, R^2); 0})+W_1(\nu_{(x_0,R^2);0}, \nu_{(x,t);0})+W_1(\nu_{(x,t);0}, \delta_{z}) \\
                 &\leq \sqrt{\Cn} R +R + \sqrt{\Cn t} \leq C_{1}R
\end{align*}
for some $C_1>1$;
in particular,
$B(z, 0, \sqrt{2\Cn t})$ is contained in $B:=B(z_0,0, C_{1}R)$.
Proposition \ref{prop:concentration2} tells us that
\begin{equation}\label{eq:time_slice3}
 \nu_{(x,t);0} (B) \geq \nu_{(x,t);0} \left( B(z, 0, \sqrt{2\Cn t} ) \right) \geq \frac12. 
\end{equation}

Let $u \in L^\infty(M \times [0, 2R^2]) \cap C^\infty(M \times (0,2R^2])$ be the solution to the heat equation with initial condition $u(\cdot, 0) = \chi_B$.
By Lemma \ref{lem:scal},
we have
\begin{equation*}
\frac{d}{dt} \int_M u \, dm_t = - \int_M u \HH \, dm_t \leq  \frac{n}{\rho} \int_M u \, dm_t.
\end{equation*}
In virtue of \eqref{eq:time_slice3},
we possess $u \geq 1/2$ on $S_t$,
and hence
\begin{equation}\label{eq:time_slice4}
 \frac12 m_t(S_t) \leq \int_M u(\cdot, t) dm_t \leq e^{ \frac{n t}{\rho}} m_0(B)\leq C_{\rho} m_0(B). 
\end{equation}
Using Proposition \ref{prop:upper volume}, 
we obtain
\begin{equation}\label{eq:time_slice5}
m_0(B)\leq C_{\rho,R} \exp(\NN_{(z_0,0)}(\rho/2)).
\end{equation}
Also,
by \eqref{eq:Nash3.5} and Lemmas \ref{lem:Nash cor} and \ref{lem:W-V},
\begin{align}\label{eq:time_slice6}
\NN_{(z_0,0)}(\rho/2)-\NN_{(x_0,R^2)}(\rho/2)&\leq \NN_{(z_0,0)}(\rho/2)-\NN_{(x_0,R^2)}(\rho/2+R^2)\\ \notag
&= \NN_{-\rho/2}(z_0,0)-\NN_{-\rho/2}(x_0,R^2)\\ \notag
&\leq C_{\rho,R}( W_1(\delta_{z_0} , \nu_{(x_0,R^2);0})+1)\leq C_{\rho,R}.
\end{align}
From \eqref{eq:time_slice4}, \eqref{eq:time_slice5} and \eqref{eq:time_slice6},
we conclude the desired estimate.
\end{proof}

Based on Lemma \ref{lem:time_slice},
we present the following (cf. \cite[Theorem 9.11]{B1}):
\begin{prop}\label{prop:covering}
Let $(M,(g_t)_{t\in I})$ be a super Ricci flow with $\mathcal{D}\geq 0$.
Let $(x_0, t_0) \in M \times I$.
For $r ,\rho,R> 0$,
we assume $[t_0 - 2(R^2+\rho)r^2, t_0] \subset I$. 
Let $S \subset P(x_0, t_0; Rr)$ and $\xi \in (0,\sqrt{\rho}]$.
Then there exists $\{(x_i, t_i)\}^{N}_{i=1}\subset S$ such that
\begin{equation}\label{eq:covering1}
S \subset \bigcup_{i=1}^N P(x_i, t_i; \xi r), \quad N \leq C_{\rho,R} \,\xi^{-n-2}. 
\end{equation}
\end{prop}
\begin{proof}
We may assume $t_0=0$ and $r = 1$.
Let $\{(x_i, t_i)\}^{N}_{i=1} \subset S$ be a maximal collection of points such that $\{P (x_i, t_i; \xi/3)\}^{N}_{i=1}$ are pairwise disjoint.
Let $(x, t) \in S$.
By the maximality we see $P(x, t; \xi/3) \cap P(x_i, t_i; \xi/3) \neq \emptyset$ for some $i$.
From Lemma \ref{lem:parab_nbhd}, we deduce $(x, t) \in P (x_i, t_i;\xi)$,
which proves the inclusion in \eqref{eq:covering1}.

Let us derive an upper bound on $N$ in \eqref{eq:covering1}.
We fix $\zeta \in (0, 1/2]$.
If $\zeta \leq \ov\zeta_{\rho,R}$,
then there exist $\mathcal{I} \subset \{ 1, \ldots, N \}$ and $s \in [-R^2 - 2\zeta \xi^2,R^2 - \zeta \xi^2]$ such that
\begin{equation}\label{eq:covering2}
|\mathcal{I}| \geq \lfloor C_{ \zeta,R} \,\xi^2 \,N \rfloor,\quad s \in [t_i -2\zeta \xi^2, t_i - \zeta\xi^2]
\end{equation}
for all $i \in \mathcal{I}$.
For each $i \in \mathcal{I}$,
let $(z_i, s)$ be a center of $(x_i, t_i)$ (see Proposition \ref{prop:center}).
By Lemma \ref{lem:W-V} and Proposition \ref{prop:monotonicity_W1},
if $\zeta \leq \ov\zeta$,
then for all $i,j \in \mathcal{I}$ with $i \neq j$, it holds that
\begin{align*}
 d_{s} (z_i, z_j) &\geq W_1(\nu_{(x_i, t_i);s},\nu_{(x_j, t_j);s}) - W_1(\delta_{z_i}, \nu_{(x_i, t_i);s}) - W_1(\delta_{z_j}, \nu_{(x_j, t_j);s}) \\
&\geq W_1(\nu_{(x_i, t_i);t_i -  (\xi/3)^2},\nu_{(x_j, t_j);t_i -   (\xi/3)^2})   - 2 \sqrt{2\Cn \zeta}\xi \geq \frac{\xi}{3}  - 2 \sqrt{2\Cn \zeta} \xi \geq \frac{\xi}{4};
\end{align*}
in particular,
$\{B(z_i, s,\xi/8)\}_{i \in \mathcal{I}}$ are pairwise disjoint.
For each $i\in \mathcal{I}$,
Lemma \ref{lem:W-V} and Proposition \ref{prop:monotonicity_W1} lead us to
\begin{align*}
&\quad \,\,W_1(\nu_{(x_0,0); -R^2-2\zeta \xi^2}, \nu_{(z_i,s); -R^2-2\zeta \xi^2})\\
&\leq W_1(\nu_{(x_0,0); -R^2-2\zeta \xi^2}, \nu_{(x_i,t_i); -R^2-2\zeta \xi^2})+W_1(\nu_{(x_i,t_i); -R^2-2\zeta \xi^2}, \nu_{(z_i,s); -R^2-2\zeta \xi^2})\\
   &\leq W_1(\nu_{(x_0,0); -R^2}, \nu_{(x_i,t_i); -R^2})+W_1(\nu_{(x_i,t_i); s}, \delta_{z_i})\\
   &\leq R+\sqrt{H_n(t_i-s)}\leq R+\sqrt{2H_n \zeta} \xi<R+\frac{ \xi}{8}
\end{align*}
if $\zeta \leq \ov{\zeta}$,
and hence $(z_i,s)\in P(x_0,0;R+ \xi/8)\cap (M\times \{s\})$.
For every $y\in B(z_i,s,\xi/8)$,
\begin{align*}
&\quad \,\,W_1(\nu_{(x_0,0);-(R+\xi/8)^2},\nu_{(z_i,s);-(R+\xi/8)^2})\\
&\leq W_1(\nu_{(x_0,0);-(R+\xi/8)^2},\nu_{(y,s);-(R+\xi/8)^2})+d_s(y,z_i)\leq R+\frac{\xi}{4},
\end{align*}
and thus
$B(z_i,s,\xi/8)\times \{s\}$ is contained in $P(x_0,0;R+\xi/4)\cap (M\times \{s\})$.
Summarizing the above observations,
from Lemma \ref{lem:time_slice} and \eqref{eq:Nash3.5},
we derive
\begin{align}\label{eq:covering4}
|\mathcal{I}| \,\inf_{i\in \mathcal{I}}m_s(B(z_i, s,\xi/8)) &\leq \sum_{i\in \mathcal{I}}m_s(B(z_i, s,\xi/8))\\ \notag
&\leq m_s(P(x_0,0;R+\xi/4)\cap (M\times \{s\})\\ \notag
&\leq C_{\rho,R} \exp ( \NN_{(x_0, 0)} (\rho/2 )).
\end{align}
On the other hand,
by Lemma \ref{lem:scal} and Proposition \ref{prop:lower volume}, 
if $\zeta \leq \ov{\zeta}$,
then 
\begin{align}\label{eq:covering3}
 m_{s}(B(z_i, s,\xi/8)) &\geq m_{s}(B(z_i, s,2\sqrt{H_n \zeta}\xi)) \geq   m_s(B(z_i,s, \sqrt{2H_n(t_i-s)}))\\ \notag
 &\geq C_{\rho,R} \exp ( \NN_{(x_i, t_i)} (t_i-s ) ) (t_i-s)^{n/2}\geq C_{\rho,R} \exp ( \NN_{(x_i, t_i)} ( \rho/2 ) ) \xi^n.
\end{align}
Moreover, \eqref{eq:Nash3.5}, Lemmas \ref{lem:Nash cor} and \ref{lem:Nash_more} yield
\begin{equation}\label{eq:covering5}
\NN_{(x_i,t_i)} ( \rho/2 )\geq \NN_{(x_i,t_i)}(t_i+R^2+\rho/2)\geq \NN_{(x_0,0)}(R^2+\rho/2)-C_{\rho,R}\geq \NN_{(x_0,0)}(\rho/2)-C_{\rho,R}.
\end{equation}
Here,
we used $\HH(\cdot, -R^2-\rho/2)\geq -n/\rho$ by Lemma \ref{lem:scal} (see Remark \ref{rem:scal2}).
Combining \eqref{eq:covering4},  \eqref{eq:covering3}, \eqref{eq:covering5} implies $| \mathcal{I} | \leq C_{\rho,R} \xi^{-n}$.
This together with \eqref{eq:covering2} proves the claim.
\end{proof}

\subsection{Effective strata}\label{sec:eff_strat}

Due to Theorem \ref{thm:main almost splitting},
we have the following (cf. \cite[Lemma 11.3]{B3}):
\begin{prop}\label{prop:pre_stratification}
For $\kappa>0,\eps,\xi \in (0,1)$,
if $\delta \leq \ov\delta_{\kappa,  \eps, \xi}$,
then the following holds:
Let $(M,(g_t)_{t\in I})$ be a super Ricci flow with $\mathcal{D}\geq 0$.
Let $(x_0,t_0)\in M\times I$.
For $r>0$,
we assume $\NN_{(x_0, t_0)} (r^2) \geq - \kappa$.
For a subset $S\subset M\times I$,
we assume that every point in $S$ is $(\delta, r)$-selfsimilar, and none of the following two properties hold:
\begin{enumerate}\setlength{\itemsep}{+1.0mm}
\item $(k+1, \eps, r)$-split;
\item $(k-1, \eps,  r)$-split and $(\eps, r)$-static.
\end{enumerate}
Then there exists $\{(x_i,t_i)\}^{N}_{i=1}\subset S \cap P (x_0, t_0; r)$ such that
\begin{equation}\label{eq:pre_stratification1}
 S \cap P (x_0, t_0; r) \subset \bigcup_{i=1}^N P (x_i, t_i; \xi r),\quad N \leq C_{\kappa}\, \xi^{-k}. 
\end{equation}
\end{prop}
\begin{proof}
We may assume $t_0 = 0$ and $r=1$.
We first notice that
by \eqref{eq:Nash3.5}, Lemmas \ref{lem:Nash cor} and \ref{lem:Nash_more},
for every $(x,t)\in P(x_0,0;1)$ we have
\begin{equation*}
\NN_{(x,t)}(1)\geq \NN_{(x,t)}(t+2)\geq \NN_{(x_0,0)}(2)-C\geq \NN_{(x_0,0)}(1)-C\geq  -C_{\kappa}.
\end{equation*}
Let $\{(x_i, t_i)\}^{N}_{i=1} \subset S \cap P(x_0, 0; 1)$ be a maximal collection of points such that $\{P (x_i, t_i; \xi/10 )\}^{N}_{i=1}$ are pairwise disjoint.
By the maximality and Lemma \ref{lem:parab_nbhd}, every point in $S \cap P(x_0, 0; 1)$ belongs to $P (x_i, t_i; 3 \xi/10)$ for some $i$,
which shows the inclusion in \eqref{eq:pre_stratification1}.

We derive an upper bound on $N$.
We may assume $t_1 \leq \ldots \leq t_N$.
We verify the following:
If $\delta\leq \ov{\delta}_{\kappa,\eps,\xi},N \geq \mathfrak{N}(100 \mathfrak{D})^{l} \xi^{-l}$,
and if $0 \leq t_i-t_1 \leq \beta (100\mathfrak{D})^{-2}\xi^2$ for all $i$,
then $(x_1, t_1)$ is $(l+1, \eps,1)$-split,
where $\beta\leq \ov\beta,\mathfrak{D} \geq \underline{\mathfrak{D}}_{\kappa}$ and $\mathfrak{N} \geq \underline{\mathfrak{N}}_{\kappa}$ are constants obtained in Theorem \ref{thm:main almost splitting}.
Indeed,
for every $i$,
Proposition \ref{prop:monotonicity_W1} yields
\begin{align*}
&\quad \,\,W_1(\nu_{(x_1, t_1); t_1-1},\nu_{(x_i,t_i);t_1-1}) \\
&\leq W_1(\nu_{(x_1,t_1);-1}, \nu_{(x_0,0); -1} )+W_1(\nu_{(x_0,0);-1}, \nu_{(x_i, t_i); -1} )\leq  2.
\end{align*}
Furthermore,
for $i \neq j$, the fact that $(x_j, t_j) \not\in P (x_i, t_i; \xi/10)$ and Proposition \ref{prop:monotonicity_W1} imply
\begin{align*}
&\quad\,\, W_1(\nu_{(x_i,t_i); t_1 - 2 (\xi/100 \mathfrak{D})^2}, \nu_{(x_j, t_j); t_1 - 2 (\xi/100 \mathfrak{D})^2} )\\
&\geq W_1(\nu_{(x_i,t_i); t_i -  (\xi/10)^2}, \nu_{(x_j, t_j); t_i -  (\xi/10)^2} )\geq \frac{\xi}{10} \geq \mathfrak{D} \left( \frac{\xi}{100 \mathfrak{D}}\right)
\end{align*}
if $\mathfrak{D}\geq \underline{\mathfrak{D}}$.
We apply Theorem \ref{thm:main almost splitting} at $(x_1, t_1)$ with $\xi$ replaced with $\xi / 100 \mathfrak{D}$ and $D=2$,
and it follows that $(x_1, t_1)$ is $(l+1, \eps,1)$-split.
We set $C_{1,\kappa}:=2\mathfrak{N}(100 \mathfrak{D})^{2n}$ and $C_{2,\kappa}:=\beta (100\mathfrak{D})^{-2}$.

Suppose $N \geq 10\,C_{1,\kappa}\,C^{-1}_{2,\kappa} \xi^{-k}$.
By the claim in the above paragraph,
there is no $\mathcal{I} \subset \{ 1, \ldots, N \}$ with $|\mathcal{I}| \geq N/2$ such that $|t_i - t_j| \leq C_{2,\kappa}\xi^2$ for all $i,j \in \mathcal{I}$.
Hence,
for all $i,j = 1, \ldots, N$ with $j - i \geq N/2$, we see $|t_i - t_j| \geq  C_{2,\kappa}\xi^2$.
By Theorem \ref{thm:almost static},
if $\delta \leq \ov\delta_{\kappa, \xi,  \eps}$, then $\{(x_i, t_i)\}^{\lfloor N/2 \rfloor}_{i=1}$ are $(\eps, 1)$-static;
in particular,
they are not $(k-1, \eps, 1)$-split.
Let $\mathcal{J} \subset \{ 1, \ldots, \lfloor N/2 \rfloor \}$ be a subset such that $|t_i  - t_j | \leq C_{2,\kappa}\xi^2$ for all $i,j \in \mathcal{J}$,
and
\begin{equation*}
|\mathcal{J} | \geq \frac{C_{2,\kappa}}{10} \xi^2 \lfloor N/2 \rfloor \geq   \frac{C_{1,\kappa}}{2}\, \xi^{-(k-2)}.
\end{equation*}
Applying the above claim to $\{(x_i, t_i)\}_{i \in \mathcal{J}}$ for $l = k-2$, we arrive at the contradiction.
\end{proof}

For $\eps \in (0,1)$ and $r_1,r_2>0$ with $r_1 < r_2$,
the \textit{effective strata $\SS^{\eps, k}_{r_1, r_2}$} are defined by the set of all points in $M\times I$ such that
for all $r \in (r_1, r_2)$ none of the following two properties hold:
\begin{enumerate}\setlength{\itemsep}{+2.0mm}
\item $(\eps, r)$-selfsimilar and $(k+1, \eps, r)$-split;
\item $(\eps, r)$-selfsimilar, $(k-1, \eps, r)$-split and $(\eps, r)$-static.
\end{enumerate}

We conclude the following quantitative stratification result,
which has been established by Bamler \cite{B3} for Ricci flow (cf. \cite[Proposition 11.2]{B3}):
\begin{thm}\label{thm:stratification}
Let $(M,(g_t)_{t\in I})$ be a super Ricci flow with $\mathcal{D}\geq 0$.
Let $(x_0, t_0) \in M \times I$.
For $r,\kappa > 0$ we assume $[t_0 - 2 r^2, t_0] \subset I$ and $\NN_{(x_0,t_0)} (r^2) \geq - \kappa$.
Then for all $\eps \in (0,1)$ and $\sigma \in (0, \eps)$,
there exists $\{(x_i, t_i)\}^{N}_{i=1} \subset \SS^{\eps,k}_{\sigma r, \eps r} \cap P (x_0, t_0; r)$ such that
\begin{equation*}
\SS^{\eps, k}_{\sigma r, \eps r} \cap P (x_0, t_0; r) \subset \bigcup_{i=1}^N P (x_i, t_i; \sigma r),\quad N \leq C_{\kappa, \eps} \,\sigma^{-k-\eps}.
\end{equation*}
\end{thm}
\begin{proof}
We may assume $t_0=0$ and $r = 1$.
For $\xi \in (0,1)$,
let $\lambda$ stand for an integer determined by $\sigma \in [\xi^{\lambda}, \xi^{\lambda-1}]$. 
Let $\delta \leq \ov\delta_{\kappa,  \eps, \xi}$ denote a constant obtained in Proposition \ref{prop:pre_stratification}.
For each $(x,t)\in P (x_0,0; 1)$,
let $\mathcal{J}_{(x,t)}$ be the set of all integers $j$ such that $(x,t)$ is not $(\delta, \xi^j)$-selfsimilar.
In virtue of Proposition \ref{prop:Nash almost},
if $\eta \leq \ov\eta_{\kappa,\delta}$,
then
\begin{equation*}
C_\kappa \geq -\NN_{(x,t)}(1)\geq \sum_{j\in \mathcal{J}_{(x,t)}}\left( \NN_{(x,t)}\left(\eta\,\xi^{2j}\right)-\NN_{(x,t)}\left(\eta^{-1}\,\xi^{2j}\right)  \right)\geq \eta \,|\mathcal{J}_{(x,t)}|.
\end{equation*}
Hence,
$|\mathcal{J}_{(x,t)}|$ is bounded by $C_2=C_{2,\kappa,\eta}$ (cf. \cite[Section 3]{CN}).
For $\{o_j\}^{\lambda-1}_{j=0} \subset \{ 0, 1 \}$, we set
\begin{equation*}
S\left(\{o_j\}^{\lambda-1}_{j=0} \right) := \left\{ (x,t) \in P (x_0, 0; 1) ~|~\text{$o_j=1$ if and only if $(x,t)$ is not $(\delta,\xi^j)$-selfsimilar} \right\}.
\end{equation*}
Then $P (x_0,0; 1)$ can be written as the union of at most $\lambda^{C_2}$ many non-empty subsets of the form $S\left(\{o_j\}^{\lambda-1}_{j=0} \right)$ since $S\left(\{o_j\}^{\lambda-1}_{j=0} \right) = \emptyset$ if $\sum_{j}o_j > C_{2}$.

Fix $l=0,\dots,\lambda-1$ and $(y, s) \in  P(x_0,0;1)$,
where let $(y, s) := (x_0,0)$ if $l=0$.
Due to Propositions \ref{prop:covering} and \ref{prop:pre_stratification},
there is $\{(y_i, s_i)\}^{N}_{i=1} \subset \SS^{\eps, k}_{\sigma, \eps} \cap S\left(\{o_j\}^{\lambda-1}_{j=0} \right)\cap P (y, s; \xi^{l})$ such that
\begin{equation*}
\SS^{\eps, k}_{\sigma, \eps} \cap S\left(\{o_j\}^{\lambda-1}_{j=0} \right)\cap P(y, s; \xi^{l})  \subset \bigcup_{i=1}^{N} P(y_i, s_i; \xi^{l+1}),\quad 
N\leq \begin{cases} C_{3,\xi} & \text{if $o_l = 1$ or $l =0$,} \\ C_{1,\kappa}\, \xi^{-k} & \text{if $o_l= 0$ and $l \geq 1$,} \end{cases}
\end{equation*}
where the assertion for $l=0$ or $o_{l} = 1$, and that for $l \geq 1$ and $o_{l} = 0$ are consequences of Propositions \ref{prop:covering} and \ref{prop:pre_stratification},
respectively.
By induction,
$\SS^{\eps, k}_{\sigma, \eps} \cap S\left(\{o_j\}^{\lambda-1}_{j=0} \right)\cap P(x_0,0; 1)$ can be covered by at most
\begin{equation*}
C_{3,\xi}^{1+\sum_j o_j} (C_{1,\kappa}\, \xi^{-k} )^{\lambda-1- \sum_j o_j}
\end{equation*}
many parabolic balls of the form $P(y,s; \sigma )$.
It follows that
$\SS^{\eps, k}_{\sigma, \eps}  \cap P(x_0,0; 1)$ can be covered by at most 
\begin{equation*}
\lambda^{C_2} \,C_{3,\xi}^{1+C_2} (C_{1,\kappa} \xi^{-k} )^{\lambda} \leq C_{\kappa,  \xi, \delta} \,\lambda^{C_2} C_{1,\kappa}^\lambda \,\xi^{-\lambda\,k}
\end{equation*}
many such parabolic balls.
If $\xi \leq \ov\xi_{\kappa, \eps}$ such that $\xi^{-\eps/2} \geq C_{1,\kappa}$,
then
\begin{equation*}
C_{\kappa,  \xi, \delta} \,\lambda^{C_2} \,\xi^{-\lambda k - m\eps/2}\leq C_{\kappa,  \xi, \delta}\, \xi^{-\lambda k - m\eps  }\leq C_{\kappa,  \xi, \delta} \,\sigma^{-k-\eps}.
\end{equation*}
Thus,
we complete the proof.
\end{proof}

\subsection*{{\rm Acknowledgements}}
The authors thank the anonymous referees for useful comments.
The first author was supported by JSPS KAKENHI (JP23K03105). 
The second author was supported by JSPS KAKENHI (JP22H04942, JP23K12967).


\end{document}